\documentclass[11pt]{article}

\usepackage[T1]{fontenc}
\usepackage[dvipsnames]{xcolor}
\usepackage{xurl}
\usepackage[colorlinks=true, linkcolor=blue, citecolor=blue, urlcolor=blue, hypertexnames=false]{hyperref}
\usepackage{natbib}
\bibpunct[, ]{(}{)}{,}{a}{}{,}%

\usepackage{amsmath, amssymb, amsfonts, mathtools, bbm, dsfont, cases}
\usepackage{array}

\DeclareFontFamily{U}{mathx}{}
\DeclareFontShape{U}{mathx}{m}{n}{<-> mathx10}{}
\DeclareSymbolFont{mathx}{U}{mathx}{m}{n}
\DeclareMathAccent{\widehat}{0}{mathx}{"70}
\DeclareMathAccent{\widecheck}{0}{mathx}{"71}

\usepackage{graphicx}
\usepackage{enumerate}
\usepackage{booktabs}
\usepackage{float}
\usepackage{multicol}
\usepackage{multirow}
\usepackage{tabularx}
\usepackage{threeparttable}
\usepackage{hhline}
\usepackage{rotating}
\usepackage{afterpage}

\usepackage[inline]{enumitem}
\usepackage{setspace}
\usepackage{fancyvrb}
\usepackage{fix-cm}
\usepackage{anyfontsize}
\usepackage{etoolbox}

\usepackage{tcolorbox}
\usepackage{tikz}
\usetikzlibrary{arrows.meta, positioning, calc, fit}
\usepackage{cleveref}

\usepackage{bibunits}
\usepackage{endnotes}
\let\footnote=\endnote

\usepackage[linesnumbered,ruled,vlined]{algorithm2e}

\def \dd{{\rm d}}
\newcommand{\R}{\mathbb{R}}
\newcommand{\E}{\mathbb{E}}

\renewcommand{\bf}[1]{\textbf{#1}}

\usepackage{subfigure}
\usepackage[protrusion=true,expansion=false]{microtype}
\graphicspath{{figures/}} 

\newtheorem{theorem}{Theorem}
\newtheorem{corollary}{Corollary}

\newtheorem{assumption}{Assumption}
\newtheorem{remark}{Remark}
\newtheorem{lemma}{Lemma}
\makeatletter

\newcommand{\Rmnum}[1]{\expandafter\@slowromancap\romannumeral #1@}
\makeatother

\newcommand{\Halmos}{\hfill\ensuremath{\square}}

\allowdisplaybreaks

\begin{document}

\title{Scalable Bi-causal Optimal Transport via KL Relaxation and Policy Gradients\footnote{The authors thank Jose Blanchent, Yinbin Han, Wenhao Yang, and other participants at the Stanford Student OR Seminar for helpful comments and feedback.}
}


\author{Haoyang Cao\thanks{Department of Applied Mathematics and Statistics, Data Science and AI Institute, and Mathematical Institute for Data Science, Johns Hopkins University. \url{hycao@jhu.edu} (email).}, \ Jesse Hoekstra\thanks{Department of Statistics, Oxford-Man Institute of Quantitative Finance, and Nuffield College, University of Oxford. \url{jesse.hoekstra@nuffield.ox.ac.uk} (email).}, \ Renyuan Xu\thanks{ Management Science \& Engineering Department, Stanford University. \url{renyuanxu@stanford.edu} (email).}, \ 
Yumin Xu\thanks{School of Mathematical Sciences, Peking University. \url{xuyumin@pku.edu.cn} (email).}, \  
and 
Ruixun Zhang\thanks{School of Mathematical Sciences, Center for Statistical Science, Laboratory for Mathematical Economics and Quantitative Finance, and National Engineering Laboratory for Big Data Analysis and Applications, Peking University. \url{zhangruixun@pku.edu.cn} (email).}}

\date{\today 
}

\maketitle
\thispagestyle{empty}

\centerline{\textbf{Abstract}} \baselineskip 14pt \vskip 10pt
\noindent
Bi-causal optimal transport (OT) is a natural framework for comparing and coupling stochastic processes under nonanticipative information constraints, with important applications in robust finance, sequential uncertainty quantification, and multistage stochastic optimization.
In particular, a learned bi-causal coupling naturally serves as a simulator for generating joint sample paths that respect both prescribed marginal laws and the underlying information flow.
Its practical use, however, is limited by the computational difficulty of enforcing bi-causal coupling constraints over path space, especially for continuous distributions and long horizons. We develop a scalable stochastic-optimization framework for computing bi-causal OT couplings under general marginals. Our approach introduces a Kullback--Leibler (KL)-penalized relaxation that replaces hard marginal constraints with tractable divergence penalties while preserving the recursive structure of the problem. We establish dynamic programming principles for both the original and relaxed formulations, prove that the relaxed problem converges to the original bi-causal OT problem as the penalty grows, and derive explicit policy-gradient representations for the relaxed objective. Building on these results, we propose a practical policy-gradient algorithm with unbiased mini-batch estimators, variance reduction, and nonasymptotic regret guarantees. Numerical experiments show that the method accurately captures marginal laws and temporal dependence, and performs well in applications including robust subhedging and time series statistical downscaling. These results provide a scalable computational approach to bi-causal OT and broaden its applicability in settings where nonanticipative information constraints are essential.


\newpage

\begingroup
\pagestyle{empty}
\singlespacing
\tableofcontents
\clearpage
\endgroup
\newpage

\setcounter{page}{1}
\pagenumbering{arabic}
\setlength{\baselineskip}{1.5\baselineskip}
\onehalfspacing
\setcounter{equation}{0}
\setcounter{table}{0}
\setcounter{figure}{0}

\section{Introduction}

Bi-causal optimal transport (OT), also known as adapted transport, extends classical OT to sequential settings by restricting admissible couplings to be \emph{non-anticipative} with respect to the natural filtration of two stochastic processes \citep{hellwig1996sequential,pflug2012distance,bartl2023sensitivity,bartl2026wasserstein}. Concretely, given two discrete-time processes, a coupling is bi-causal if, at every time step $n$, the conditional law of each process's next state depends only on the joint history up to time $n$ and coincides with its own prescribed transition---so neither process can exploit or anticipate the other's future. 
The bi-causal OT problem then seeks the coupling in this restricted set that minimizes a cumulative expected cost, yielding a distance between stochastic processes that respects their temporal information structure.

This information-aware structure makes bi-causal OT a natural framework for learning couplings between stochastic processes under non-anticipative constraints, with broad applications in operations research and quantitative finance. Moreover, once computed, a bi-causal coupling also serves as a {\it principled simulator} for generating joint sample paths that are consistent with prescribed marginals and information flows. Among these applications, two settings particularly highlight the need for scalable bi-causal transport: robust hedging and time series statistical downscaling. In robust hedging, pricing bounds are defined through worst-case expectations over models that must be consistent with prescribed marginals and adapted to the market filtration \citep{burzoni2019pointwise,backhoff2020adapted,neufeld2022modelfree,krsek2025duality}, so the coupling must respect the underlying information flow; at the same time, the learned coupling provides a simulator for generating dynamically consistent stress scenarios. In statistical downscaling, one seeks to transform low-fidelity time series into high-fidelity ones while preserving temporal dependence, where naive marginal matching can lead to dynamically inconsistent samples \citep{wan2024debias,lopez2025dynamical}; here, the bi-causal coupling similarly acts as a simulator for generating temporally consistent high-fidelity trajectories. More broadly, any sequential modeling or uncertainty quantification task in which the underlying information flow must be explicitly respected \citep{casey2005scenario,pflug2012distance} can benefit from the bi-causal OT framework \citep{hoyland2001generating,heitsch2009scenario}.

The main obstacle to the broader adoption of bi-causal OT is {\it computational scalability}. Existing methods are most effective either for discrete distributions or for empirical samples from continuous laws at small scale, where one can exploit Sinkhorn-type iterations via entropic regularization or linear-programming formulations \citep{pichler2021nested,eckstein2024computational}. However, the resulting computational and memory costs typically increase rapidly with the sample size, time horizon, and state dimension. For continuous distributions, the recently proposed fitted value-iteration methods based on dynamic programming provide an important alternative \citep{bayraktar2025fitted}. Their computational cost, however, can remain substantial, since the bi-causal constraint must be enforced at each iteration, often via an additional inner iterative procedure. In addition, the sample complexity of their method suffers from the curse of dimensionality. These challenges become even more pronounced when bi-causal OT must be repeatedly evaluated within a larger optimization or statistical pipeline \citep{backhoff2020adapted,bartl2023sensitivity,gao2024multistage}. As a result, scalable computational methods for bi-causal OT under general marginals remain limited.

\paragraph{Our work and contributions.} This paper addresses the above gap by developing an efficient end-to-end stochastic-optimization framework for computing bi-causal OT couplings with general marginals, including continuous ones. The central difficulty is that bi-causal OT imposes hard sequential marginal constraints on a path-space coupling, so the coupling must respect the underlying information structure, even though direct optimization over the resulting global feasible set is intractable. Our key idea is to replace the hard marginal-matching constraints, which render the problem combinatorially difficult, with a Kullback--Leibler (KL)-penalized formulation that remains differentiable, computationally tractable, and amenable to policy-gradient methods, while preserving the recursive structure needed for computation.

The resulting framework learns joint laws that capture both marginal dynamics and temporal information flow. Such couplings can be used to generate dynamically consistent trajectories, and in some cases admit decision-theoretic interpretations, such as in robust hedging. Our primary focus, however, is on the learning and computation of the coupling itself. For clarity, we focus on the Markov setting, although our analysis extends to more general adapted marginals. Our main contributions are as follows.

\emph{Relaxation theory, well-posedness, and asymptotic consistency.}
We establish the dynamic programming principle for both the original bi-causal OT problem and the KL-relaxed problem (Lemmas \ref{lem:dpp} and \ref{lem:dpp2}), providing the recursive characterization of the coupling and value functions underlying our algorithmic design. We further show that, as the penalty parameter tends to infinity, both the relaxed objective and the corresponding optimal couplings converge to their counterparts in the original bi-causal OT problem (Theorem \ref{thm:gamma-convergence}). This result establishes that the relaxed formulation is not merely computationally convenient, but a principled surrogate for the constrained problem.

\emph{Scalable policy-gradient algorithm with convergence guarantees.}
Building on a Markov parameterization of bi-causal couplings, we derive explicit policy-gradient representations (Theorem \ref{thm:policy-grad}), construct unbiased gradient estimators with asymptotic variance reduction (Lemma \ref{lem:unbiased-cov}), and obtain a practical algorithm that can be implemented with expressive conditional models such as normalizing flows (Algorithm \ref{alg:kl-pg}). A key advantage of our policy-gradient formulation is that its sampling-based implementation is {\it dimension-free}, in the sense that it does not require discretization over the state or action spaces.  In addition, we establish a nonasymptotic convergence guarantee showing that the average regret is of order $\mathcal{O}(1/K)$, where $K$ denotes the number of episodes (Theorem~\ref{thm:beta-PG-converge}). Taken together, these results turn the recursive structure of bi-causal OT into an efficient computational procedure that scales beyond discrete settings.

\emph{Numerical validation and applications.}
We evaluate the performance of our method across several experiments. In controlled synthetic settings, the proposed algorithm performs well in capturing both marginal fidelity in multimodal distributions and temporal dependence in high-dimensional and long-horizon problems (Section \ref{sec: synthetic experiments}). We then demonstrate its practical value in two representative applications: robust subhedging under adaptedness constraints, where the method recovers near-optimal pricing bounds with relative errors below $1\%$ (Section \ref{sec: robust hedging}), and time series statistical downscaling, where OT-debiased conditioning substantially outperforms unconditional generation (Section \ref{sec: statistical downscaling}). In addition, our implementation is computationally faster than existing approaches in the literature, further supporting the practical scalability of the proposed framework.
Our code is publicly available on GitHub and all experimental results can be replicated.\footnote{\href{https://github.com/jesse-hoekstra/statistical_downscaling}{https://github.com/jesse-hoekstra/statistical\_downscaling}.}

\paragraph{Related literature.}
The literature most closely related to our work concerns bi-causal OT and its direct extensions.
The discrete-time theory of causal and bi-causal transport was developed by \citet{backhoff2017causal}, 
who established duality and dynamic programming formulations that remain central to the recursive structure of the problem. 
Subsequent work has explored both theoretical and computational aspects of the framework. Specifically,
\citet{lassalle2018ot} formulated causal transport on filtered spaces as a non-anticipative Monge--Kantorovich problem, 
while \citet{backhoff2020fundamental} studied the nested distance (also known as the bi-causal Wasserstein distance) 
and established its basic metric and topological properties. 
Dynamic-programming formulations have also been investigated for Markov chains 
\citep{moulos2021bicausal} and via fitted value-iteration \citep{bayraktar2025fitted}. 
Recent work further develops computational and structural aspects, including efficient algorithms 
\citep{calo2024bisimulation,eckstein2024computational}, general duality results \citep{krsek2025duality}, 
multi-process transport \citep{acciaio2025multicausal}, and specialized settings such as stochastic differential equation (SDE) laws 
\citep{hitz2024bicausal} or Gaussian processes \citep{gunasingam2025adapted}. 
Additional directions include equilibrium formulations for time-inconsistent costs \citep{bayraktar2025equilibrium} 
and statistical estimation of bi-causal transport distances \citep{calo2025distances}. However, most existing computational approaches either focus on discrete settings or rely on recursive dynamic-programming formulations tailored to specific model classes. 
Our work instead develops a stochastic-optimization framework that enables scalable computation of bi-causal OT couplings under general marginals.

Another strand of literature studies adapted transport and the geometry of stochastic processes under 
information constraints. 
Important precursors include scenario-tree modeling \citep{hoyland2001generating,casey2005scenario} and the nested distance \citep{pflug2012distance} for multistage stochastic optimization. 
Building on this perspective, \citet{backhoff2020adapted} introduced adapted Wasserstein distances and established 
stability results for dynamic optimization problems. 
Further work clarified the topological structure of stochastic-process spaces under adapted metrics 
\citep{backhoff2020all} and developed statistical estimation and empirical approximation methods 
\citep{backhoff2022estimating}. 
More recently, \citet{gao2024wasserstein} and \citet{bartl2026wasserstein} proposed a geometric framework for Wasserstein spaces 
of stochastic processes. Whereas this line of work focuses primarily on geometry, stability, and statistical properties, our paper addresses the algorithmic problem of computing bi-causal OT couplings efficiently.

Finally, our work is connected to several broader strands of literature. 
Classical OT provides the foundational framework, and entropic regularization methods 
\citep{cuturi2013sinkhorn} have become a standard tool for scalable OT computation. 
Connections between OT and robust finance were explored by \citet{beiglbock2013model}, 
while dynamic arbitrage and duality under model uncertainty were studied by \citet{bouchard2015arbitrage}. 
Optimal-transport-based objectives have also played an important role in modern generative modeling, 
including large-scale learning with Sinkhorn divergences \citep{genevay2018learning}, flow-matching approaches 
\citep{lipman2023flow}, diffusion Schr\"odinger bridge models \citep{shi2023dsbm,ma2025schr,alouadi2026lightsbb}, and stochastic-interpolant 
frameworks \citep{albergo2025stochastic}. 
These works develop powerful computational tools for classical transport problems but do not incorporate the non-anticipative information constraints that characterize bi-causal OT.

\paragraph{Paper organization.} 
Section~\ref{sec: bicausal ot} lays the mathematical foundation for the bi-causal OT problem between Markov sequences and establishes its dynamic programming principle. Section~\ref{sec: relaxed formulation} introduces the KL-penalized relaxation and proves \(\Gamma\)-convergence to the original problem. Section~\ref{sec: algorithm and convergence} develops the policy gradient algorithm with convergence guarantees. Section~\ref{sec: synthetic experiments} presents synthetic experiments and Section~\ref{sec:applications} develops two applications in robust hedging and time series statistical downscaling, respectively. All formal proofs and technical details are deferred to the appendices.

\section{Bi-causal optimal transport of Markov sequences}
\label{sec: bicausal ot}
In this section, we introduce the mathematical setting for the bi-causal OT problem between Markov sequences. We first revisit the definitions of both couplings and bi-causal couplings, and introduce equivalent Markov bi-causal couplings. Then, we establish the dynamic programming principle for the bi-causal OT problem.

\subsection{Set-up and notation}
\paragraph{Sequential data and their distributions.}
We consider $d$-dimensional time series data with length $N+1$. Formally, let $S=\mathbb R^d$ be a state space with Borel $\sigma$-algebra $\mathcal S=\mathcal B(S)$. Let $\mathcal Y=S^{N+1}$ be the set of paths on \(S\) with length \(N+1\), and its Borel $\sigma$-algebra is $\mathcal B(\mathcal Y)=\mathcal S^{\otimes(N+1)}$. To model two coupled time series $\pmb Y$ and $\pmb Y'$, we have the product space \(\mathcal Y\times\mathcal Y\) and its Borel \(\sigma\)-algebra \(\mathcal B(\mathcal Y\times\mathcal Y)=\mathcal B(\mathcal Y)\otimes\mathcal B(\mathcal Y)\). Throughout the paper, we write $[N]:=\{1,\dots,N\}$ and $[N]_0 := \{0,1,\dots,N\}$. For any $\pmb y=(y_0,\dots,y_N)\in\mathcal Y$ and $n \in [N]_0$, we write the prefix $\pmb y_{0:n}=(y_0,\dots,y_n)$ and the suffix
$\pmb y_{n:N}=(y_n,\dots,y_N)$. We define $\mu\otimes\mu'$ as the product coupling of two independent marginal laws $\mu$ and $\mu'$.

Consider a probability space \((\Omega, \mathcal F,\mathbb P)\). Let \(\mathcal P(S)\), \(\mathcal P(\mathcal Y)\) and \(\mathcal P(\mathcal Y\times\mathcal Y)\) be the sets of probability distributions over the state space \(S\), the path space \(\mathcal Y\) and the product space \(\mathcal Y\times\mathcal Y\), respectively. For any path law $\mu\in\mathcal P(\mathcal Y)$ and \(\pmb Y\sim\mu\), given any $n \in [N]$ and $\pmb y_{0:n} \in S^{\otimes(n+1)}$, its one-step conditional kernel at step \(n\), \(\mu(\cdot \,|\, \pmb y_{0:n-1})\), is given by
\begin{equation}
\label{eq:marginal-transition-history}
\mu(A \,|\, \pmb y_{0:n-1})
=\mathbb P \big(Y_n\in A \,\big|\, \pmb Y_{0:n-1}=\pmb y_{0:n-1}\big),
\qquad \forall\,A\in\mathcal S.
\end{equation}
and its marginal law is $\mu_n(B)=\mathbb{P}(Y_n \in B)$ for any $B \in \mathcal S$. Moreover, if $\mu$ is Markov with initial law $\mu_0$ and transitions $\pmb T=(T_1,\dots,T_N)$, where \(T_n(\cdot \,|\, y)\in\mathcal P(S)\) for any $n \in [N]$ and \(y\in S\), then 
\begin{equation*}
\mu(\dd \pmb y)=\mu_0(\dd y_0)\prod_{n=1}^N T_n(\dd y_n \,|\, y_{n-1}),
\end{equation*}
and \eqref{eq:marginal-transition-history} reduces to $\mu(A \,|\, \pmb y_{0:n-1})=T_n(A \,|\, y_{n-1})$.

\paragraph{Couplings and bi-causality.}
To model two coupled time series---such as two price processes, or a real sample path and a synthetic sample path---we introduce the notion of bi-causal couplings. They describe two stochastic processes that are non-anticipative in both directions, so neither process is allowed to use the other’s future when coupled.

In particular, we consider two distributions \(\mu,\mu'\in\mathcal P(\mathcal Y)\), and a probability distribution $\pi\in\mathcal P(\mathcal Y\times\mathcal Y)$ is said to be a coupling of $\mu$ and $\mu'$ if \(\pi\) satisfies
\begin{align}
\label{eq:coupling}
\pi(A\times\mathcal Y)=\mu(A),\quad
\pi(\mathcal Y\times B)=\mu'(B),
\qquad \forall\,A,B\in\mathcal B(\mathcal Y).
\end{align}
Denote the set of all possible couplings of \(\mu\) and \(\mu'\) by \(\Pi(\mu,\mu')\). For $(\pmb Y,\pmb Y')\sim\pi\in\Pi(\mu,\mu')$ and $n\in[N]$, we define one-step \emph{marginal} conditionals as
\begin{align*}
\pi_n^{(1)}(A \,|\, \pmb y_{0:n-1},\pmb y'_{0:n-1})
&:=\mathbb P_\pi \big(Y_n\in A \,\big|\, \pmb Y_{0:n-1}=\pmb y_{0:n-1}, \pmb Y'_{0:n-1}=\pmb y'_{0:n-1}\big), \\
\pi_n^{(2)}(B \,|\, \pmb y_{0:n-1},\pmb y'_{0:n-1})
&:=\mathbb P_\pi \big(Y_n'\in B \,\big|\, \pmb Y_{0:n-1}=\pmb y_{0:n-1}, \pmb Y'_{0:n-1}=\pmb y'_{0:n-1}\big), 
\end{align*}
for all $A,B\in\mathcal S$. Throughout the paper, when there is no confusion, we denote $\E_{\pi}[\cdot] := \E_{(\pmb Y, \pmb Y') \sim\pi}[\cdot]$.

Furthermore, define the set of \emph{bi-causal couplings} as
\begin{align}
\label{eq:bicausal-coupling}
\Pi_{\textrm{bc}}(\mu,\mu')
:=\Big\{ \pi\in\Pi(\mu,\mu'):\ 
&\pi_n^{(1)}(A \,|\, \pmb y_{0:n-1},\pmb y'_{0:n-1})=\mu(A \,|\, \pmb y_{0:n-1}) \textrm{ and } \\ 
&\pi_n^{(2)}(B \,|\, \pmb y_{0:n-1},\pmb y'_{0:n-1})=\mu'(B \,|\, \pmb y'_{0:n-1}), \forall A, B \in \mathcal S, n \in [N] \Big\}. \nonumber
\end{align}
Bi-causal couplings are central to our framework: they preserve the prescribed marginal dynamics of both processes while allowing dependence only through the joint past, and therefore exclude any use of the other process's future. This non-anticipative structure makes the problem recursive, leading to the Markov kernel representation in Lemma~\ref{lem:mbc-equivalence}, the dynamic programming principle in Lemma~\ref{lem:dpp}, and the subsequent relaxed formulation and algorithmic developments in Sections~\ref{sec: relaxed formulation}--\ref{sec: algorithm and convergence}.

In addition, we assume that $\mu,\mu'$ are Markov with transitions $\pmb T,\pmb T'$, respectively. In this case, \(\pi\in\Pi_{\textrm{bc}}(\mu,\mu')\) in \eqref{eq:bicausal-coupling} is equivalent to
\begin{align}
\label{eq:markov-bicausal-coupling}
\pi_n^{(1)}(A \,|\,\pmb{y}_{0:n-1},\pmb{y}_{0:n-1}')=T_n(A \,|\, y_{n-1}), \quad  \pi_n^{(2)}(B \,|\, \pmb{y}_{0:n-1},\pmb{y}_{0:n-1}')=T_n'(B \,|\, y_{n-1}'),
\end{align}
for any $ A,B \in \mathcal S$, $n\in[N]$, and $\pmb{y}_{0:n-1},\pmb{y}_{0:n-1}'\in S^{\otimes n}$. 

\begin{remark}[Beyond the Markov setting]
The Markov setting in \eqref{eq:markov-bicausal-coupling} is imposed purely for technical and notational convenience and is not intrinsic to the framework. More generally, one may allow the coupling kernels at time $n$ to depend on the full past histories $(\pmb Y_{0:n-1}, \pmb Y'_{0:n-1})$, rather than only on the current states $(Y_{n-1},Y'_{n-1})$. The extension is conceptually straightforward: one simply replaces current-state-dependent kernels with history-dependent kernels, at the cost of additional notation but without any substantial new technical difficulty.
\end{remark}

\paragraph{Bi-causal optimal transport problem.}
Given admissible bi-causal couplings \eqref{eq:bicausal-coupling} and a sequence of bounded cost functions \(c_n:S\times S\to[0,c]\subset\mathbb R\), \(n \in [N]_0\), our goal is to find a coupling \(\pi\in\Pi_{\mathrm{bc}}(\mu,\mu')\) that minimizes the expected cumulative discrepancy between the two time series while respecting the non-anticipative information structure:
\begin{equation}
    \label{eq:bicausal-ot}
    W_{\textrm{bc}}(\mu,\mu')
    =
    \inf_{\pi\in\Pi_{\textrm{bc}}(\mu,\mu')}
    \mathbb E_{\pi}\left[\sum_{n=0}^N c_n(Y_n,Y_n')\right].
\end{equation}
Here, the cost function $c_n(y_n,y_n')$ measures the mismatch between two time series $\pmb y$ and $\pmb y'$ at step~$n$.

\paragraph{One-step Markov joint kernels.}
For Markov data, the above bi-causal couplings~\eqref{eq:bicausal-coupling} can be equivalently characterized by their one-step transition kernels, which reduces the global pathwise constraint to a local one-step description and thereby makes the analysis tractable.

In particular, we fix any pair of Markov path laws $(\mu, \mu')$ and \((\pmb Y,\pmb Y')\sim \pi\) with $\pi\in\Pi_{\textrm{bc}}(\mu,\mu')$, and define the one-step Markov joint kernel at each time $n\in[N]$ as
\begin{align}
\label{eq:coupling-kernel}
Q_n^\pi(A\times B  \,|\,  y,y')
:= \mathbb P_{\pi}\big(Y_n\in A, Y'_n\in B \,\big|\, Y_{n-1}=y, Y'_{n-1}=y'\big),
\ \forall A,B\in\mathcal S, \forall y,y'\in S.
\end{align}
In addition, denote the first and second marginal distributions of $Q_n^\pi(A\times B  \,|\,  y,y')$ in \eqref{eq:coupling-kernel} by $Q_n^{\pi,(1)}(A \,|\, y,y')$ and $Q_n^{\pi,(2)}(B \,|\, y,y')$, respectively.

In the Markov setting, define the collection of one-step bi-causal coupling kernels between $T_n$ and $T_n'$ as
\begin{align}
\label{eq:transition-coupling}
\Pi_M(T_n,T'_n)
:=\Big\{Q:S\times S\to\mathcal P(S\times S):\
&Q^{(1)}(A  \,|\,  y,y')=Q(A\times\mathcal Y|y,y')=T_n(A  \,|\,  y),\\
&\hspace{-80pt}Q^{(2)}(B  \,|\,  y,y')=Q(\mathcal Y\times B \,|\, y,y')=T'_n(B  \,|\,  y'),\ \forall A, B \in \mathcal S, \forall y,y' \in S \Big\}. \nonumber
\end{align}
Given any joint distribution \(\pi_0\in\mathcal P(S\times S)\) and kernels $\pmb Q=(Q_1,\dots,Q_N)$ with \(Q_n:S\times S\to P(S\times S)\) for \(n \in [N]\), we can define the corresponding Markov coupling $\pi_{\pi_0,\pmb Q}$ as
\begin{align}
\label{eq:markov-coupling}
\pi_{\pi_0,\pmb Q}(\dd \pmb y,\dd \pmb y')
:=\pi_0(\dd y_0,\dd y_0')\prod_{n=1}^N Q_n(\dd y_n,\dd y_n' \,|\,  y_{n-1},y_{n-1}').
\end{align}
Using \eqref{eq:transition-coupling} and \eqref{eq:markov-coupling}, we denote the collection of Markov bi-causal couplings as
\begin{align}
\label{eq:set-Markov-coupling}
    \mathcal M_{\textrm{bc}}(\mu,\mu')
    :=\Big\{\pi_{\pi_0,\pmb Q}:\ \pi_0\in\Pi(\mu_0,\mu_0'),\ Q_n\in\Pi_M(T_n,T_n'),\ \forall n \in [N]\Big\}.
\end{align}
Then we establish the following equivalent characterizations of Markov bi-causal couplings.
\begin{lemma}[Equivalent characterization of Markov bi-causal couplings]
\label{lem:mbc-equivalence}
Given that both $\mu$, $\mu'\in\mathcal P(\mathcal Y)$ are Markov, the set of Markov bi-causal couplings defined in \eqref{eq:set-Markov-coupling} coincides with the set of all bi-causal couplings, that is,
\[
\mathcal M_{\mathrm{bc}}(\mu,\mu')=\Pi_{\mathrm{bc}}(\mu,\mu').
\]
\end{lemma}

\subsection{Dynamic programming principle}
The dynamic programming principle provides a recursive characterization of the coupling and enables us to solve the bi-causal OT problem \eqref{eq:bicausal-ot} in the same spirit as sequential decision-making. It serves as the foundation for our relaxed formulation in Section~\ref{sec: relaxed formulation} and the policy-gradient method in Section~\ref{sec: algorithm and convergence}.

For any \(\pi\in\Pi_{\textrm{bc}}(\mu,\mu')\) and any \(y,y'\in S\), define the terminal condition as
\begin{equation}
    \label{eq:terminal-condition}
    J_N^\pi(y,y')=V_N(y,y')=c_N(y,y').
\end{equation}
For all $n \in [N-1]_0$, we denote the objective function at time $n$ by
\begin{equation}
    J_{n}^\pi(y,y') = \mathbb E_{\pi}\left[\sum_{k=n}^N c_k(Y_k,Y_k')\,\big|\,Y_{n}=y,Y_{n}'=y'\right]=c_n(y,y')+Q^\pi_{n+1}[J_{n+1}^\pi](y,y'),\label{eq:cost-n}
\end{equation}
which is the expected remaining cumulative cost from time $n$ onward. Here, for any kernel $Q: S\times S \to \mathcal P(S\times S)$, any bounded function $v: S \times S \to \mathbb{R}$ and any $y,y'\in S$, we define the operator 
\begin{equation}\label{eq:kernel-operator}
Q[v](y,y')=\int v(z,z')\,Q(\dd z,\dd z' \,|\,  y,y'),
\end{equation}
which measures the future expectation of $v$ from time $n$ onward given the current state $(y,y')$.

Next, we minimize the objective function $J_{n}^\pi$ in \eqref{eq:cost-n} over the admissible bi-causal couplings $\Pi_{\textrm{bc}}(\mu, \mu')$ to obtain the corresponding value function $V_n$, which is defined as
\begin{equation}
    V_n(y,y') = \inf_{\pi\in\Pi_{\textrm{bc}}(\mu,\mu')}J_n^\pi(y,y').\label{eq:opt-cost-n}
\end{equation}
For any \(n \in [N-1]_0\), given \((\pmb Y,\pmb Y')\sim \pi\), denote the conditional law from time $n+1$ to $N$ by
\begin{equation}
    \label{eq:joint-conditional-current}
    \pi_{(n+1):N}(y,y')={\rm Law}\Big(\pmb Y_{(n+1):N},\pmb Y_{(n+1):N}'\,\Big|\,Y_{n}=y,Y_{n}'=y'\Big),\quad \forall y,y'\in S.
\end{equation}
Then, \(J^\pi_n\) in \eqref{eq:cost-n} is equivalent to the following expression,
\begin{equation*} 
    J_n^\pi(y,y')=c_n(y,y')+\mathbb E_{\pi_{(n+1):N}(y,y')}\left[\sum_{k=n+1}^Nc_k(Y_k,Y_k')\right].
\end{equation*}
Here, for notational simplicity, we abbreviate $\mathbb E_{(\pmb Y_{(n+1):N},\pmb Y_{(n+1):N}')\sim\pi_{(n+1):N}(y,y')}[\cdot]$ as $\mathbb E_{\pi_{(n+1):N}(y,y')}[\cdot]$. For any bounded function $v:S\times S\to\R$, define the Bellman operator and Bellman optimality operator as follows,
\begin{align}
    \Gamma^Q_n[v]&=c_n+Q[v],\quad\forall Q\in\Pi_M(T_{n+1},T_{n+1}'); \nonumber \\
    \Gamma_n[v]&=c_n+\inf_{Q\in\Pi_M(T_{n+1},T'_{n+1})}Q[v]. \label{eq:bellman-op-n-2}
\end{align}
We have the following dynamic programming principle for \eqref{eq:bicausal-ot}, which characterizes the optimal solution to the bi-causal OT problem.

\begin{lemma}[Dynamic programming principle]\label{lem:dpp}
The following statements hold.
\begin{enumerate}
    \item For any \(n\in[N-1]_0\), \(V_{n}=\Gamma_{n}[V_{n+1}]\).
    \item \(W_{\mathrm{bc}}(\mu,\mu')=\underset{\pi_0\in\Pi(\mu_0,\mu_0')}{\inf}\E_{\pi_0}[V_0(Y_0,Y_0')]=\underset{\pi\in\mathcal{M}_{\mathrm{bc}}(\mu,\mu')}{\inf}\mathbb E_{\pi}\left[\sum_{n=0}^Nc_n(Y_n,Y_n')\right]\).
\end{enumerate}
\end{lemma}
The above lemma shows that the bi-causal OT problem \eqref{eq:bicausal-ot} can be viewed as a {\em Markov decision process} with coupling constraints. This result implies that, to solve the bi-causal OT problem \eqref{eq:bicausal-ot}, we only need to optimize recursively over the initial coupling and the one-step Markov coupling kernels instead of the entire path-space set. This reduction in complexity is what makes the relaxed formulation in Section~\ref{sec: relaxed formulation} and the policy gradient method in Section~\ref{sec: algorithm and convergence} tractable.

\section{Transport-relaxed problem}
\label{sec: relaxed formulation}
To tackle the computational challenges induced by the coupling constraints in Lemma \ref{lem:dpp}, we introduce a relaxed formulation of the bi-causal OT problem \eqref{eq:bicausal-ot} in this section. Specifically, we employ a KL penalization to regulate violations of the coupling constraints while preserving the dynamic programming principle.

\subsection{Relaxed formulation}
The original bi-causal OT problem \eqref{eq:bicausal-ot} requires every one-step coupling kernel to match the prescribed marginal transitions exactly, which poses the main computational difficulty. We therefore enlarge the feasible set by allowing arbitrary joint transition kernels and then penalize deviations from the prescribed marginals and transitions. Let us first consider the following set of general transition kernels $Q$ without marginal constraints,
\begin{equation*}
    \mathcal{U}_{\textrm{rel}}:=\Big\{Q: S\times S\rightarrow \mathcal{P}(S\times S)\Big\}.
\end{equation*}
Define the relaxed policy space over the entire horizon as 
\begin{equation}
\begin{split}
\label{eq:set_relaxed_bicausal_couplings}
\Pi_\textrm{rel}=\biggl\{\pi\in\mathcal{P}(\mathcal Y\times\mathcal{Y})\,\,\biggl|\,\,
&\exists\pi_0\in\mathcal P(S\times S),\,\exists Q_1,\dots,Q_{N}\in\mathcal{U}_\textrm{rel}\text{ s.t. }\forall \pmb B,\pmb B'\in\mathcal S^{\bigotimes(N+1)},\\
&\pi(\pmb B\times\pmb B')=\int_{\pmb B\times\pmb B'}\pi_0(\dd y_0,\dd y_0')\prod_{n=1}^NQ_n(\dd y_n,\dd y_n' \,|\, y_{n-1},y_{n-1}')\biggl\},
\end{split}
\end{equation}
which keeps the same Markov factorization structure as in the original formulation $\mathcal{M}_{\rm bc}$ in \eqref{eq:bicausal-coupling}, but removes the hard constraint that the one-step marginals of the joint kernels must coincide with the prescribed transitions $(T_n,T_n')$.
{
For any \(\pi\in\Pi_{\rm rel}\), define the marginals of $\pi_0$ as
\begin{equation*}
    \pi^{(1)}_0(A) = \pi_0(A\times \mathcal{S}),\,\,  \pi^{(2)}_0(B) = \pi_0(\mathcal{S}\times B),\quad \textrm{ for all } A,B\in\mathcal S. 
\end{equation*}

\paragraph{KL penalization.}
Under the set of relaxed couplings $\Pi_{\textrm{rel}}$ in \eqref{eq:set_relaxed_bicausal_couplings}, we use KL-divergence to penalize violations of the original bi-causal constraints. For any $\beta>0$, we consider the relaxed coupling problem
\begin{equation}
    \label{eq:relaxed-bicausal-ot}
    \begin{aligned}
    W_{\textrm{rel}}(\mu,\mu')&=\inf_{\pi\in\Pi_{\textrm{rel}}}\mathbb E_{\pi}\biggl[\sum_{n=0}^N c_n(Y_n,Y_n') + \beta \Big(\mathrm{KL}(\mu_0, \pi_0^{(1)}) + \mathrm{KL}(\mu_0', \pi_0^{(2)})\Big)\\
    &\quad +\beta\sum_{n=1}^N \E_{\mu \otimes \mu'}\Big[\mathrm{KL}\big(T_n(\cdot \,|\, Y_{n-1}), Q_n^{\pi,(1)}(\cdot \,|\, Y_{n-1},Y_{n-1}')\big)\Big] \\
    &\quad \quad + \beta\sum_{n=1}^N \E_{\mu \otimes \mu'}\Big[\mathrm{KL}\big(T_n'(\cdot \,|\, Y_{n-1}'), Q_n^{\pi,(2)}(\cdot \,|\, Y_{n-1},Y_{n-1}')\big)\Big]\biggl],
    \end{aligned}
\end{equation}
where the first two KL terms penalize deviations of the initial marginals, and the remaining terms penalize deviations of the one-step transition marginals over the entire horizon.
\footnote{For probability measures $\rho,\rho'\in\mathcal P(S)$, the KL divergence is defined as
${\rm KL}(\rho,\rho')=\mathbb E_{X\sim \rho}\left[\log\left(\frac{\dd \rho(X)}{\dd \rho'(X)}\right)\right]$ whenever $\rho\ll\rho'$, and is set to $+\infty$ otherwise.}

The KL divergence terms jointly penalize the deviation of $\pi\in\Pi_{\textrm{rel}}$ from the bi-causal constraint set $\Pi_{\textrm{bc}}(\mu,\mu')$. 
Here the prescribed marginals and transition kernels $(\mu_0,\mu_0',T_n,T_n')$
serve as the \emph{reference} objects, while the corresponding marginals and kernels induced by $\pi$, namely $(\pi_0^{(1)},\pi_0^{(2)},Q_n^{\pi,(1)},Q_n^{\pi,(2)}),$ act as the \emph{candidate} objects.

Finally, for a fixed reference measure, the KL divergence is convex in the candidate density. 
Moreover, the chosen orientation enforces absolute continuity of the reference with respect to the candidate, assigning an infinite penalty under support mismatch. 
As a result, the relaxation is mass-covering and provides a stable surrogate for enforcing the bi-causal constraints 
\citep{minka2005divergence,blei2017variational,zhang2022transport}.
\footnote{We provide a remark regarding the orientation of the KL penalty.
In many machine learning formulations involving KL penalties, one often considers problems of the form
\[
\min_{\mu} \; F(\mu) + \mathrm{KL}(\mu,\nu),
\]
where $\nu$ is a fixed reference measure. In contrast, in \eqref{eq:relaxed-bicausal-ot}, the reference objects appear in the first argument of the KL divergence. This choice is motivated by computational considerations. With this orientation, the KL terms are evaluated under the prescribed marginals and transition kernels $(\mu_0,\mu_0',T_n,T_n')$, which are fixed and directly sampleable from data. This structure enables efficient stochastic gradient estimation in our policy-gradient algorithm (see Algorithm \ref{alg:kl-pg}).  Finally, we note that this relaxation also enjoys a $\Gamma$-convergence property: as $\beta\to\infty$, the relaxed objective converges to the original bi-causal OT problem (see Theorem \ref{thm:gamma-convergence}).}

For notational simplicity of the recursive analysis in Lemma~\ref{lem:dpp2}, we isolate KL penalties at each time step. For $n=0$, we denote
\begin{equation}
\label{eq:expectation_kl_0}
    \mathrm{KL}_{0}^{\!(1)}\big(\mu, \pi\big) := \mathrm{KL}\big(\mu_0, \pi_0^{(1)}\big),\qquad
    \mathrm{KL}_{0}^{\!(2)}\big(\mu', \pi\big) := \mathrm{KL}\big(\mu_0', \pi_0^{(2)}\big).
\end{equation}
For any $n \in [N]$, define
\begin{align}
    \mathrm{KL}_{n}^{\!(1)}\big(\mu, \pi\big) &:= \E_{\mu \otimes \mu'}\Big[\mathrm{KL}\big(T_n(\cdot \,|\, Y_{n-1}), Q_n^{\pi,(1)}(\cdot \,|\, Y_{n-1},Y_{n-1}')\big)\Big], \label{eq:expectation_kl_n_1} \\
    \mathrm{KL}_{n}^{\!(2)}\big(\mu', \pi\big) &:= \E_{\mu \otimes \mu'}\Big[\mathrm{KL}\big(T_n'(\cdot \,|\, Y_{n-1}'), Q_n^{\pi,(2)}(\cdot \,|\, Y_{n-1},Y_{n-1}')\big)\Big] \label{eq:expectation_kl_n_2}.
\end{align}

\paragraph{Value functions.} The difference between the original problem \eqref{eq:bicausal-ot} and the relaxed bi-causal OT problem \eqref{eq:relaxed-bicausal-ot} in terms of the objective function and the corresponding value function is that the continuation cost now includes both the transport cost and the future KL penalties. Mathematically, given any \(\pi\in\Pi_{\textrm{rel}}\) and any \(y,y'\in S\), we first define the terminal condition as
\begin{equation}
\label{eq:relaxed-terminal-cost}
    \widetilde J_N^{\pi}(y,y')=\widetilde V_N(y,y')=c_N(y,y').
\end{equation}
Recalling the definition of \(\pi_{(n+1):N}(y,y')\) in \eqref{eq:joint-conditional-current} and for any \(n\in\{0,\dots,N-1\}\), we denote the relaxed objective function at time $n$ as
\begin{align}
    \widetilde J_n^\pi(y,y')&=\mathbb E_{\pi_{(n+1):N}(y,y')}\bigg[\sum_{k=n}^N c_k(Y_k,Y_k') +\beta \sum_{k=n+1}^{N} \Big[\mathrm{KL}_{k}^{\!(1)}\big(\mu, \pi\big)+\mathrm{KL}_{k}^{\!(2)}\big(\mu', \pi\big)\Big] \,\biggl|\, Y_n=y, Y_n'=y' \bigg] \label{eq:relaxed-cost-n-OG} \\
    &=c_n(y,y')+
    \beta \Big[\mathrm{KL}_{n+1}^{\!(1)}\big(\mu, \pi\big)+\mathrm{KL}_{n+1}^{\!(2)}\big(\mu', \pi\big)\Big] + Q^\pi_{n+1}[\tilde{J}^\pi_{n+1}](y,y') \label{eq:relaxed-cost-n},
   \end{align}
and the corresponding value function as
\begin{equation}
    \widetilde V_n(y,y')=\inf_{\pi\in\Pi_{\textrm{rel}}}\widetilde J_n^\pi(y,y').\label{eq:relaxed-opt-cost-n}
\end{equation}
Under the relaxed bi-causal OT problem \eqref{eq:relaxed-bicausal-ot}, for any bounded function \(v:S\times S\to\mathbb R\) and any $n\in[N-1]_0$, we further define the Bellman operator $\widetilde\Gamma^Q_n$ and the corresponding Bellman optimality operator $\widetilde\Gamma_n$ as follows,
\begin{align}
    \widetilde\Gamma^Q_n[v]&=c_n+\beta \Big[\mathrm{KL}_{n+1}^{\!(1)}\big(\mu, \pi\big)+\mathrm{KL}_{n+1}^{\!(2)}\big(\mu', \pi\big)\Big] + Q\left[v \right],\quad \forall Q \in \mathcal U_{\textrm{rel}}; \nonumber \\
    \widetilde\Gamma_n[v]&=\inf_{Q\in \mathcal U_{\textrm{rel}}} \widetilde\Gamma^Q_n[v].\label{eq:bellman-op-n-2-2}
\end{align}

To show the above relaxation still preserves the dynamic programming principle, we establish the following result, a relaxed counterpart of Lemma~\ref{lem:dpp}.
\begin{lemma}[Dynamic programming principle of the relaxed problem]\label{lem:dpp2}
The following statements are true:
\begin{enumerate}
    \item For any \(n\in[N-1]_0\), \(\widetilde V_{n}=\widetilde\Gamma_{n}[\widetilde V_{n+1}]\).
    \item In addition, \(W_{{\mathrm{rel}}}(\mu,\mu')\) defined in \eqref{eq:relaxed-bicausal-ot} satisfies
    \[W_{\mathrm{rel}}(\mu,\mu')= \inf_{\pi_0\in \mathcal P(S \times S)}\E_{\pi_0}\left[\widetilde V_0(Y_0,Y_0')+\beta \Big[\mathrm{KL}_{0}^{\!(1)}\big(\mu, \pi\big)+\mathrm{KL}_{0}^{\!(2)}\big(\mu', \pi\big)\Big]\right].\]
\end{enumerate}
\end{lemma}
Lemma~\ref{lem:dpp2} shows that the relaxed problem also admits a recursive dynamic programming principle. As a result, the optimization reduces from a path-space problem to a sequence of one-step problems over the initial coupling and the transition kernels, which is the key to the tractability of the policy-gradient method developed in Section~\ref{sec: algorithm and convergence}.
}

\subsection{\texorpdfstring{$\Gamma$}{Gamma}-convergence: recovering the original bi-causal optimal transport}
Having introduced the relaxed problem above, a fundamental question is whether the original bi-causal OT problem can be recovered in the limit as the penalty parameter $\beta\rightarrow\infty$. Addressing this question is essential for validating the relaxation, as it ensures that the penalized formulation is not only computationally tractable but also asymptotically consistent with the original constrained problem. In what follows, we show that, as $\beta\rightarrow\infty$, both the relaxed bi-causal OT coupling and the value in \eqref{eq:relaxed-bicausal-ot} converge to their respective counterparts in the original problem \eqref{eq:bicausal-ot}. To establish this convergence result, we first introduce the following standard assumption in OT.
\begin{assumption}
\label{asm:lsc_cost}
For each $n \in [N]_0$, the cost function $c_n(y,y')$ is nonnegative and lower semi-continuous.
Moreover, the original problem \eqref{eq:bicausal-ot} has finite infimum, i.e.,
\[
{\mathcal M}_{\mathrm{bc}}(\mu,\mu')\neq\emptyset
\quad\text{and}\quad
\inf_{\pi\in{\mathcal M}_{\mathrm{bc}}(\mu,\mu')}
\int \Big(\sum_{n=0}^N c_n\Big)\,d\pi<\infty.
\]
\end{assumption}
In Assumption~\ref{asm:lsc_cost}, nonnegativity and lower semi-continuity ensure weak stability of the transport objective under weak convergence \citep{villani2009optimal,santambrogio2015optimal}; the finiteness condition rules out degenerate cases where the bi-causal constraint renders the value $+\infty$, so the limit problem remains well-defined and nontrivial \citep{rachev1998mass,villani2009optimal}.

We define the closures of \(\Pi_{\rm rel}\) and \(\mathcal M_{\rm bc}\) under the weak convergence of measures ($\rightharpoonup$) by $\overline{\Pi}_{\textrm{rel}}$ and $\overline{\mathcal{M}}_{\textrm{bc}}$, respectively.
For notational simplicity, we denote the summation of the expectation of KL divergences by
\begin{equation*}
\mathrm{KL}^{(1)}\big(\mu, \pi\big) := \sum_{n=0}^N \mathrm{KL}_{n}^{\!(1)}\big(\mu, \pi\big),\quad \mathrm{KL}^{(2)}\big(\mu', \pi\big) := \sum_{n=0}^N \mathrm{KL}_{n}^{\!(2)}\big(\mu', \pi\big).
\end{equation*}
Furthermore, we introduce two functionals: the limiting constrained functional $F$ that enforces the bi-causal constraints through indicator terms, and the penalized functional $F_\beta$ that replaces these hard constraints by KL penalties.
\begin{align}
    F(\pi) &:= \mathbb E_{\pi}\left[\sum_{n=0}^N c_n(Y_n,Y_n')\right] + \iota\Big(\mathrm{KL}^{(1)}\big(\mu, \pi\big)\Big) + \iota\Big(\mathrm{KL}^{(2)}\big(\mu', \pi\big)\Big), \label{eq:def_F} \\
    F_{\beta}(\pi) &:=\mathbb E_{\pi}\left[\sum_{n=0}^N c_n(Y_n,Y_n') +\beta \mathrm{KL}^{(1)}\big(\mu, \pi\big) + \beta \mathrm{KL}^{(2)}\big(\mu', \pi\big)\right],\label{eq:def_F_beta}
\end{align}
where $\iota(z)=0$ if $z=0$, otherwise $\iota(z)=+\infty$.  Let $\{\pi^\beta\}_{\beta>0}\subset\overline{\Pi}_{\mathrm{rel}}$
be any collection of $\epsilon_\beta$-almost minimizers such that \(\epsilon_\beta\downarrow 0\) as \(\beta\to+\infty\), and for any \(\beta>0\),
\begin{equation}
\label{eq:def_almost_minimizer}
F_\beta(\pi^\beta)\le \inf_{\pi\in\overline{\Pi}_{\mathrm{rel}}}F_\beta(\pi)+\varepsilon_\beta.
\end{equation}

Then we are ready to state the $\Gamma$-convergence in the following theorem, which captures the variational convergence of the penalized functionals to the constrained problem, guaranteeing convergence of both infimum values and minimizers.

\begin{theorem}[$\Gamma$-convergence]
\label{thm:gamma-convergence}
Under Assumption~\ref{asm:lsc_cost}, the following two claims hold.
\begin{enumerate}
\item[(i)] 
Infimum values under relaxation converge to the optimal value of the original problem in \eqref{eq:bicausal-ot}: as $\beta \to \infty$,
\begin{eqnarray*}
    \inf_{\pi\in \overline{\Pi}_{\mathrm{rel}}}\mathbb E_{\pi}\left[\sum_{n=0}^N c_n(Y_n,Y_n') +\beta \mathrm{KL}^{(1)}\big(\mu, \pi\big) + \beta   \mathrm{KL}^{(2)}\big(\mu', \pi\big)\right]\longrightarrow \min_{\pi \in \overline{\mathcal{M}}_{\mathrm{bc}}(\mu,\mu')}\mathbb E_{\pi}\left[\sum_{n=0}^N c_n(Y_n,Y_n')\right].
\end{eqnarray*}
\item[(ii)] Every weak limit point of any sequence of $\epsilon_\beta$-almost minimizers $\{\pi^{\beta}\}_{\beta>0}$ as \(\beta\to+\infty\), denoted by $\pi^*$, is the optimal solution of the original problem in \eqref{eq:bicausal-ot} satisfying
\[
\E_{\pi^*}\Big[ \sum_{n=0}^N c_n(Y_n,Y_n')\Big]
=\min_{\pi \in \overline{\mathcal{M}}_{\mathrm{bc}}(\mu,\mu')}\mathbb E_{\pi}\left[\sum_{n=0}^N c_n(Y_n,Y_n')\right].
\]
\end{enumerate}

\end{theorem}
Here, under the notion of weak convergence, no uniqueness of minimizers is required, either for the relaxed
problems or for the original problem.

\begin{remark}[Proof sketch]
The main technical challenge is to show that a KL-penalized problem posed on a strictly larger coupling class still recovers the original bi-causal constraint in the limit. This step is delicate because the penalty must enforce feasibility while preserving compactness under weak convergence. Meanwhile, it is necessary to establish not only convergence of the penalized values, but also convergence of almost minimizers to the feasible Markov bi-causal set.

We address these difficulties through a three-step \(\Gamma\)-convergence argument. First, lower semicontinuity of the transport cost and the KL terms yields the \(\Gamma\)-liminf bound. Second, we construct recovery sequences on the feasible set to give the \(\Gamma\)-limsup bound. Finally, we prove that KL control yields asymptotic marginal consistency and tightness of almost minimizers. These properties imply that every weak limit point is feasible and optimal.
\end{remark}

Theorem~\ref{thm:gamma-convergence} establishes the asymptotic consistency of the relaxed bi-causal OT problem \eqref{eq:relaxed-bicausal-ot} with the original constrained problem \eqref{eq:bicausal-ot}. In particular, as $\beta\to\infty$, minimizers of the penalized objective~\eqref{eq:relaxed-bicausal-ot} recover both the optimal value and the optimal coupling of the original problem~\eqref{eq:bicausal-ot}.

\section{Policy gradient algorithm and convergence guarantee}
\label{sec: algorithm and convergence}
Given the relaxed formulation and dynamic programming principle in Section~\ref{sec: relaxed formulation}, we now formally develop a KL-regularized policy gradient method for the bi-causal OT problem \eqref{eq:bicausal-ot}. Specifically, we parameterize the relaxed coupling class and derive a policy gradient representation of the relaxed bi-causal OT problem in Section~\ref{subsec: parameterized coupling and policy gradient representation}. Building on this representation, Section~\ref{subsec: gradient estimators} constructs the estimators for the gradient terms with control variates, Section~\ref{subsec: algorithm implementation} specifies its implementation in Algorithm~\ref{alg:kl-pg}, and Section~\ref{subsec: convergence guarantee and regret analysis} establishes the convergence guarantee and derives the regret bound for our policy gradient method.

\subsection{Parameterized coupling and policy gradient representation}
\label{subsec: parameterized coupling and policy gradient representation}
Building on the recursive structure in Lemma~\ref{lem:dpp2}, this subsection parameterizes the relaxed coupling class and derives a policy gradient representation of the relaxed bi-causal OT problem~(\ref{eq:relaxed-bicausal-ot}). This representation converts the relaxed bi-causal OT problem into a computationally tractable optimization over a finite-dimensional parameter space.

We first consider parameterized measures in the relaxed feasible set, which reduces the optimization over measures to optimization over a finite-dimensional parameter vector. Specifically, let \(\Theta\subset\R^p\) denote the set of parameters, with \(p\in\mathbb N^+\); define $\theta_n\in\Theta$ as the parameter of the kernel $Q^{\theta_n}_n\in \mathcal{U}_{\textrm{rel}}$ at time step $n \in [N]$; denote $\pmb \theta = (\theta_0,\cdots,\theta_N)\in\Theta^{N+1}$ and construct $\pi^{\pmb \theta}\in\Pi_{\rm rel}$ as in \eqref{eq:markov-coupling}, namely, for any \(B_n,B'_n\in\mathcal S\), \(n \in [N]_0\), and \(\pmb B=\prod_{n=0}^NB_n,\pmb B'=\prod_{n=0}^NB_n'\),
\begin{eqnarray}
\label{eq:parameterized_markov_coupling}
    \pi^{\pmb \theta}(\pmb B\times \pmb B') =\int_{\pmb B\times \pmb B'}\pi^{\theta_0}_0(\dd y_0,\dd y_0')\prod_{n=1}^NQ^{\theta_n}_n(\dd y_n,\dd y_n' \,|\, y_{n-1},y_{n-1}').
\end{eqnarray}
With such a parameterization, we consider a {\em policy gradient} method. For notational simplicity, when there is no confusion, we denote $(\widehat{\pmb Y}, \widehat{\pmb Y}\mkern-1mu')=(\widehat Y_0,\dots, \widehat Y_N, \widehat Y_0',\dots, \widehat Y_N')$ as the generated path from $\pi^{\pmb\theta}$ and abbreviate $\E_{(\widehat{\pmb Y}, \widehat{\pmb Y}\mkern-1mu') \sim \pi^{\pmb\theta}}[\cdot]$ as $\E_{\pi^{\pmb\theta}}[\cdot]$. To highlight the role of $\pmb \theta$, we rewrite the corresponding cost function as
\begin{eqnarray*}
    J_\beta(\pmb \theta) =\E_{\pi^{\theta_0}_0}\bigg[\widetilde J^{\pi^{\pmb \theta}}_0(\widehat Y_0,\widehat Y_0')+\beta\Big[\mathrm{KL}_{0}^{\!(1)}\big(\mu, \pi^{\pmb\theta}\big) + \mathrm{KL}_{0}^{\!(2)}\big(\mu', \pi^{\pmb\theta}\big)\Big]\bigg],
\end{eqnarray*}
where $\mathrm{KL}_{0}^{\!(1)}$ and $\mathrm{KL}_{0}^{\!(2)}$ are defined in \eqref{eq:expectation_kl_0}. Note that under the parameterization \eqref{eq:parameterized_markov_coupling}, \(\inf_{\pmb\theta\in\Theta^{N+1}}J_{\beta}(\pmb\theta)\) is a parameterized approximation of the relaxed bi-causal OT cost in \eqref{eq:relaxed-bicausal-ot}.

To distinguish the parameterized and relaxed transport cost \(J_0^{\pi^{\pmb\theta}}\) defined in \eqref{eq:cost-n} and the KL regularizer in $J_\beta(\pmb\theta)$, we make the following decomposition:
\begin{equation}
\label{eq:Jbeta-decom}
J_\beta(\pmb\theta)\ =\ J_{\mathrm{val}}(\pmb\theta)\ +\ \beta\,J_{\mathrm{KL}}(\pmb\theta),
\end{equation}
where $J_{\mathrm{val}}(\pmb\theta)$ and $J_{\mathrm{KL}}(\pmb\theta)$ are defined as
\begin{align}
    J_{\mathrm{val}}(\pmb\theta)\, &\coloneqq\, \E_{\pi^{\pmb\theta}}\big[J_0^{\pi^{\pmb\theta}}(\widehat Y_0, \widehat Y_0')\big] = \E_{\pi^{\pmb\theta}}\Big[\sum_{n=0}^N c_n(\widehat Y_n,\widehat Y_n')\Big],\label{eq:Jval-def}\\
    J_{\mathrm{KL}}(\pmb\theta)\, &\coloneqq\,
    \sum_{n=0}^N \mathrm{KL}_{n}^{\!(1)}\big(\mu, \pi^{\pmb\theta}\big) + \mathrm{KL}_{n}^{\!(2)}\big(\mu', \pi^{\pmb\theta}\big).\label{eq:JKL-def}
\end{align}
with $\mathrm{KL}_{n}^{\!(i)}$ defined in \eqref{eq:expectation_kl_n_1} and \eqref{eq:expectation_kl_n_2} for \(i=1,2\) and \(n=0,\dots, N\). Here, \(J_{\mathrm{val}}\) represents the transport objective under the parameterized coupling, whereas \(J_{\mathrm{KL}}\) measures the aggregate violation of the prescribed marginal and transition constraints. This decomposition is the basis for the gradient formulas in Theorem~\ref{thm:policy-grad} and the estimator construction in Section~\ref{sec: algorithm and convergence}. To derive score-based gradient formulas for \(J_\beta(\pmb\theta)\) in Theorem~\ref{thm:policy-grad}, we make the following assumptions on regularity and moments.

\begin{assumption}
\label{asm:abs_cont_density}
For $n=0$, assume $\pi_0^{\theta_0}$ has a density $q_0^{\theta_0}$, with $q_0^{\theta_0, (i)}$ being the $i$-th marginal density of $\pi_0^{\theta_0,(i)}$, $i=1,2$. Assume the following Radon-Nikodym derivatives 
\begin{equation}
\label{eq:def_zeta_0}
    \zeta_0^{\theta_0,(1)}(y_0)=\frac{\dd\mu_0}{\dd\pi_0^{\theta_0,(1)}}(y_0),\,\, \zeta_0^{\theta_0,(2)}(y_0')=\frac{\dd\mu_0'}{\dd\pi_0^{\theta_0,(2)}}(y_0'),
\end{equation}
are well-defined. 

Similarly for each $n \in [N]$ and any $y_{n-1},y_{n-1}'\in S$, assume the transition kernel $Q_n^{\theta_n}(\cdot  \,|\,  y_{n-1}, y_{n-1}')$ has a density $q_n^{\theta_n}(\cdot  \,|\,  y_{n-1}, y_{n-1}')$, with $q_n^{\theta_n, (i)}$ being the conditional marginal density of $Q_n^{\theta_n,(i)}$, $i=1,2$. Assume the following Radon-Nikodym derivatives
\begin{equation}
\label{eq:def_zeta_n}
    \zeta_n^{\theta_n,(1)}(y_n  \,|\,  y_{n-1}, y_{n-1}')=\frac{\dd T_n}{\dd Q_n^{\theta_n,(1)}}(y_n  \,|\,  y_{n-1}, y_{n-1}'),\,\, \zeta_n^{\theta_n,(2)}(y_n'  \,|\,  y_{n-1}, y_{n-1}')=\frac{\dd T_n'}{\dd Q_n^{\theta_n,(2)}}(y_n'  \,|\,  y_{n-1}, y_{n-1}'),
\end{equation}
are well-defined. 

Moreover, assume that all the above densities are strictly positive on their supports, and for each $i=1,2$ and each $n=0,1,\dots,N$, both $q_n^{\theta_n,(i)}$ and $\log q_n^{\theta_n,(i)}$ are $\mathcal C^1$ in $\theta_n$ almost everywhere.
\end{assumption}

In the remainder of the paper, with a slight abuse of notation, we write $Y_{-1}=Y'_{-1}=\emptyset$ so that the one-step conditional objects can be written uniformly for all $n \in [N]_0$.
\begin{assumption}
\label{asm:exchange_expectation_gradient}
For any $\pmb\theta \in \Theta^{N+1}$ and each $n \in [N]_0$, assume $\widetilde J^{\pi^{\pmb\theta}}_n$ defined in \eqref{eq:relaxed-cost-n} with the terminal condition \eqref{eq:relaxed-terminal-cost} satisfies
\begin{equation*}
\E_{\pi^{\pmb\theta}}\Big[\big|\widetilde J^{\pi^{\pmb\theta}}_n(\widehat Y_n,\widehat Y_n')\big|\Big] < \infty.
\end{equation*}
There exist neighborhoods $\mathcal N_n(\theta_n) \subset \Theta$ for all $n \in [N]_0$, such that
\begin{equation*}
    \sup_{\vartheta_n\in \mathcal N_n(\theta_n)}\E_{\pi^{\pmb\theta}}\Big[
    \,|\widetilde J^{\pi^{\pmb\theta}}_n(\widehat Y_n,\widehat Y_n')|\,
    \big\|\nabla_{\vartheta_n}\log q_n^{\vartheta_n}(\widehat Y_n,\widehat Y_n' \,|\,  \widehat Y_{n-1},\widehat Y_{n-1}')\big\|
    \Big]\ <\ \infty,
\end{equation*}
and
\begin{equation*}
\begin{split}
    \sup_{\vartheta_n\in \mathcal N_n(\theta_n)}
\E_{\pi^{\pmb\theta}}\Big[&\big|\log \zeta_n^{\vartheta_n,(1)}(\widehat Y_n \,|\,  \widehat Y_{n-1},\widehat Y_{n-1}')\big|^2 + \big\|\nabla_{\vartheta_n}\log q_n^{\vartheta_n,(1)}(\widehat Y_n \,|\,  \widehat Y_{n-1},\widehat Y_{n-1}')\big\|^2 \\
&+\big|\log \zeta_n^{\vartheta_n,(2)}(\widehat Y_n' \,|\,  \widehat Y_{n-1},\widehat Y_{n-1}')\big|^2 + \big\|\nabla_{\vartheta_n}\log q_n^{\vartheta_n,(2)}(\widehat Y_n' \,|\,  \widehat Y_{n-1},\widehat Y_{n-1}')\big\|^2\Big] < \infty.
\end{split}
\end{equation*}
\end{assumption}

The regularity condition in Assumption~\ref{asm:abs_cont_density} is standard in both parametric statistics and policy gradient theory \citep{schervish1995theory,vandervaart1998asymptotic,bogachev2007measure,agarwal2021theory}. It guarantees that the score functions $\nabla_{\pmb\theta} \log q^{\pmb\theta}$ are well defined and that the KL divergence terms remain finite.

Assumption~\ref{asm:exchange_expectation_gradient} imposes the integrability and local moment conditions required to interchange differentiation and expectation for parameterized measures and to ensure finite-variance score estimators. These are the classical sufficient conditions underlying the logarithmic derivative trick used in Monte Carlo gradient estimation and policy gradient methods \citep{glynn1989likelihood,williams1992simple,billingsley1995probability}. For instance, exponential-family or Gaussian kernels with smooth parameterizations satisfy this assumption.

The following theorem provides an explicit formula for the gradient of $J_\beta(\pmb\theta)$ and a closed-form decomposition.

\begin{theorem}[Policy gradient theorem]
\label{thm:policy-grad} 
Under Assumptions \ref{asm:abs_cont_density} and \ref{asm:exchange_expectation_gradient}, for any \(\pi^{\pmb{\theta}}\) as in \eqref{eq:parameterized_markov_coupling} parameterized by \(\pmb\theta\in\Theta^{N+1}\), the gradient can be represented as
\begin{equation}
\label{eq:Jbeta-grad-decom}
\nabla_{\theta_n} J_\beta(\pmb\theta)\ =\ \nabla_{\theta_n} J_{\mathrm{val}}(\pmb\theta)\ +\ \beta\,\nabla_{\theta_n} J_{\mathrm{KL}}(\pmb\theta).
\end{equation}
For any $n \in [N]_0$, we have
\begin{align}
    \nabla_{\theta_n}J_{\mathrm{val}}(\pmb\theta)
    &= \E_{\pi^{\pmb\theta}}\bigg[\Big(\sum_{k=n}^{N} c_k(\widehat Y_k,\widehat Y_k')\Big)\ \nabla_{\theta_n}\log q_n^{\theta_n}\big(\widehat Y_n,\widehat Y_n' \,\big|\, \widehat Y_{n-1},\widehat Y_{n-1}'\big)\bigg],
    \label{eq:grad-Jval-n} \\
    \nabla_{\theta_n}J_{\mathrm{KL}}(\pmb\theta)
    &= -\E_{\mu \otimes \mu'}\Big[\nabla_{\theta_n} \log q_n^{\theta_n,(1)}(Y_n \,|\,  Y_{n-1},Y_{n-1}')+\nabla_{\theta_n} \log q_n^{\theta_n,(2)}(Y_n' \,|\,  Y_{n-1},Y_{n-1}')\Big].\label{eq:grad-JKL-n}
\end{align}
\end{theorem}

\begin{remark}[Proof sketch]
The main technical challenge is to derive a gradient formula that respects both the sequential structure of the coupling and the asymmetric KL penalization:
\begin{enumerate*}
    \item[{(1)}] For the value component, the difficulty is to isolate the contribution of $\theta_n$ and express it only through the remaining cost from time \(n\) onward. We address this point through the likelihood-ratio identity and the Markov factorization of the parameterized coupling.
    \item[{(2)}] For the KL component, the difficulty is that the penalty is defined through marginal kernels under the reference measure rather than under the current coupling. We handle this problem by differentiating under the \textit{reference measure} and then representing the resulting marginal-score terms as score expectations through the Radon--Nikodym representation.
\end{enumerate*}

The proof therefore has two parts. We first derive the gradient of the value term by applying the likelihood-ratio identity and then using the Markov structure to reduce the weight to the continuation cost from time \(n\) onward. We next derive the gradient of the KL term under the reference dynamics and represent it as the expectation of marginal scores. Combining these two identities yields the decomposition in Theorem~\ref{thm:policy-grad}.
\end{remark}

Theorem~\ref{thm:policy-grad} has several implications.
First, consistent with the decomposition of \(J_\beta(\pmb\theta)\) in \eqref{eq:Jbeta-decom}, Theorem~\ref{thm:policy-grad} shows that the gradient \(\nabla J_\beta(\pmb\theta)\) also splits into two structurally distinct components: \begin{enumerate*}
    \item[{(1)}] \(\nabla J_{\mathrm{val}}\), which is an expectation under the current parameterized coupling weighted by the remaining cumulative cost; and
    \item[{(2)}] \(\nabla J_{\mathrm{KL}}\), which is an expectation under the prescribed marginal dynamics and captures the KL-penalty term.
\end{enumerate*}
This decomposition is central to the estimator construction in Section~\ref{sec: algorithm and convergence}.

Second, the representation in Theorem~\ref{thm:policy-grad} expresses the gradient entirely in terms of the Stein's score functions of the parameterized kernels. In particular, the gradient takes the form of weighted log-derivatives, with weights given by the cumulative future cost. This structure enables unbiased Monte Carlo estimation of the gradient using only samples from $\pi^{\pmb\theta}$, without requiring differentiation through the underlying dynamics. As a result, the method is particularly attractive in high-dimensional settings, where explicit density evaluation or backpropagation through the full trajectory is computationally expensive. 

Third, in terms of computational benefits, the KL contribution admits a particularly simple form: it is given by an expectation under the reference marginals and transition kernels, and hence can be computed without sampling from the current policy. This makes the regularization term particularly tractable to implement, while also clarifying the role of $\beta$ in controlling the strength of the relaxation.

Finally, the above identity is closely related to the classical likelihood-ratio policy gradient theorem in reinforcement learning. In this correspondence, the cumulative future cost plays the role of the return-to-go, while the KL penalization is analogous to regularization terms commonly used in entropy-regularized or trust-region-based policy optimization.

\subsection{Mini-batch gradient estimators and variance reduction}
\label{subsec: gradient estimators}
Theorem~\ref{thm:policy-grad} gives an exact gradient formula for the relaxed objective. For practical implementation, however, this gradient must be estimated from samples. This subsection constructs unbiased mini-batch estimators for the value and KL terms, and introduces control variates for variance reduction.

More specifically, evaluating $\nabla J_{\beta}(\pmb\theta)$ relies on two sources of randomness: independent trajectories $\pmb Y\sim\mu$ and $\pmb Y'\sim\mu'$ from the reference marginal distributions, and generated trajectories $(\widehat{\pmb Y},\widehat{\pmb Y}\mkern-1mu')\sim\pi^{\pmb\theta}$ from the $\pmb\theta$-parameterized joint coupling. This sampling structure underlies both the estimator construction below and the implementation of Algorithm~\ref{alg:kl-pg} in Section~\ref{subsec: algorithm implementation}. 

To describe the implementation in detail, we introduce the iteration index $k\in\mathbb N$, the mini-batch size $B\in\mathbb N^+$, and the sample index $b \in [B]$ within each mini-batch. At iteration $k$, given the parameter $\pmb\theta_k := (\theta_{k,0}, \theta_{k,1}, \dots, \theta_{k,N})$, we first draw a mini-batch of $B$ i.i.d.\ trajectories $\{(\widehat{\pmb Y}^{(b)},\widehat{\pmb Y}'^{(b)})\}_{b=1}^{B}\sim \pi^{\pmb\theta_k}$ and $\{(\pmb Y^{(b)},\pmb Y'^{(b)})\}_{b=1}^B \sim \mu \otimes \mu'$. To construct the control variates later, we also draw an auxiliary mini-batch of $B$ i.i.d.\ generated trajectories
$$
\{(\widetilde{\pmb Y}^{(b)},\widetilde{\pmb Y}'^{(b)})\}_{b=1}^{B} \sim \pi^{\pmb\theta_k},
$$
independent of the batch $\{(\widehat{\pmb Y}^{(b)},\widehat{\pmb Y}'^{(b)})\}_{b=1}^{B}$. Motivated by the gradient decomposition of $J_\beta$ in \eqref{eq:Jbeta-grad-decom} evaluated at $\pmb\theta=\pmb\theta_k$, we can estimate $\nabla_{\pmb\theta_k} J_{\mathrm{val}}(\pmb\theta_k)$ and $\nabla_{\pmb\theta_k} J_{\mathrm{KL}}(\pmb\theta_k)$ separately.

\paragraph{Per–trajectory statistics.}
To mirror the gradient decomposition in Theorem~\ref{thm:policy-grad}, we record three per-trajectory quantities: a remaining cost term \(\widehat V_{k,n}^{(b)}\), a generated-sample score term \(\widehat S_{k,n}^{(b)}\) for the value component, and a reference-sample score term \(\widehat K_{k,n}^{(b)}\) for the KL component.
Specifically, for any $n \in [N]_0$ and each trajectory index $b \in [B]$, to estimate the remaining cumulative cost in the value part $\nabla_{\pmb\theta_k} J_{\mathrm{val}}(\pmb\theta_k)$, we denote the ``per-trajectory cost'' by
\begin{equation}
\label{eq:Vhat-def}
    \widehat V_{k,n}^{(b)}:=\sum_{j=n}^{N} c_j\big(\widehat Y_j^{(b)}, \widehat Y_j'^{(b)}\big).
\end{equation}
Then, we define the ``per-trajectory joint score'' of the parameterized coupling evaluated on the generated trajectory $(\widehat{\pmb Y}^{(b)},\widehat{\pmb Y}'^{(b)})$ as
\begin{align}
\widehat S_{k,n}^{(b)}&:=
\nabla_{\theta_{k,n}}\log q_n^{\theta_{k,n}}\big(\widehat Y_n^{(b)},\widehat Y_n'^{(b)} \,\big|\, \widehat Y_{n-1}^{(b)},\widehat Y_{n-1}'^{(b)}\big). \label{eq:Shat-def}
\end{align}
Finally, for the KL-gradient term $\nabla_{\pmb\theta_k} J_{\mathrm{KL}}(\pmb\theta_k)$, we define the ``per-trajectory marginal score'' evaluated on the training data $(\pmb Y^{(b)},\pmb Y'^{(b)})$ as
\begin{align}
    \widehat K_{k,n}^{(b)}&:=
    \nabla_{\theta_{k,n}}\log q_n^{\theta_{k,n}, (1)}\big(Y_n^{(b)} \,\big|\, Y_{n-1}^{(b)},Y_{n-1}'^{(b)}\big)+\nabla_{\theta_{k,n}}\log q_n^{\theta_{k,n}, (2)}\big(Y_n'^{(b)} \,\big|\, Y_{n-1}^{(b)},Y_{n-1}'^{(b)}\big). \label{eq:Khat-def}
\end{align}

\paragraph{Mini-batch gradient estimators.}
Using the per–trajectory statistics in \eqref{eq:Shat-def}--\eqref{eq:Khat-def}, for each $n=0,1,\dots,N$, we define mini-batch estimators of $\nabla_{\theta_{k,n}}J_{\textrm{val}}(\pmb\theta_k)$ and $\nabla_{\theta_{k,n}}J_{\textrm{KL}}(\pmb\theta_k)$ as
\begin{align}
\widehat G_{k,\mathrm{val},n}
&:=\frac{1}{B}\sum_{b=1}^{B}\widehat V_{k,n}^{(b)}\,\widehat S_{k,n}^{(b)},\quad \widehat G_{k,\mathrm{KL},n}:=-\frac{1}{B}\sum_{b=1}^{B}\widehat K_{k,n}^{(b)},\label{eq:Gval-GKL-n-def}
\end{align}
and aggregate estimators as
\begin{equation}
\label{eq:agg-Gval-GKL-def}
    \widehat G_{k,\mathrm{val}}:=\big(\widehat G_{k,\mathrm{val},0},\dots,\widehat G_{k,\mathrm{val},N}\big),\quad \widehat G_{k,\mathrm{KL}}:=\big(\widehat G_{k,\mathrm{KL},0},\dots,\widehat G_{k,\mathrm{KL},N}\big)
\end{equation}
Here, \(\widehat G_{k,\mathrm{val},n}\) combines the remaining cost with the generated-sample score, whereas \(\widehat G_{k,\mathrm{KL},n}\) is constructed from the score evaluated on the training data. Their unbiasedness is established in Lemma~\ref{lem:unbiased-cov}.

\paragraph{Control variates.}
To reduce the variance of the estimators without introducing bias, we employ state–dependent control variates for the value component \citep{williams1992simple,greensmith2004variance}. We consider linear control variates, which are standard and effective \citep{konda2000actor,sutton2000policy}.
Mathematically, for any $n \in [N]$, given a feature map $\phi_{k,n}$, we define linear control variates as
\begin{align}
&\widehat l_{k,n}(y,y')=\widehat w_{k,n}^\top\phi_{k,n}(y,y'), \label{eq:control-variate-n-def}
\end{align}
with the coefficient fitted on the auxiliary mini-batch using ordinary least squares:
\begin{align}
&\widehat w_{k,n}\in\arg\min_{w}\ \frac{1}{B}\sum_{b=1}^{B}\Big\|\big(\widetilde V_{k,n}^{(b)}-w^\top\phi_{k,n}(\widetilde Y_{n-1}^{(b)},\widetilde Y_{n-1}'^{(b)})\big)\widetilde S_{k,n}^{(b)}\Big\|^2, \label{eq:wls-val}
\end{align}
where $\widetilde V_{k,n}^{(b)}$ and $\widetilde S_{k,n}^{(b)}$ are computed by replacing the input mini-batch $(\widehat{\pmb Y}^{(b)},\widehat{\pmb Y}'^{(b)})$ of \eqref{eq:Vhat-def} and \eqref{eq:Shat-def} by the auxiliary mini-batch $(\widetilde{\pmb Y}^{(b)},\widetilde{\pmb Y}'^{(b)})$.
For $n=0$, since no previous state is available, we use the following constant control variates
\begin{align}
\widehat l_{k,0}
&=\frac{\sum_{b=1}^{B}\widetilde V_{k,0}^{(b)}\big\|\widetilde S_{k,0}^{(b)}\big\|_2^2}{\sum_{b=1}^{B}\big\|\widetilde S_{k,0}^{(b)}\big\|_2^2} \in \arg\min_{c\in\R}\ \frac{1}{B}\sum_{b=1}^{B}\Big\|\big(\widetilde V_{k,0}^{(b)}-c\big)\widetilde S_{k,0}^{(b)}\Big\|^2. \label{eq:control-variate-0-def}
\end{align}
Multiplying $\widehat l_{k,n}$ by $\widehat S_{k,n}^{(b)}$ and stacking with respect to time index $n$ and trajectory index $b$, we denote
\begin{equation}
\label{eq:agg-control-variate-def}
\widehat L_k = \Big(\widehat L_{k,0},\dots, \widehat L_{k,N}\Big) := \frac{1}{B}\Big(\sum_{b=1}^{B}\widehat l_{k,0}\,\widehat S_{k,0}^{(b)},\,\dots,\, \sum_{b=1}^{B}\widehat l_{k,N} \big(\widehat Y_{N-1}^{(b)},\widehat Y_{N-1}'^{(b)}\big)\,\widehat S_{k,N}^{(b)}\Big).
\end{equation}

In Lemma~\ref{lem:unbiased-cov} below, we provide details of unbiasedness and asymptotic variance reduction achieved by introducing the above control variates \eqref{eq:control-variate-n-def}--\eqref{eq:agg-control-variate-def}.

\paragraph{Aggregate gradient estimators with control variates.}
Combining the mini-batch estimators $\widehat G_{k,\mathrm{val}}$, $\widehat G_{k,\mathrm{KL}}$, and the control variates $\widehat L_k$ defined in \eqref{eq:agg-Gval-GKL-def} and \eqref{eq:agg-control-variate-def} together, we denote the aggregate estimator of $\nabla_{\pmb \theta_k} J_\beta(\pmb\theta_k)$ by
\begin{equation}
\label{eq:gkhat-def}
\widehat g_k\ :=\ \widehat G_{k,\mathrm{val}}\ +\ \beta\,\widehat G_{k,\mathrm{KL}} - \widehat L_k.
\end{equation}

The following lemma establishes unbiasedness and asymptotic positive-semidefinite (PSD) covariance improvements for the mini-batch gradient estimators with control variates.

\begin{lemma}[Unbiasedness and PSD covariance improvement]\label{lem:unbiased-cov}
At iteration $k$, fix $\pmb\theta_k$ and suppose that $\{(\pmb Y^{(b)},\pmb Y'^{(b)})\}_{b=1}^B \sim \mu \otimes \mu'$, $\{(\widehat{\pmb Y}^{(b)},\widehat{\pmb Y}'^{(b)})\}_{b=1}^{B} \sim \pi^{\pmb \theta_k}$, and $\{(\widetilde{\pmb Y}^{(b)},\widetilde{\pmb Y}'^{(b)})\}_{b=1}^{B} \sim \pi^{\pmb \theta_k}$ are i.i.d.\ random trajectories. Then, the following two results hold.
\begin{enumerate}
\item \emph{(Unbiasedness).} The gradient estimators $\widehat G_{k,\mathrm{val}}$ and $\widehat G_{k,\mathrm{KL}}$ in \eqref{eq:agg-Gval-GKL-def} satisfy
\begin{align}
    \E_{\pi^{\pmb\theta_k}}[\widehat G_{k,\mathrm{val}} + \beta\widehat G_{k,\mathrm{KL}} \mid \pmb\theta_k] = \,\nabla_{\pmb\theta_k} J_{\mathrm{val}}(\pmb\theta_k)+\,\beta\,\nabla_{\pmb\theta_k} J_{\mathrm{KL}}(\pmb\theta_k), \label{eq:unbiased-gk0}
\end{align}
Moreover, the aggregate estimator $\widehat g_k$ in \eqref{eq:gkhat-def} is unbiased with
\begin{align}
    \E_{\pi^{\pmb\theta_k}}\left[\widehat g_k \,\big|\, \pmb\theta_k\right] = \nabla_{\pmb\theta_k} J_\beta(\pmb\theta_k).\label{eq:unbiased-gk}
\end{align}
\item \emph{(PSD covariance improvement).} For the gradient estimator $\widehat g_k$ in \eqref{eq:gkhat-def} with control variates $\widehat L_k$ defined in \eqref{eq:agg-control-variate-def}, it holds that as $B \to \infty$, we have
\begin{equation}\label{eq:cov-composite-psd}
    \operatorname{Cov}\big(\widehat g_k \,\big|\, \pmb\theta_k\big)\ \preceq\ \operatorname{Cov}\big(\widehat g_k^{0} \,\big|\, \pmb\theta_k\big),
\end{equation}
where $\operatorname{Cov}$ represents the covariance matrix and $\widehat g_k^{0}$ denotes $\widehat g_k$ in \eqref{eq:gkhat-def} with all control variates set to zero.
\end{enumerate}
\end{lemma}

Lemma~\ref{lem:unbiased-cov} justifies the use of \(\widehat g_k\) in Algorithm~\ref{alg:kl-pg}: unbiasedness ensures that the estimator targets the correct descent direction, while variance reduction improves numerical stability.

\subsection{Algorithm implementation}
\label{subsec: algorithm implementation}
Building on the gradient estimator developed in Section~\ref{subsec: gradient estimators}, we now formulate the KL-regularized policy gradient method, summarized in Algorithm~\ref{alg:kl-pg}. In each iteration, the method draws samples from both the current parameterized coupling and the reference distributions, constructs the mini-batch gradient estimator with control variates, and updates the parameter vector through a stochastic gradient step.

\begin{algorithm}[!ht]
\caption{KL-regularized policy gradient for Bi-causal OT}
\label{alg:kl-pg}
\KwIn{regularization $\beta > 0$; horizon $N$; mini-batch size $B$; update rounds $K$; initial parameters $\pmb\theta_{1}\in\Theta^{N+1}$; step sizes $\{\eta_k\}_{k=1}^{K}$; baseline features $\{\phi_{k,n}\}_{1 \le k \le K, 1 \le n \le N}$.}
\KwOut{parameters $\pmb\theta_{K+1}$.}

\For{$k=1,\dots,K$}{
\textbf{1) Sample:} draw $\{(\widehat{\pmb Y}^{(b)},\widehat{\pmb Y}'^{(b)})\}_{b=1}^{B}\sim\pi^{\pmb\theta_k}$, $\{(\widetilde{\pmb Y}^{(b)},\widetilde{\pmb Y}'^{(b)})\}_{b=1}^{B}\sim\pi^{\pmb\theta_k}$, $\{(\pmb Y^{(b)},\pmb Y'^{(b)})\}_{b=1}^{B} \sim \mu \otimes \mu'$.

\textbf{2) Per–trajectory statistics:} for all $n \in [N]_0$ and $b=1,\dots,B$, compute
\newline
\indent\textbullet\enspace $\widehat V_{k,n}^{(b)}$---the ``per-trajectory cost'' on the generated data $(\widehat{\pmb Y}^{(b)},\widehat{\pmb Y}'^{(b)})$ by \eqref{eq:Vhat-def};
\newline
\indent\textbullet\enspace $\widehat S_{k,n}^{(b)}$---the ``per-trajectory joint scores'' on the generated data $(\widehat{\pmb Y}^{(b)},\widehat{\pmb Y}'^{(b)})$ by \eqref{eq:Shat-def};
\newline
\indent\textbullet\enspace $\widehat K_{k,n}^{(b)}$---the ``per-trajectory marginal scores'' on the training data $(\pmb Y^{(b)},\pmb Y'^{(b)})$ by \eqref{eq:Khat-def}.

\textbf{3) Control variates (stop–gradient):} for all $n \in [N]$, fit
\newline
\indent\textbullet\enspace $\widehat l_{k,n}(\cdot)$---the linear control variates  on the auxiliary data $(\widetilde{\pmb Y}^{(b)},\widetilde{\pmb Y}'^{(b)})$ via \eqref{eq:wls-val};
\newline
\indent\textbullet\enspace fit $\widehat l_{k,0}$---the constant control variate on the auxiliary data $(\widetilde{\pmb Y}^{(b)},\widetilde{\pmb Y}'^{(b)})$ via \eqref{eq:control-variate-0-def}.

\textbf{4) Mini-batch sums and gradient:}
\newline
\indent\textbullet\enspace compute $\widehat G_{k,\mathrm{val}}$, $\widehat G_{k,\mathrm{KL}}$, and $\widehat L_{k}$ by \eqref{eq:agg-Gval-GKL-def} and \eqref{eq:agg-control-variate-def};
\newline
\indent\textbullet\enspace compute the aggregate estimator $\widehat g_k$ via \eqref{eq:gkhat-def}.

\textbf{5) Update:} $\pmb\theta_{k+1}\gets \pmb\theta_{k}-\eta_k\,\widehat g_k$.
}
\end{algorithm}

\subsection{Convergence guarantee and regret analysis}
\label{subsec: convergence guarantee and regret analysis}
We next study the convergence behavior of Algorithm~\ref{alg:kl-pg} and quantify how the optimization error depends on the iteration number \(K\), the mini-batch size \(B\), and the step-size schedule. To this end, let $J_\beta^* := \inf_{\pmb\theta \in \Theta^{N+1}} J_\beta(\pmb\theta)$,
and consider the stochastic gradient iterates $\pmb\theta_{k+1} = \pmb\theta_k - \eta_k \widehat g_k$. We make the following assumptions on the objective $J_\beta(\pmb\theta)$ and the stochastic gradient iterates.

\begin{assumption}
\label{asm:beta-smooth-PL-var}
For any $\beta \geq 0$, assume there exist positive constants $L(\beta),\mu(\beta)>0$ and finite constants $\sigma_{\mathrm{val}},\sigma_{\mathrm{KL}} \geq 0$ such that
\begin{enumerate}
\item \emph{(Smoothness).} $J_\beta$ is $L(\beta)$-smooth:
\begin{equation}
\label{eq:asm-smooth}
\|\nabla J_{\beta}(\pmb\theta)-\nabla J_{\beta}(\pmb\theta')\|\ \le\ L(\beta)\,\|\pmb\theta-\pmb\theta'\|, \ \textrm{ for all }\ \pmb\theta,\pmb\theta' \in \Theta^{N+1}.
\end{equation}
\item \emph{(Polyak--\L ojasiewicz condition).} $J_\beta$ satisfies
\begin{equation}
\label{eq:asm-PL-condition}
    \frac{1}{2}\,\|\nabla J_{\beta}(\pmb\theta)\|^2 \ge \mu(\beta)\,\big(J_{\beta}(\pmb\theta)-J_{\beta}^{*}\big).
\end{equation}
\item \emph{(Bounded second moments).}
The conditional second moments of $\widehat G_{k,\mathrm{val}}$ and $\widehat G_{k,\mathrm{KL}}$ scale with the mini-batch size $B$ as
\begin{align}
\label{eq:component-variance-bounds-G}
\E\Big[\big\|\widehat G_{k,\mathrm{val}}-\E\big[\widehat G_{k,\mathrm{val}}\mid \pmb\theta\big]\big\|^2\,\big|\,\pmb\theta\Big] &\leq \frac{\sigma_{\mathrm{val}}^{2}}{B}, \quad
\E\Big[\big\|\widehat G_{k,\mathrm{KL}}-\E\big[\widehat G_{k,\mathrm{KL}}\mid \pmb\theta\big] \big\|^2 \,\big|\,\pmb\theta\Big] \leq \frac{\sigma_{\mathrm{KL}}^{2}}{B}.
\end{align}
\end{enumerate}
\end{assumption}

In Assumption~\ref{asm:beta-smooth-PL-var}, the Lipschitz–gradient hypothesis is standard for smooth models and underpins the descent lemma \citep{nesterov2004introductory}. The Polyak--Lojasiewicz (PL) condition, originating in \citet{lojasiewicz1963une} and \citet{polyak1963gradient}, is strictly weaker than strong convexity, yet still guarantees global linear (noise–limited) convergence for first–order methods; see \citet{karimi2016linear} for a modern treatment. 
Importantly, for the bounded second moments, scaling with the mini-batch size $B$ matches the classical stochastic approximation and stochastic gradient descent theory \citep{nemirovski2009robust,bottou2018optimization}. 

To summarize the optimization performance over the first $K$ iterations, we measure the average regret, namely the average suboptimality of the iterates $K$ relative to $J_\beta^*$, which is defined as
\begin{equation}
\label{eq:regret-def}
\overline{\mathrm{Reg}}_K(\beta)\ :=\ \frac{1}{K}\sum_{k=1}^{K}\,\mathrm{Err}_k(\beta),
\end{equation}
with the error of $k$-th iteration by $\mathrm{Err}_k(\beta) :=\E[J_\beta(\pmb\theta_k)-J_\beta^{*}]$ for all $k=1,\dots,K$.
We establish the error bound for the average regret over $K$ iterations and provide the convergence guarantee for $J_\beta(\pmb\theta_{K})$ with a fixed \(\beta\) in the following theorem.
\begin{theorem}
\label{thm:beta-PG-converge}
Under Assumption~\ref{asm:beta-smooth-PL-var} and suppose the step sizes $\{\eta_k\}_{k=1}^{K}$ satisfy $0< \eta_K \leq \eta_{K-1} \leq \dots \leq \eta_1 \leq \min\{1/L(\beta), 1/\mu(\beta)\}$. With a slight abuse of notation $\sum_{k=1}^{0} \eta_k := 0$, the following two statements hold.
\begin{enumerate}
\item \emph{(Average regret bound).} For any $K \geq 1$, the average regret in \eqref{eq:regret-def} is bounded by
\begin{equation}
\label{eq:bound-avg-regret-monotone}
\overline{\mathrm{Reg}}_K(\beta) \leq
\frac{\mathrm{Err}_1(\beta)}{K}\sum_{k=1}^{K}\prod_{s=1}^{k-1}\big(1-\mu(\beta)\eta_s\big)
\ +\ \frac{L(\beta)\sigma^2(\beta)}{2 B K}\sum_{k=1}^{K-1}\eta_k^2\sum_{t=k}^{K-1}\prod_{s=k+1}^{t}\big(1-\mu(\beta)\eta_s\big).
\end{equation}
\item \emph{(Last–iterate proximity).}
If the set $\mathcal A_{\mathrm{KL}}:=\{\pmb\theta:\, J_{\mathrm{KL}}(\pmb\theta)=0\}$ is nonempty, then for all $K \geq 1$,
\begin{equation}
\label{eq:bound-last-iterate-prox}
\begin{aligned}
    &\quad \,\, \E\Big[J_\beta(\pmb\theta_{K})-\inf_{\pmb\theta\in\mathcal A_{\mathrm{KL}}}J_{\mathrm{val}}(\pmb\theta)\Big] \\
    &\leq
    \Big(\prod_{k=1}^{K}1-\mu(\beta)\eta_k\Big) \cdot \mathrm{Err}_1(\beta)\, +\, \frac{L(\beta)\,\sigma^2(\beta)}{2B}\sum_{k=1}^{K-1}\Big(\eta_k^2 \big(\prod_{t=k+1}^{K-1}(1-\mu(\beta)\eta_t)\big)\Big).
\end{aligned}
\end{equation}
\end{enumerate}
\end{theorem}

\begin{remark}[Proof sketch]
The main technical challenge is to derive nonasymptotic error bounds for a stochastic gradient method whose estimator depends on both the sampling noise and the penalty parameter \(\beta\). Specifically, the first difficulty is to obtain a one-step error recursion that simultaneously captures the descent effect of the algorithm and the stochastic fluctuation of the gradient estimator. To address this point, we establish one-step error in Lemma~\ref{lem:onestep-PG} by combining unbiasedness of estimators in Lemma~\ref{lem:unbiased-cov}, variance bound in Lemma~\ref{lem:beta-noise-gk}, and regularity conditions. The second difficulty is to convert the resulting one-step recursion into explicit bounds for the average regret and the last iterate. We address this point by unrolling the recursion through the discrete Gr\"onwall inequality in Lemma~\ref{lem:gronwall-err}.

The proof can be split into two parts. First, we establish the one-step descent inequality for the stochastic gradient iterate (as in Lemma~\ref{lem:onestep-PG}). We then unroll the recursion to derive the average-regret bound and the last-iterate proximity bound in Theorem~\ref{thm:beta-PG-converge}.
\end{remark}

Theorem~\ref{thm:beta-PG-converge} quantifies the optimization error of Algorithm~\ref{alg:kl-pg} for a fixed penalty parameter $\beta$. In particular, it provides both an average regret guarantee over the first $K$ iterations and the proximity of the last iterate to the optimal solution.

To make the regret bound in Theorem~\ref{thm:beta-PG-converge} more concrete, we discuss two typical step-size schedules and the corresponding regret rates.
The step-size sequence ${\eta_k}$ should be chosen so that $\sum_k \eta_k \to \infty$ to eliminate the initialization error, while $\sum_k \eta_k^2$ remains small to control the noise.
\begin{itemize}
\item \emph{Constant step size $0 < \eta_k\equiv\eta \le \min\{1/L(\beta), 1/\mu(\beta)\}$.} Substituting $\eta_k \equiv \eta$ into \eqref{eq:bound-avg-regret-monotone}, we have
\[
\overline{\mathrm{Reg}}_K(\beta)
=
\mathcal O\!\left(
\frac{\mathrm{Err}_1(\beta)}{K\eta}
+
\frac{L(\beta)\sigma^2(\beta)}{2B\,\mu(\beta)}\,\eta
\right).
\]
Here, the error consists of a $\mathcal O(1/(K\eta))$ initialization error and a $\mathcal O(\eta)$ stochastic error.

\item \emph{Polynomial decay $\eta_k=\frac{1}{k^{\alpha}\mu(\beta)}$ with $\alpha\in(\tfrac{1}{2},1)$.} Using the fact that $\sum_{t=k}^{K-1}\prod_{s=k+1}^{t}\bigl(1-\mu(\beta)\lambda s^{-\alpha}\bigr) = \mathcal{O}(k^\alpha)$ and $\sum_{k=1}^{K} k^{-\alpha} = \mathcal{O}(K^{1-\alpha})$, plugging such $\eta_k$ into \eqref{eq:bound-avg-regret-monotone} yields the following results.
$$
\overline{\mathrm{Reg}}_K(\beta)=
\mathcal{O}\left(K^{-\alpha}\right), \quad\textrm{if}\,\, \alpha\in\left(\tfrac{1}{2},1\right)\,\,;\,\,
\overline{\mathrm{Reg}}_K(\beta)=\mathcal{O}\left(\dfrac{\log K}{K}\right), \quad\textrm{if}\,\, \alpha=1.
$$

\end{itemize}

\begin{remark}[Relation to stochastic approximation]
The significance of the above bounds is that, after the KL relaxation and Markov parameterization, the constrained infinite-dimensional bi-causal OT problem admits the same iteration-level behavior as a standard stochastic approximation problem. In classical stochastic optimization, smooth objectives satisfying a strong-convexity or Polyak--\L ojasiewicz-type growth condition achieve the canonical $\mathcal O(1/K)$ rate with unbiased stochastic gradients of controlled variance \citep{robbins1951stochastic,polyak1992acceleration,karimi2016linear,bottou2018optimization}. Indeed, with the optimized decay $\eta_k=\mathcal{O}(k^{-1})$, we recover the canonical $\mathcal O(1/K)$ rate without grid-based enforcement of bi-causality.
\end{remark}

Combined with the large-$\beta$ consistency result in Theorem~\ref{thm:gamma-convergence}, Theorem~\ref{thm:beta-PG-converge} yields an approximation guarantee for the original bi-causal OT problem. The following corollary formalizes this result for Algorithm~\ref{alg:kl-pg}.
\begin{corollary}[Approximation under large-$\beta$]
\label{cor:converge-largebeta}
Under the assumptions in Theorem~\ref{thm:beta-PG-converge}, fix $\epsilon\in(0,1)$ and define
\begin{align}
\label{eq:K0-choice}
    K_0(\beta,\epsilon)\ &:=\ \min\Bigg\{K\ge1:\ \sum_{k=1}^{K-1}\eta_k\ \ge\ \frac{1}{\mu(\beta)}\,\log\frac{2\cdot \mathrm{Err}_1(\beta)}{\epsilon}\Bigg\},\\
\label{eq:B0-choice}
    B_0(\beta,\epsilon,K)\ &:=\ \min\Bigg\{B \geq 1:\ \sum_{k=1}^{K}\eta_k^2 \ \leq\ \frac{B\,\epsilon}{L(\beta)\,\sigma^2(\beta)}\Bigg\}.
\end{align}
Then, for all $K\ge K_0(\beta,\epsilon)$ and $B\ge B_0(\beta,\epsilon,K)$, it holds
\begin{equation*}
\E\Big[J_\beta(\pmb\theta_{K})-\inf_{\pmb\theta\in\mathcal A_{\mathrm{KL}}}J_{\mathrm{val}}(\pmb\theta)\Big]
\ \le\ \epsilon,
\end{equation*}
where $\mathcal A_{\mathrm{KL}}$ is defined in Theorem \ref{thm:beta-PG-converge}.
\end{corollary}

Corollary~\ref{cor:converge-largebeta} provides an approximation guarantee for the bi-causal OT problem. In particular, once \(\beta\) is sufficiently large and the number of iterations \(K\) and the mini-batch size \(B\) are chosen appropriately, the output of Algorithm~\ref{alg:kl-pg} is guaranteed to be arbitrarily close in objective value.

\section{Synthetic experiments}
\label{sec: synthetic experiments}
In this section, we apply our framework, specifically Algorithm~\ref{alg:kl-pg}, to learn the bi-causal OT map in controlled synthetic settings. We demonstrate that the learned generator performs well along two axes: (i) \emph{marginal fidelity} at each time step and (ii) \emph{temporal dependence}, quantified via adjacent-step correlations along the trajectory. 

\subsection{Set-up}
\label{subsec:synthetic_exp_setup}

\paragraph{Data-generating process.}
We generate paired source and target trajectories from a first-order autoregressive model (AR(1)). For each asset $i\in\{1,\dots,d\}$, we simulate two trajectories $\{Y_{n,i}\}_{n=0}^N$ and $\{Y'_{n,i}\}_{n=0}^N$ according to
\begin{equation}
\label{eq:ar1_simulation}
Y_{n,i}   = \begin{cases}
    \mu_i + \varepsilon_{0,i}, & n=0 \\
    \mu_i + \phi\, Y_{n-1,i} + \varepsilon_{n,i}, & n=1,\ldots,N
\end{cases},
\qquad
Y'_{n,i}  = \begin{cases}
    \mu'_i + \varepsilon'_{0,i}, & n=0 \\
    \mu'_i + \phi\, Y'_{n-1,i} + \varepsilon'_{n,i}, & n=1,\ldots,N
\end{cases}
\end{equation}
where $\phi\in(-1,1)$ is the autoregressive coefficient, $\mu_i$ and $\mu'_i$ are asset-specific means, and
$\{\varepsilon_{n,i}\}$ and $\{\varepsilon'_{n,i}\}$ are innovation noises. Throughout, we model the distributions of the innovations to produce both Gaussian and non-Gaussian marginals while controlling the strength of temporal correlation.

We consider two complementary data-generating regimes, a unimodal setting and a bimodal setting with a time-constant selector:

\begin{itemize}
    \item  \emph{Unimodal setting.}
We set $\phi=0.5$ and $\mu_i=\mu'_i=-2+\frac{4(i-1)}{d-1}$, and draw the innovations from Gaussian and Beta distributions, respectively:
\footnote{The two shape parameters in Beta distribution control the distributional form (e.g., symmetric, skewed, and U-shaped), which makes it highly adaptable for modeling diverse phenomena \citep{johnson1995continuous}.}
\begin{equation}
\label{eq:ar1_unimodal}
\varepsilon_{n,i}\sim\mathcal{N}(0,0.5^2),
\qquad
\varepsilon'_{n,i}\sim 4.0\cdot\mathrm{Beta}(2.0,5.0)-1.5,
\qquad n=0,\dots,N.
\end{equation}
All assets share the same innovation law and autocorrelation structure, so the adjacent-step correlation satisfies
\begin{equation*}
\mathrm{Corr}(Y_{n,i},Y_{n-1,i})
=\rho_{n, \mathrm{uni}}:=\phi\sqrt{\frac{1-\phi^{2n}}{1-\phi^{2(n+1)}}},
\qquad n=1,\dots,N,
\end{equation*}
and the same expression holds for $\mathrm{Corr}(Y'_{n,i},Y'_{n-1,i})$.

\item \emph{Bimodal setting (time-constant selector).}
We set $\phi=0.5$ and $\mu_i=\mu'_i=-0.5+\frac{0.5(i-1)}{d-1}$. For each asset $i$, we draw a selector
$S_i=(S^{(y)}_i,S^{(y')}_i)$ once, with independent entries
$S^{(y)}_i,S^{(y')}_i\sim \mathrm{Bernoulli}(0.5)$, and keep it fixed across time. The innovations are then sampled from the regime-dependent mixtures
\begin{equation*}
\begin{split}
\varepsilon_{n,i} &= (1-S^{(y)}_i)\cdot \mathcal{N}(1.5,0.5^2) + S^{(y)}_i \cdot \big( 2.5 \cdot \mathrm{Beta}(2.0,5.0)-2.5 \big), \\
\varepsilon'_{n,i} &= (1-S^{(y')}_i)\cdot \mathcal{N}(-1.5,0.5^2) + S^{(y')}_i \cdot \big( 2.5\cdot \mathrm{Beta}(5.0,2.0)+0.5 \big).
\end{split}
\end{equation*}
In this setting, the adjacent-step correlation is given by
\begin{equation*}
\mathrm{Corr}(Y_{n,i},Y_{n-1,i})=\rho_{n,\mathrm{bi}}
:=\frac{\rho_{n,\mathrm{uni}}+\sqrt{\kappa_n\kappa_{n-1}}}
{\sqrt{(1+\kappa_n)(1+\kappa_{n-1})}},
\text{ with } \kappa_n:=\frac{\frac{1+\phi}{1+\phi^{n+1}} \cdot \mathrm{Var}\big(\E[ \varepsilon_{n,i}\mid S^{(y)}_i ]\big)}{\frac{1-\phi}{1-\phi^{n+1}} \cdot \E\big[\mathrm{Var}( \varepsilon_{n,i}\mid S^{(y)}_i)\big]}.
\end{equation*}
An analogous expression holds for $\mathrm{Corr}(Y'_{n,i},Y'_{n-1,i})$ by replacing $(\varepsilon_{n,i},S_i^{(y)})$ with $(\varepsilon'_{n,i},S_i^{(y')})$. Because the selector is constant along the trajectory, the process is an AR(1) conditional on a latent regime, and hence exhibits a persistent random-intercept effect. Unconditionally, this increases adjacent-step dependence beyond the $\phi$-driven component and yields substantially larger correlations than in the unimodal setting.

\end{itemize}

\paragraph{Normalizing-flow parameterization.}
We parametrize the neural network in Algorithm~\ref{alg:kl-pg} using a pair of coupled normalizing flows composed of neural spline layers \citep{durkan2019neural,papamakarios2021normalizing}. The resulting model represents the one-step transition kernels required by the bi-causal generator while preserving exact log-density evaluation.

For each time step $n \in [N]_0$, we generate a pair of dependent noises $(U_n,V_n)$ through a time- and state-dependent network $\mathfrak{N}_n^{\theta_n}$, and apply invertible transforms $\widehat Y_n=f^{-1}_{\varphi}(U_n;C_n)$ and $\widehat Y'_n=g^{-1}_{\psi}(V_n;C'_n)$. The conditioning variables $C_n$ and $C'_n$ consist of their respective lagged states $(\widehat Y_{n-1}, \widehat Y'_{n-1})$ and a sinusoidal embedding of $n$. The dependence between $\widehat Y_n$ and $\widehat Y'_n$ is induced by a copula on $(U_n,V_n)$ with a time- and state-dependent dependence parameter in $\mathfrak{N}_n^{\theta_n}$ \citep{laszkiewicz2021copula}. Thus, the two conditional flows specify the marginals, whereas the copula captures the cross-sectional dependence. This construction yields exact joint log-likelihood terms and enables efficient evaluation of the KL objective and stochastic optimization in Algorithm~\ref{alg:kl-pg} \citep{durkan2019neural,papamakarios2021normalizing}. In implementation, we set the penalty coefficient $\beta$ to 50. See more details in Appendix~\ref{sec: details of nf implementation}.

\paragraph{Evaluation metrics.}
Let $\{(\pmb Y^{(m)}, \pmb Y'^{(m)})\}_{m=1}^M$ and $\{(\widehat{\pmb Y}^{(m)}, \widehat{\pmb Y}'^{(m)})\}_{m=1}^M$ denote i.i.d.\ samples from the ground-truth distribution and the learned model, respectively.\footnote{Our synthetic experiments use 2000 samples generated from the ground truth and learned coupling for training.}  
Each sample $(\pmb Y^{(m)}, \pmb Y'^{(m)})=\{(Y_{n,i}^{(m)},Y'^{(m)}_{n,i})\}_{0 \le n \le N,\; 1 \le i \le d}$ represents a length-$(N+1)$ trajectory of $d$ assets, and $(\widehat{\pmb Y}^{(m)}, \widehat{\pmb Y}'^{(m)})$ is defined analogously. We set the stage cost in \eqref{eq:bicausal-ot} to be
$c_n(y,y')=\|y-y'\|_2^2 / d$. We evaluate the learned generator from several aspects, including comparisons of marginal distributions, path-space discrepancy metrics between sample trajectories, and temporal-dependence metrics.
See Appendix \ref{sec:gen_set_metrics} for the definition of each evaluation metric.
\begin{itemize}
    \item \textbf{Transport cost}, $\mathtt{Cost\_avg}$, which quantifies the objective value attained by the relaxed formulation~\eqref{eq:relaxed-bicausal-ot}.
    
   \item \textbf{Wasserstein--2 distance for marginal distributions}, $\mathtt{W2\_avg}$, which quantifies discrepancies between one-step marginals across time and assets.

    \item \textbf{Kolmogorov--Smirnov distance for marginal distributions}, $\mathtt{KS\_avg}$, which provides a complementary and distribution-free measure of marginal mismatch.

    \item \textbf{Sliced Wasserstein distance} \citep{Bonneel2015}, $\mathtt{SWD}$, which captures discrepancies in the joint distribution of the stacked trajectories through one-dimensional random projections, and is approximated numerically via Monte Carlo sampling over random projection directions.

    \item \textbf{Maximum mean discrepancy} \citep{Gretton2012}, $\mathtt{MMD}^2$, which captures discrepancies in the full joint distribution of the stacked trajectories.

    \item \textbf{Adjacent-step correlation discrepancy}, $\mathtt{T}_{\max}$ and $\mathtt{T}_2$ and their corresponding $p$-values computed using a trajectory-level permutation test \citep{Good2005,WestfallYoung1993}.
\end{itemize}

\subsection{Results}

We consider two synthetic settings: a lower-dimensional, shorter-horizon case $(d=2, N=3)$, and a more challenging higher-dimensional, longer-horizon case \((d=3, N=10)\). We present the latter in the main text and defer the former to Appendix~\ref{sec: additional results of synthetic experiments}.

The generated marginals closely align with the true distributions across all time steps.
Figures~\ref{fig:unimodal_d3_N10_distributions} and \ref{fig:bimodal_d3_N10_distributions} compare the empirical marginal distributions at times $n\in\{0,3,7,10\}$ obtained from the normalizing-flow generator with those from direct sampling of the ground-truth model. Our method accurately captures both Gaussian marginals and non-Gaussian features induced by the shifted and scaled Beta distributions.

\begin{figure}[!ht]
    \centering
    \caption{Marginal distributions across different timesteps for the unimodal setting ($d=3$, $N=10$).\label{fig:unimodal_d3_N10_distributions}}
    
    \subfigure[$n=0$]{
        \includegraphics[width=0.48\textwidth]{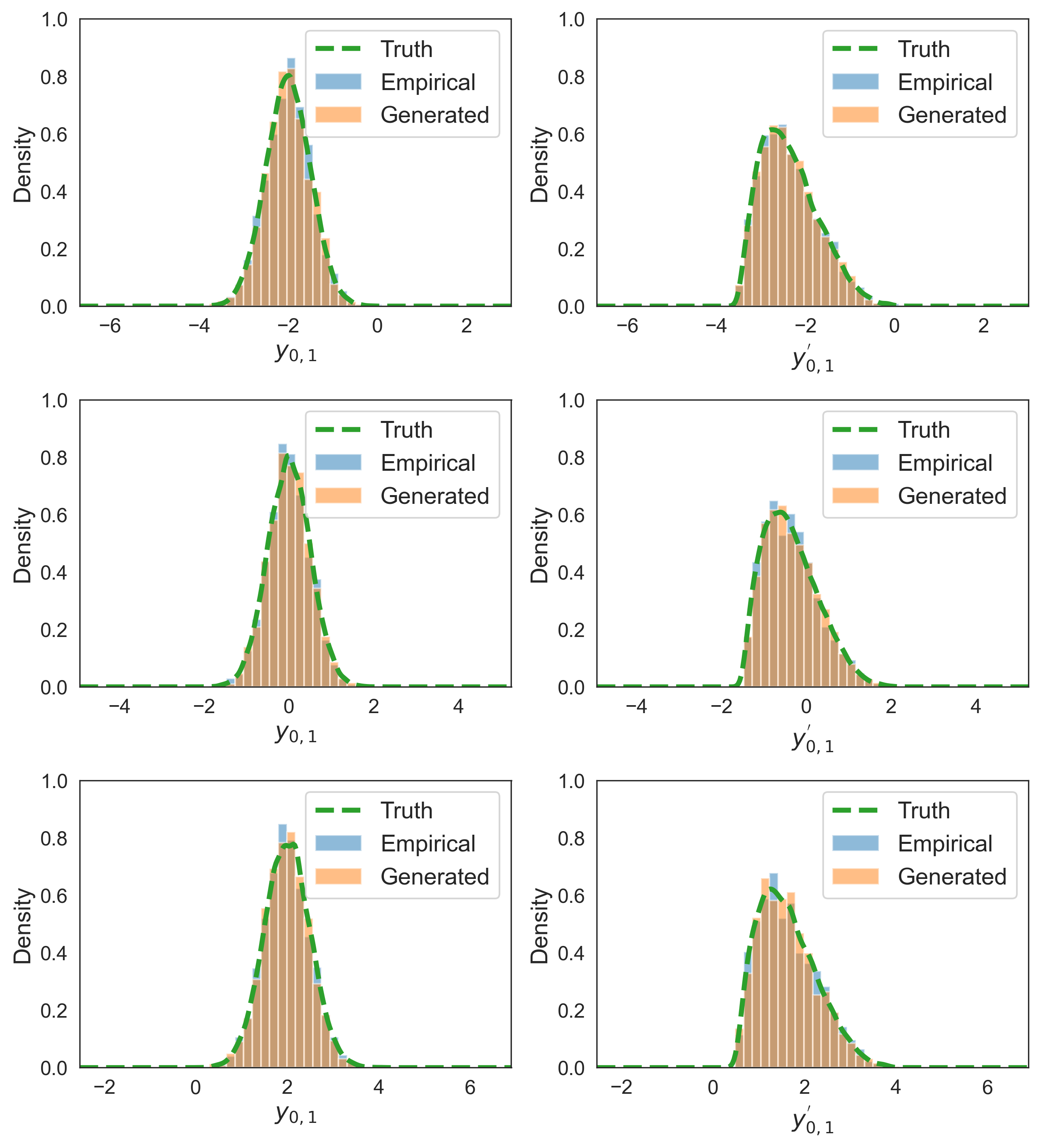}
    }
    \hfill
    \subfigure[$n=3$]{
        \includegraphics[width=0.48\textwidth]{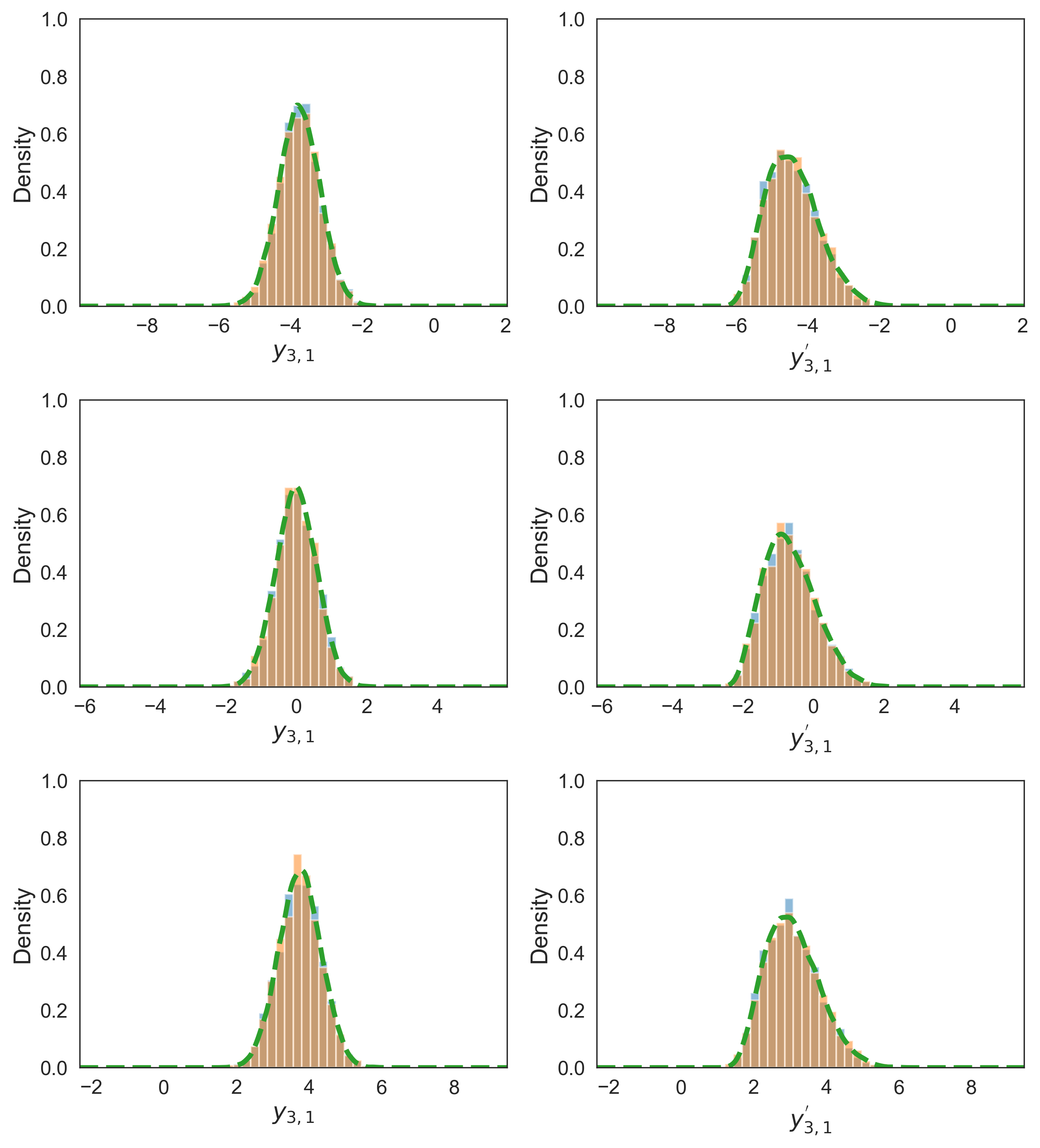}
    }
    
    \subfigure[$n=7$]{
        \includegraphics[width=0.48\textwidth]{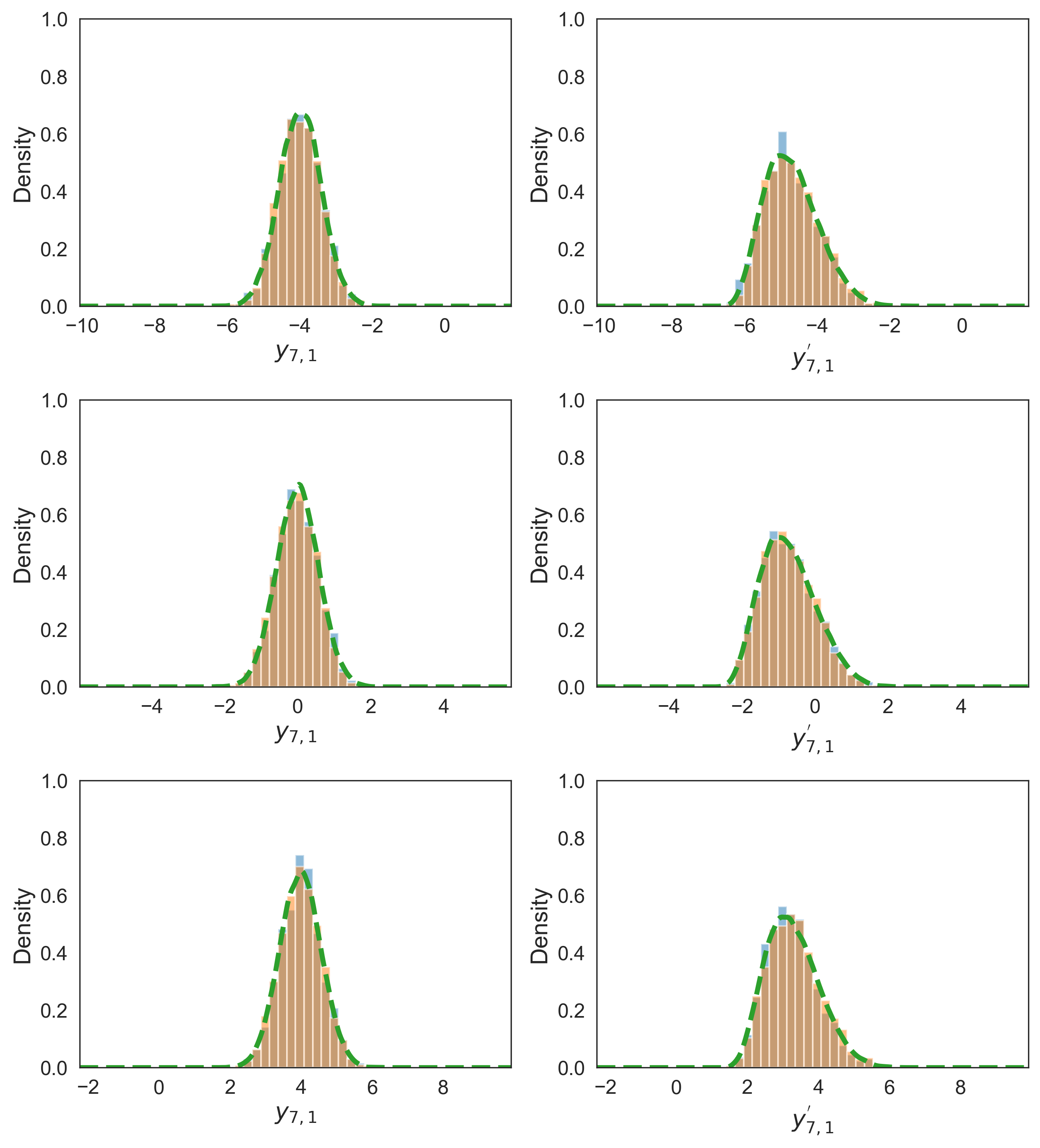}
    }
    \hfill
    \subfigure[$n=10$]{
        \includegraphics[width=0.48\textwidth]{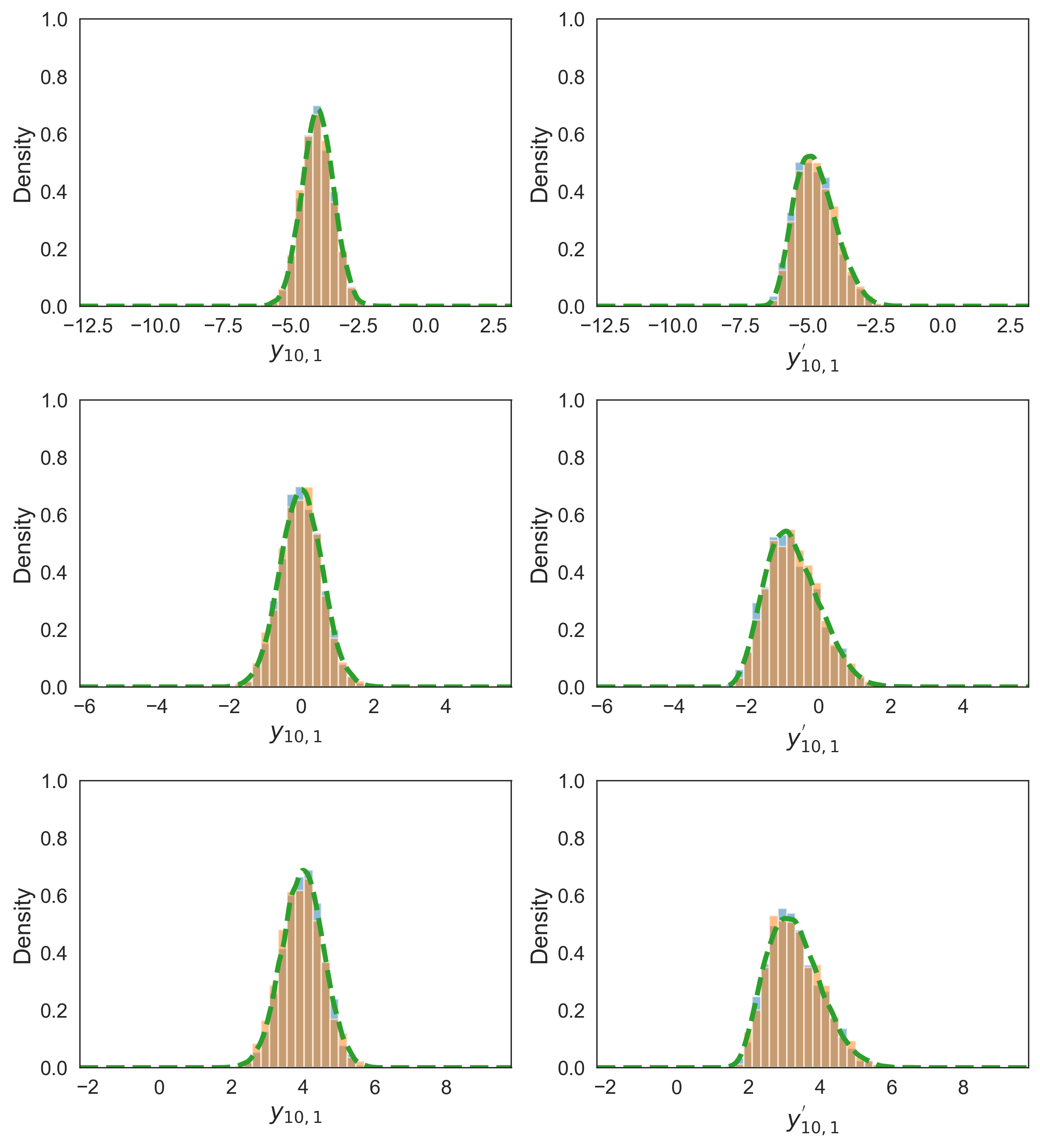}
    }
\end{figure}

\begin{figure}[!ht]
    \centering
    \caption{Marginal distributions across different timesteps for the bimodal setting ($d=3$, $N=10$).\label{fig:bimodal_d3_N10_distributions}}
    
    \subfigure[$n=0$]{
        \includegraphics[width=0.48\textwidth]{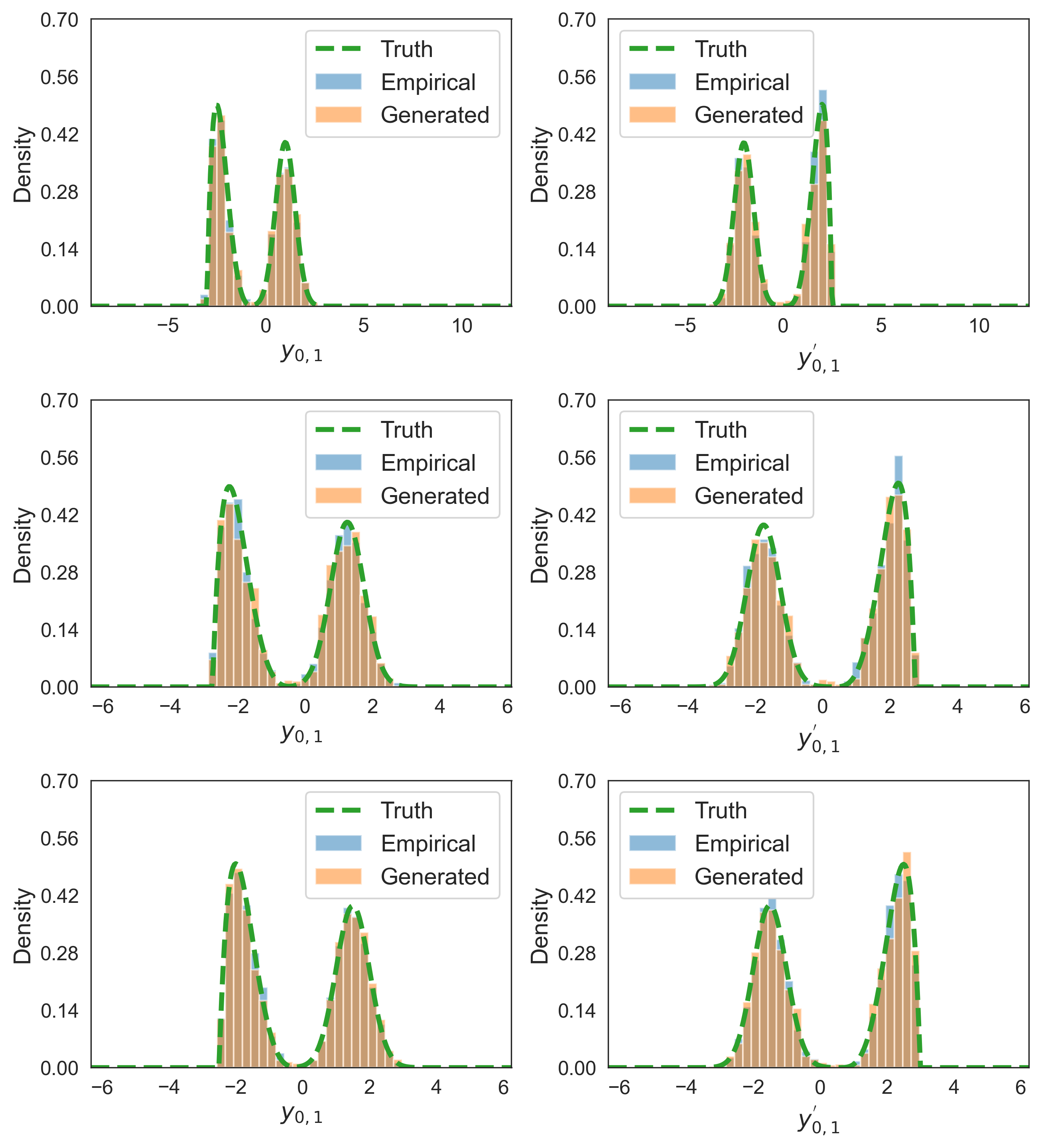}
    }
    \hfill
    \subfigure[$n=3$]{
        \includegraphics[width=0.48\textwidth]{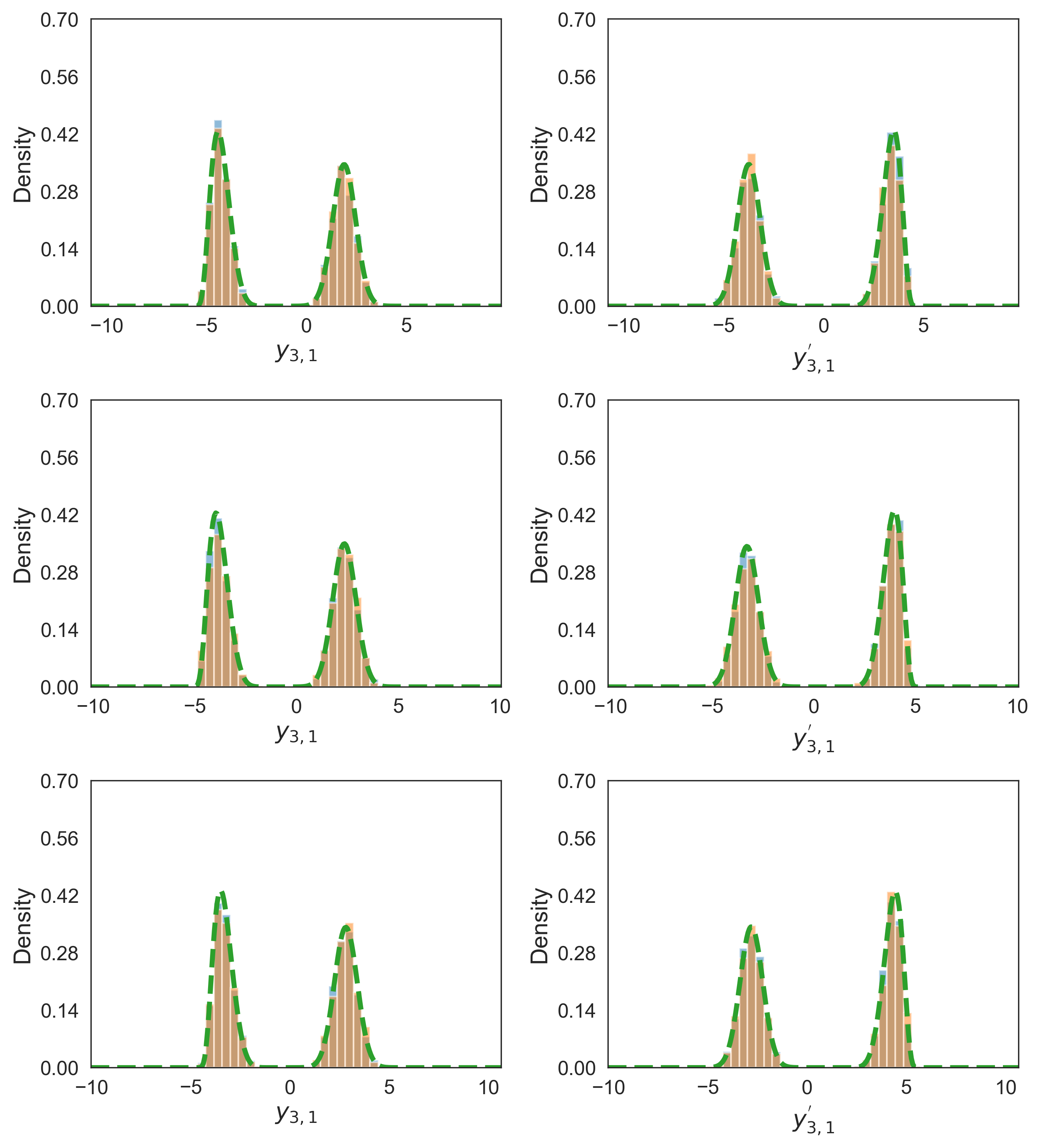}
    }
    
    \subfigure[$n=7$]{
        \includegraphics[width=0.48\textwidth]{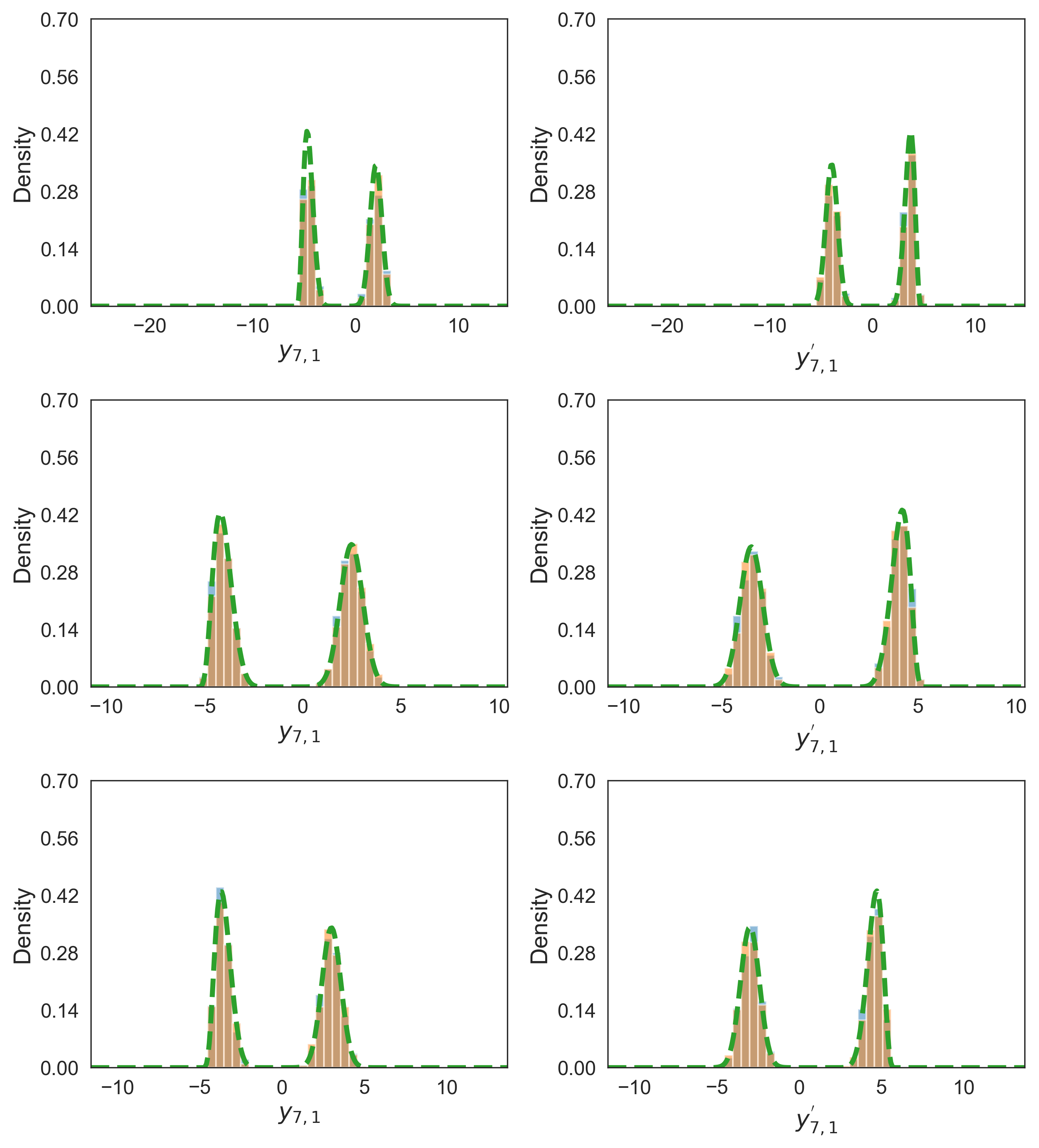}
    }
    \hfill
    \subfigure[$n=10$]{
        \includegraphics[width=0.48\textwidth]{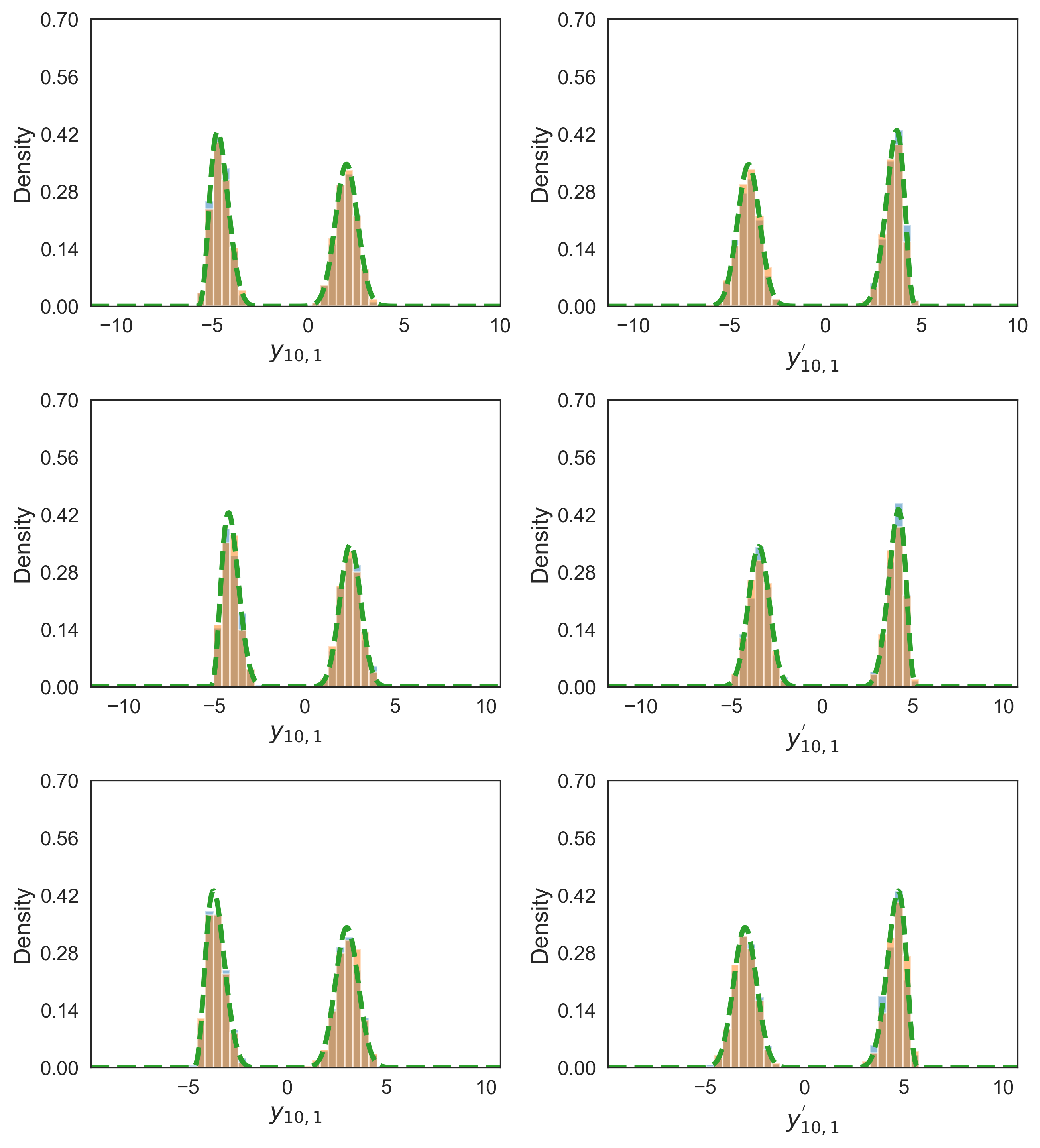}
    }
\end{figure}

In addition, the learned generator reproduces adjacent-step correlations accurately across all time steps, as shown in Figure~\ref{fig:simulation_d3_N10_corr}. The absolute discrepancy between the generated and true adjacent-step correlations remains uniformly small. This demonstrates that the proposed method robustly preserves temporal dependence as both the dimensionality and the time horizon increase.

\begin{figure}[!ht]
    \centering
    \caption{Temporal correlations between adjacent time steps ($d=3$, $N=10$).\label{fig:simulation_d3_N10_corr}}
    
    \subfigure[Unimodal.]{
        \includegraphics[width=0.48\linewidth]{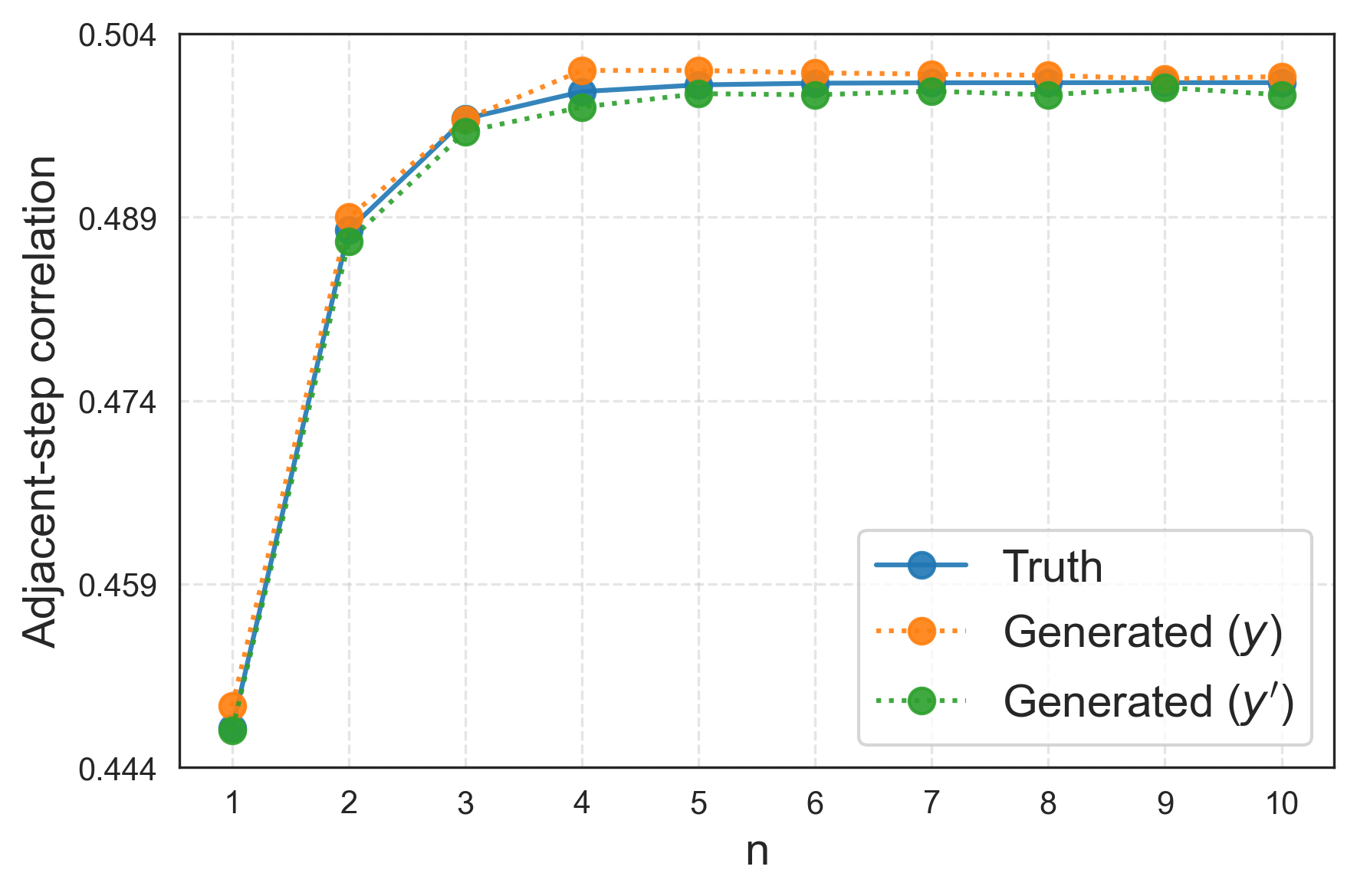}
    }
    \hfill
    \subfigure[Bimodal.]{
        \includegraphics[width=0.48\linewidth]{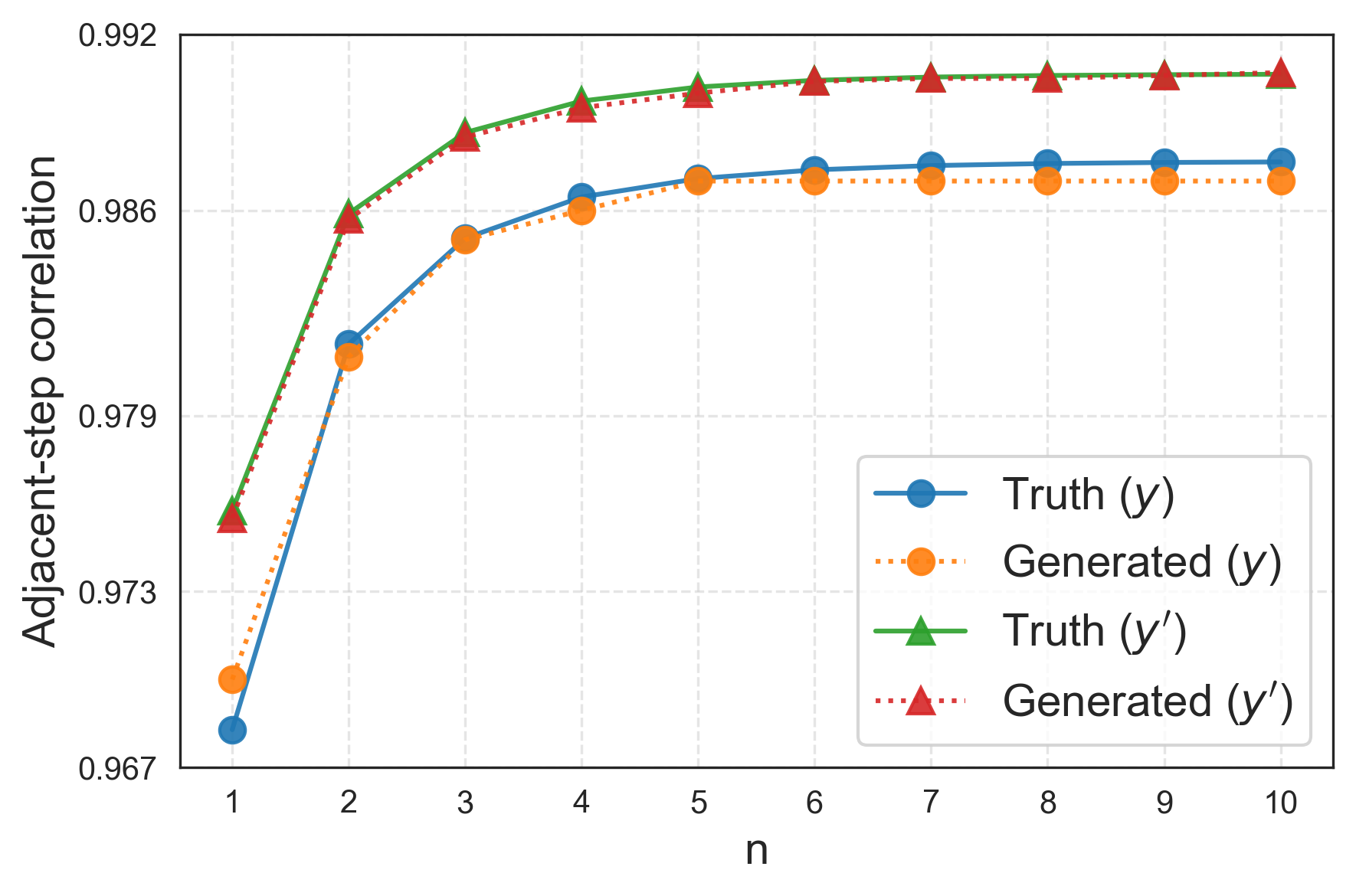}
    }
\end{figure}

Finally, we run 10 repeated experiments and summarize distributional discrepancies and the adjacent-correlation tests in Table~\ref{tab:d3_N10_metrics}. 
The learned model closely matches the ground-truth distribution in both unimodal and bimodal settings. 
In particular, the marginal discrepancy metrics ($\mathtt{W2\_avg}$ and $\mathtt{KS\_avg}$), trajectory-level metrics ($\mathtt{SWD}$ and $\mathtt{MMD}^2$), and the learned average cost ($\mathtt{Cost\_avg}$) all remain uniformly moderate, indicating accurate recovery of both single-step marginals and the full joint path distribution. Moreover, the adjacent-correlation discrepancies are small, with the permutation tests yielding large $p$-values (unimodal: $\mathtt{p(T_{\max})}=0.9590$, $\mathtt{p(T_2)}=0.9732$; bimodal: $\mathtt{p(T_{\max})}=0.5686$, $\mathtt{p(T_2)}=0.3030$). Thus, there is no statistically significant evidence that the generated and ground-truth adjacent-step correlations differ for $n \in [N]$. Overall, these results demonstrate that the proposed method captures not only non-Gaussian distributional features but also the underlying temporal dependence structure.

\begingroup
\renewcommand{\arraystretch}{0.9}
\setlength{\extrarowheight}{0pt}
\begin{table}[!ht]
\centering
\small
\caption{
Performance under distributional metrics and adjacent-correlation tests ($d=3$, $N=10$). Mean values are reported with standard deviations in parentheses.}
\label{tab:d3_N10_metrics}
{
\begin{tabular*}{\textwidth}{@{\extracolsep{\fill}}l *{9}{c}@{}}
\toprule
& \multicolumn{5}{c}{\textbf{Distributional Metrics}} & \multicolumn{4}{c}{\textbf{Adjacent-Correlation Tests}} \\
\cmidrule(lr){2-6} \cmidrule(lr){7-10}
\textbf{Setting} & $\mathtt{Cost\_avg}$ & $\mathtt{W2\_avg}$ & $\mathtt{KS\_avg}$ & $\mathtt{SWD}$ & $\mathtt{MMD}^2$ & $\mathtt{T_{\max}}$ & $\mathtt{p(T_{\max})}$ & $\mathtt{T_{2}}$ & $\mathtt{p(T_{2})}$ \\
\midrule
\textbf{Unimodal} & 0.3882 & 0.0011 & 0.0099 & 0.0042 & 0.0058 & 0.0118 & 0.9590 & 0.0004 & 0.9732 \\
& (0.0376) & (0.0001) & (0.0003) & (0.0004) & (0.0003) & (0.0023) & (0.0377) & (0.0001) & (0.0245) \\
\addlinespace
\textbf{Bimodal} & 2.3849 & 0.0070 & 0.0108 & 0.0871 & 0.0240 & 0.0016 & 0.5686 & 0.0000 & 0.3030 \\
& (0.1035) & (0.0035) & (0.0006) & (0.0542) & (0.0016) & (0.0004) & (0.1799) & (0.0000) & (0.1524) \\
\bottomrule
\end{tabular*}
}
{}
\end{table}
\endgroup

\section{Applications}
\label{sec:applications}
We now illustrate how the proposed framework applies to two representative problems in which respecting temporal information structure is essential: robust hedging and time series statistical downscaling. While these applications arise in different domains, they share a common feature: the objective depends on pathwise interactions between stochastic processes under non-anticipative constraints. This makes them natural settings for bi-causal OT, and highlights the advantage of our approach in learning couplings that are both distributionally accurate and dynamically consistent.

\subsection{Robust hedging}
\label{sec: robust hedging}
We apply Algorithm~\ref{alg:kl-pg} to learn robust subhedging for financial derivatives under the adaptiveness constraint to the natural information filtration. We evaluate the learned coupling in terms of (i) marginal fidelity over time, (ii) temporal dependence, quantified by adjacent-step correlations, and (iii) pricing accuracy, measured by the learned subhedging value relative to the ground truth.

\subsubsection{Framework}
We first describe how robust hedging can be formulated as a bi-causal OT problem. Consider the dynamics of a financial market over discrete timestamps $n \in [N]_0$ with $2d$ tradable assets, and denote the price vector by $(Y_n, Y'_n)\in \mathbb{R}^{d} \times \mathbb{R}^{d}$.
\footnote{This two-group structure arises naturally in practice. For instance, they may represent a market index and a tradable ETF that tracks it, latent asset-pricing factors and the portfolios constructed to replicate them, or a benchmark portfolio and a hedging portfolio whose tracking error the investor seeks to minimize.}
Trading strategies are required to be predictable with respect to the natural filtration
\[
\mathcal{F}_n := \sigma\big(Y_0,Y'_0,\dots,Y_n,Y'_n\big),\qquad n \in [N]_0,
\]
so that portfolio positions at time $n$ depend only on the observed history up to time $n$.

For a target payoff/criterion $\xi$ on asset prices $\pmb Y=(Y_0,\dots,Y_N)$ and $\pmb Y'=(Y'_0,\dots,Y'_N)$, the robust subhedging price is defined as the maximal initial capital whose terminal wealth can be guaranteed \emph{not to exceed} the target payoff $\xi$ under all admissible price dynamics:
\begin{align}
\label{eq:subhedge_primal_robust}
\underline p(\xi)
:= \sup\biggl\{ x \in \mathbb{R} \;:\;
&\exists\, H=(H_n)_{n=0}^{N-1},\ H_n\in\mathbb{R}^{2d},\ \text{with } H_n \text{ being } \mathcal{F}_n\text{-adapted, s.t. } \\
&x + \sum_{n=1}^N H_{n-1}\cdot (Y_n-Y_{n-1}, Y'_n-Y'_{n-1})
\ \le\ \xi(\pmb Y,\pmb Y') \, ,
\quad \Pi_{\mathrm{bc}}\textrm{-quasi-surely}
\biggr\}, \nonumber
\end{align}
where $H_n$ represents a trading strategy at step $n$, and $\Pi_{\mathrm{bc}}$ denotes the set of admissible bi-causal laws on path space satisfying prescribed marginal constraints $(\mu,\mu')$ as well as the bi-causal information structure.
\footnote{Here, “$\Pi_{\rm bc}$-quasi-surely” means  $\pi$-almost surely for every $\pi\in \Pi_{\rm bc}$.}
In incomplete markets, $\underline p(\xi)$ corresponds to a \emph{robust lower (bid) bound} for $\xi$, complementing the robust upper (ask) bound from superhedging and forming a no-arbitrage price interval \citep{beiglbock2013model}.

Under mild conditions, the problem \eqref{eq:subhedge_primal_robust} admits a dual adapted-transport representation \citep[Section 3.1]{krsek2025duality}: 
\begin{equation}
\label{eq:robust_duality_bc}
\underline p(\xi)
\;=\;
\inf_{\pi \in \Pi_{\mathrm{bc}}(\mu,\mu')}
\E_{(\pmb Y,\pmb Y')\sim\pi}\left[ \xi(\pmb Y,\pmb Y') \right].
\end{equation} 
Thus, $\underline p(\xi)$ is equivalent to the smallest attainable expected value of $\xi$ over all admissible bi-causal models consistent with $(\mu,\mu')$ and the information constraint. To connect \eqref{eq:robust_duality_bc} with bi-causal OT, consider an additive payoff $\xi$ specified by a sequence of nonnegative one-stage costs $\{c_n\}_{n=0}^N$ \citep[Chapter~6]{bertsekas2012dpoc2}:
\begin{equation}
\label{eq:connect_xi_cn}
\xi(\pmb y,\pmb y')
=
\sum_{n=0}^{N} c_n(y_n,y'_n)
\quad \textrm{ and } \quad
\underline p(\xi)
\;=\;
\inf_{\pi \in \Pi_{\mathrm{bc}}(\mu,\mu')}
\E_{(\pmb Y,\pmb Y')\sim\pi}\left[ \sum_{n=0}^{N} c_n(Y_n, Y'_n) \right].
\end{equation}
Under this specification, $\underline p(\xi)$ can be interpreted as a robust bid valuation for the cumulative cost (tracking error) incurred along the path \citep{duffie1991mean}, under all bi-causal models consistent with $(\mu,\mu')$. The objective on the right-hand side of $\underline p(\xi)$ in \eqref{eq:connect_xi_cn} coincides with a bi-causal OT problem \eqref{eq:bicausal-ot} with stage costs $\{c_n\}$, so learning the optimal coupling $\pi\in\Pi_{\mathrm{bc}}(\mu,\mu')$ directly corresponds to learning the robust subhedging value. Accordingly, Algorithm~\ref{alg:kl-pg} can be viewed as a numerical method for computing this robust bid bound under the imposed information constraint. In addition,  an optimal hedging strategy 
$H_n$ can be recovered from the value function 
$V_n$, for instance, via subgradients or martingale representation.

\subsubsection{Numerical experiment}
We evaluate Algorithm~\ref{alg:kl-pg} on a widely-used robust pricing example that admits a closed-form optimal value \citep{givens1984wasserstein, bayraktar2025fitted, han2025kalman}, which serves as a benchmark for comparison.

\paragraph{Setup.}
We consider a market with two one-dimensional price processes ($d=1$) jointly observed over a horizon of $N=5$ time steps. Following \citet[Sections~5.1--5.2]{bayraktar2025fitted}, we take the data-generating process to be martingale random walks with different initial values and price volatilities: 
\begin{equation*}
\begin{split}
& Y_{n}=Y_{n-1}+\varepsilon_{n},\qquad\text{where}\quad Y_{0}=1.0 \ \text{and}\ \varepsilon_{n}\sim\mathcal{N}(0,1),\quad n\in[N]; \\
& Y'_{n}=Y'_{n-1}+\varepsilon'_{n},\qquad\text{where}\quad Y'_{0}=2.0 \ \text{and}\ \varepsilon'_{n}\sim\mathcal{N}(0,0.5^2), \quad n\in[N].
\end{split}
\end{equation*}
The adjacent-step correlations for both processes admit a closed form
\begin{equation*}
\mathrm{Corr}(Y_{n},Y_{n-1})
=\mathrm{Corr}(Y'_{n},Y'_{n-1})=\rho_n:=\sqrt{\frac{n-1}{n}},\qquad n \in [N].
\end{equation*}
In particular, the correlation $\rho_0$ is defined as 0.

We set the stage cost in \eqref{eq:bicausal-ot} to be the quadratic tracking loss $c_n(y,y')=(y-y')^2$ for all $n \in [N]$ and $c_0 = 0$, so that the target functional \eqref{eq:connect_xi_cn} becomes the cumulative quadratic tracking error, which is commonly used for residual tracking risk between a target exposure and a traded proxy \citep{duffie1991mean, schweizer2010mvh}:
\[
\xi(\pmb y, \pmb y')=\sum_{n=1}^N (y_n-y'_n)^2.
\]

We evaluate performance using the distributional and temporal-dependence metrics introduced in Section \ref{subsec:synthetic_exp_setup} (definitions in Appendix \ref{sec:gen_set_metrics}). In addition, we report two pricing metrics (definitions in Appendix \ref{sec:eval_robust_hedging}):
\footnote{Without further specification, we set the KL-divergence penalty coefficient $\beta$ to be $50$.}
\begin{itemize}
    \item \textbf{Empirical subhedging price}, $\widehat{p(\xi)}$, under the learned coupling.
    \item \textbf{Relative error}, $\mathtt{RE}$, with respect to the closed-form optimal value $\underline p(\xi)$ \citep[Lemma~2]{han2025kalman}.
\end{itemize}

\paragraph{Main findings.}
The generated marginals closely track the ground truth.
Figure~\ref{fig:martingale_d1_N5_distributions} compares the empirical marginal distributions from samples generated by the normalizing flow with direct samples from the ground truth, at time steps $n\in\{1,2,4,5\}$. The generated marginals closely track the ground truth, indicating that our method preserves the Gaussian random-walk marginals.

\begin{figure}[!ht]
    \centering
    \caption{Marginal distributions across time for the martingale setting ($d=1$, $N=5$).\label{fig:martingale_d1_N5_distributions}}
    
    \subfigure[$n=1$]{
        \includegraphics[width=0.48\textwidth]{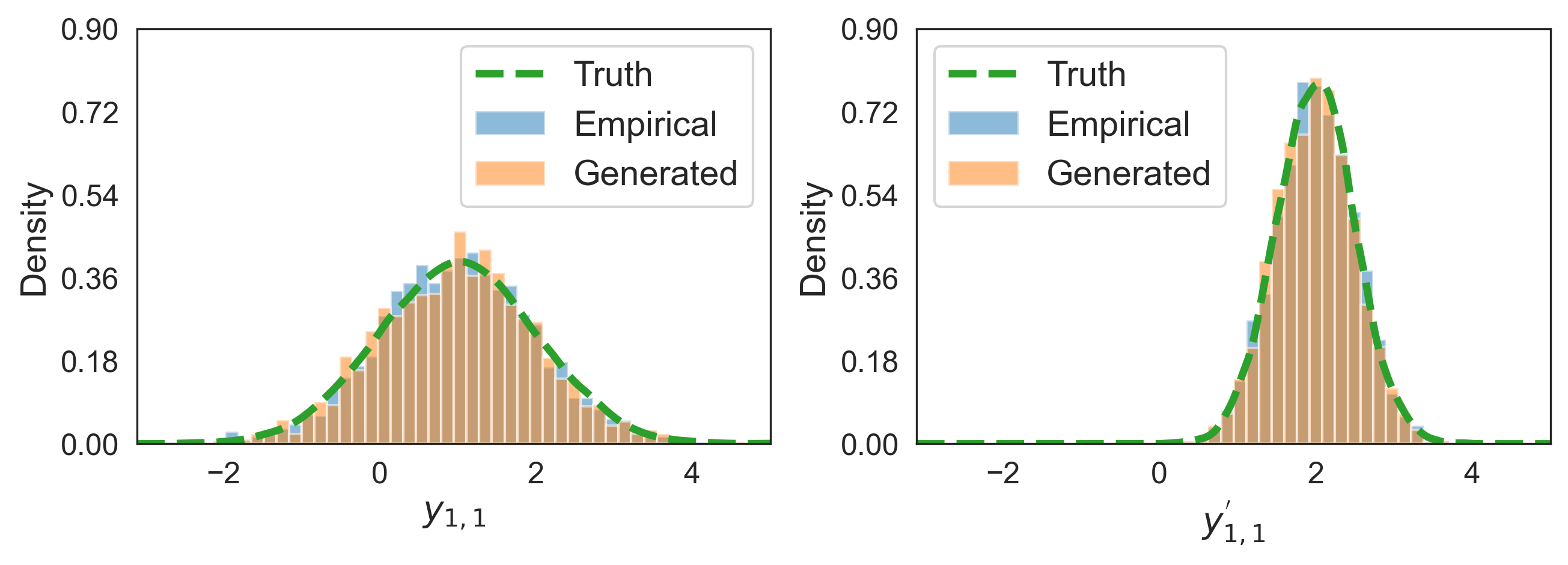}
    }
    \hfill
    \subfigure[$n=2$]{
        \includegraphics[width=0.48\textwidth]{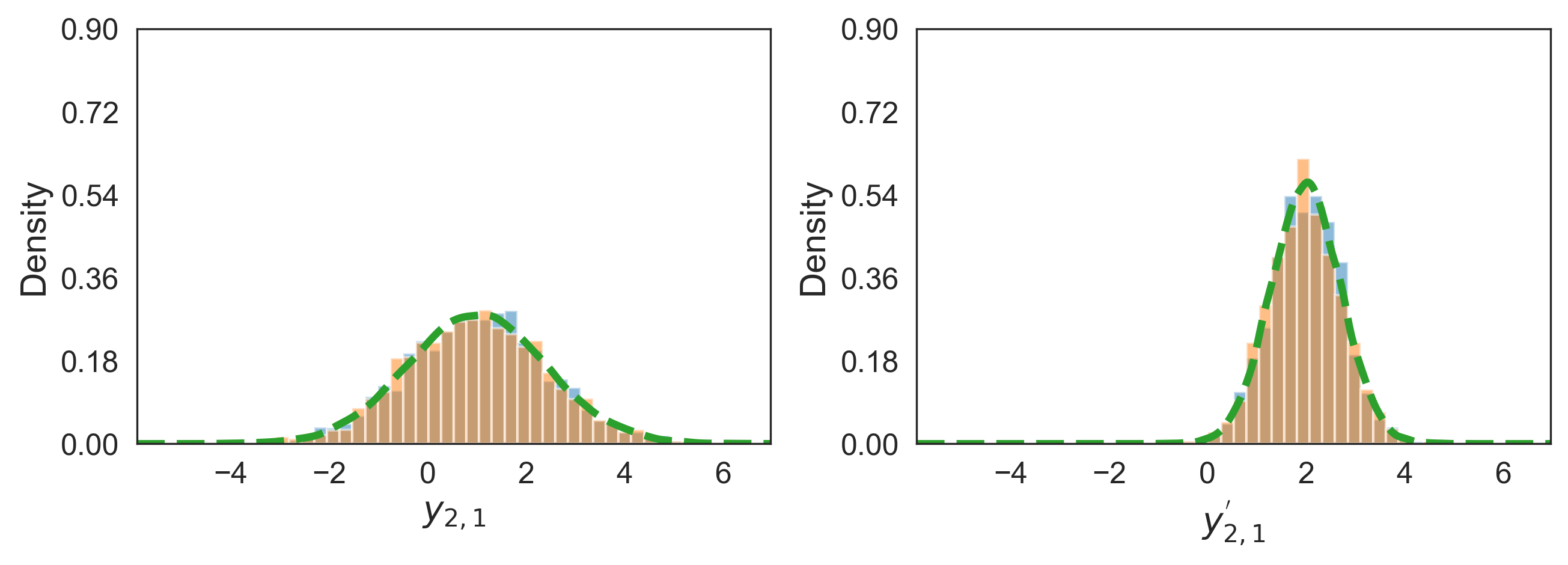}
    }
    
    \subfigure[$n=4$]{
        \includegraphics[width=0.48\textwidth]{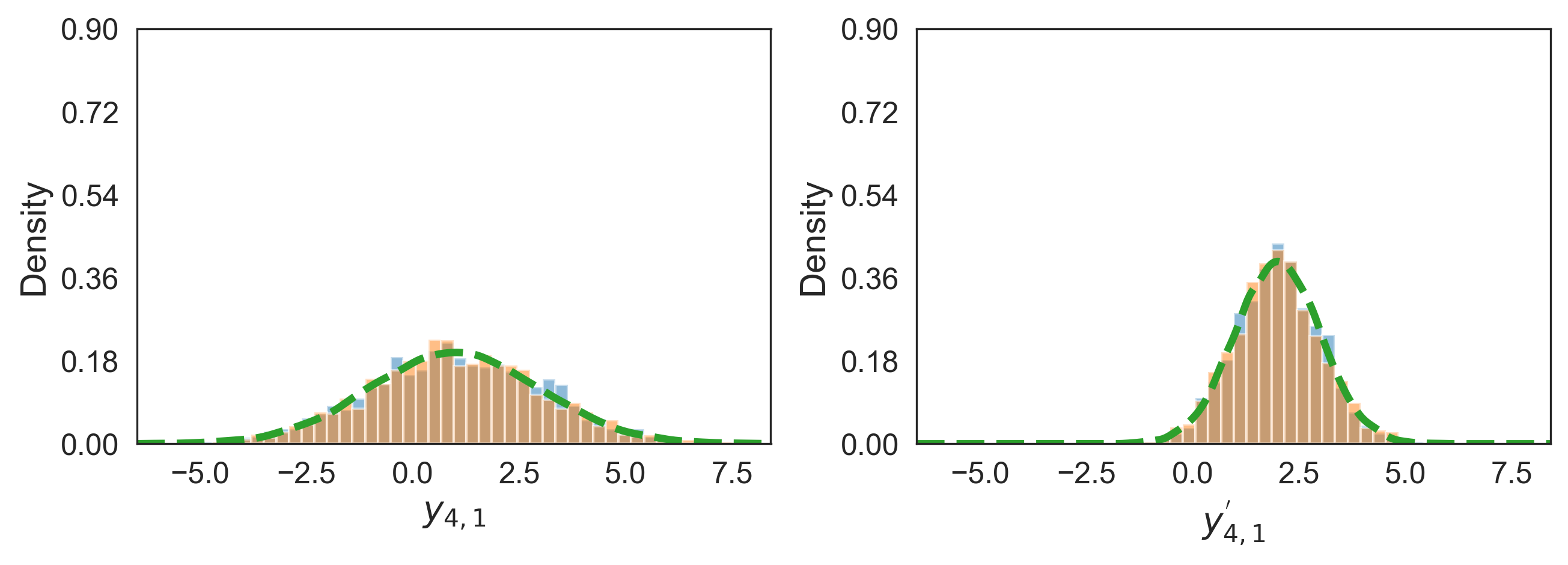}
    }
    \hfill
    \subfigure[$n=5$]{
        \includegraphics[width=0.48\textwidth]{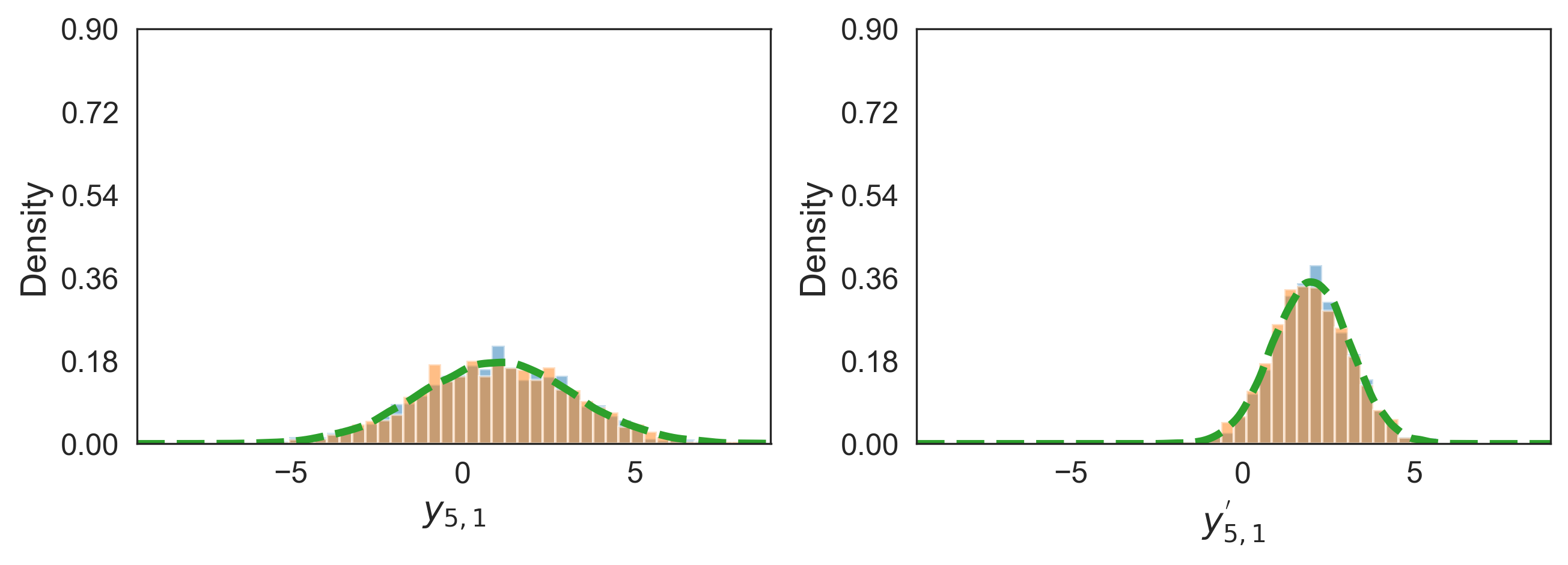}
    }
\end{figure}

The estimated temporal correlations closely match the ground truth over the entire horizon. Figure~\ref{fig:martingale_d1_N5_corr} shows that the learned coupling reproduces serial dependence very well. In particular, the absolute discrepancy between the generated and true adjacent-step correlations remains small, on the order of $10^{-3}$.

\begin{figure}[!ht]
    \centering
    \caption{Adjacent-step correlations over $n=1,\dots,N$ ($d=1$, $N=5$).\label{fig:martingale_d1_N5_corr}}
    \includegraphics[width=0.5\linewidth]{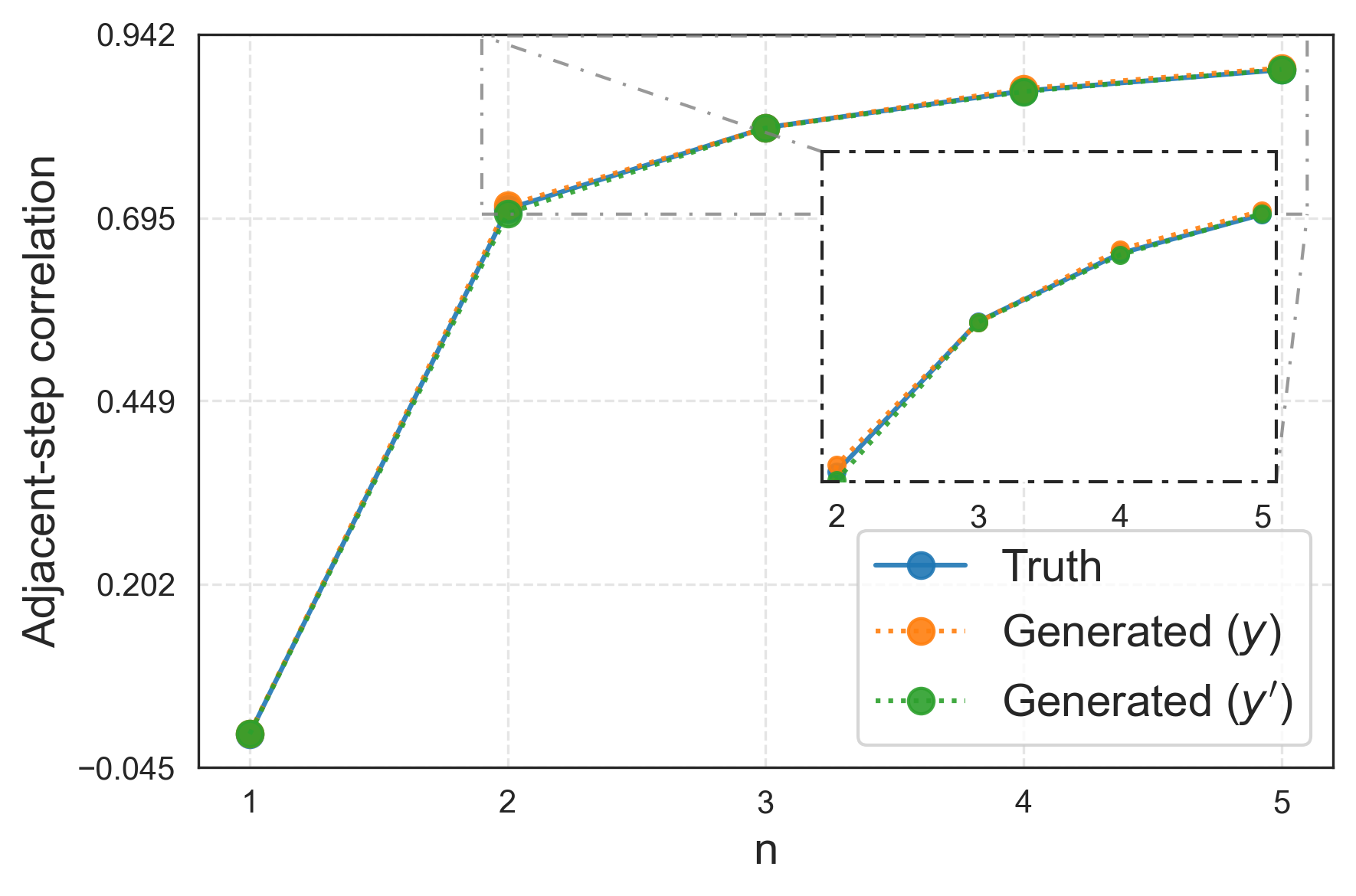}
\end{figure}

Furthermore, we report evaluation metrics over 10 independent runs in Table~\ref{tab:martingale_metrics}. Both the marginal discrepancy measures ($\mathtt{W2\_avg}$ and $\mathtt{KS\_avg}$) and trajectory-level metrics ($\mathtt{SWD}$ and $\mathtt{MMD}^2$) are uniformly small, demonstrating the accurate recovery of the one-step marginals as well as the joint law of the full trajectory. Moreover, the correlation-based permutation tests yield large $p$-values ($\mathtt{p(T_{\max})}=0.9768$, $\mathtt{p(T_2)}=0.9739$), so there is no statistically significant evidence that the generated and ground-truth adjacent-step correlations differ for \(n \in [N]\). Overall, the method preserves the martingale dynamics while maintaining a close fit to the ground truth.

Finally, it is worth noting that our implementation of Algorithm~\ref{alg:kl-pg} runs faster than the method reported by \citet{bayraktar2025fitted}.
\footnote{Our JAX implementation of Algorithm~\ref{alg:kl-pg} requires 31.60 seconds of end-to-end wall-clock time for a 50-step run, including one-time JAX just-in-time (JIT) compilation overhead, data sampling, and training. After compilation, each subsequent 50-step run takes about 1 second. In terms of per-run wall-clock time, this is faster than the 35.56 seconds reported by \citet{bayraktar2025fitted} using the method therein.}

\begingroup
\renewcommand{\arraystretch}{0.9}
\setlength{\extrarowheight}{0pt}
\begin{table}[!ht]
\centering
\small
\caption{
Performance metrics with mean values and standard deviations (in parentheses).
\label{tab:martingale_metrics}}
{
\begin{tabular*}{\textwidth}{@{\extracolsep{\fill}}*{10}{c}@{}}
\toprule
\multicolumn{4}{c}{\textbf{Distributional Metrics}} & \multicolumn{4}{c}{\textbf{Adjacent-Correlation Tests}} & \multicolumn{2}{c}{\textbf{Pricing}} \\
\cmidrule(lr){1-4} \cmidrule(lr){5-8} \cmidrule(lr){9-10}
$\mathtt{W2\_avg}$ & $\mathtt{KS\_avg}$ & $\mathtt{SWD}$ & $\mathtt{MMD}^2$ & $\mathtt{T_{\max}}$ & $\mathtt{p(T_{\max})}$ & $\mathtt{T_{2}}$ & $\mathtt{p(T_{2})}$ & $\widehat{p(\xi)}$ & $\mathtt{RE}(\%)$ \\
\midrule
0.0044 & 0.0198 & 0.1022 & 0.0184 & 0.0038 & 0.9768 & 0.0001 & 0.9739 & 8.7545 & 0.0976 \\
(0.0004) & (0.0009) & (0.0113) & (0.0025) & (0.0010) & (0.0168) & (0.0001) & (0.0239) & (0.1059) & (0.1077) \\
\bottomrule
\end{tabular*}
}
{}
\end{table}
\endgroup

\paragraph{Convergence with respect to $\beta$.}

As the penalty on marginal constraint violations increases, we show that the estimated subhedging price converges to the optimal value, which is $\underline p(\xi)=8.75$ under our parameter setting \citep[Lemma~2]{han2025kalman}. Table~\ref{tab:beta_costs} reports estimates $\widehat{p(\xi)}$ and the corresponding relative errors for several choices of $\beta$. Figure~\ref{fig:relerr_plot} plots $\mathtt{RE}$ over a uniform grid $\beta\in[10,1000]$, with markers indicating the values used in Table~\ref{tab:beta_costs}. As $\beta$ increases, $\widehat{p(\xi)}$ approaches $\underline p(\xi)$, and $\mathtt{RE}$ falls below $1\%$ for $\beta\ge 50$.

\begingroup
\renewcommand{\arraystretch}{0.9}
\setlength{\extrarowheight}{0pt}
\begin{table}[!ht]
\centering
\small
\caption{
Estimated subhedging price $\widehat{p(\xi)}$ and relative error under different $\beta$ (standard deviations in parentheses).
\label{tab:beta_costs}}
{
\begin{tabular*}{\textwidth}{@{\extracolsep{\fill}}l *{7}{c}@{}}
\toprule
$\beta$ & 10 & 50 & 100 & 500 & 1000 & 5000 & 10000 \\
\midrule
$\widehat{p(\xi)}$ & 8.4026 & 8.7545 & 8.7490 & 8.7507 & 8.7506 & 8.749974 & 8.749983 \\
& (0.1136) & (0.1059) & (0.0067) & (0.0033) & (0.0005) & ($<$0.0001) & ($<$0.0001) \\
\addlinespace
$\mathtt{RE}(\%)$ & 3.9700 & 0.0976 & 0.0615 & 0.0321 & 0.0070 & 0.0009 & 0.0003 \\
& (1.2979) & (0.1077) & (0.0410) & (0.0361) & (0.0061) & (0.0005) & (0.0002) \\
\bottomrule
\end{tabular*}
}
{}
\end{table}
\endgroup

\begin{figure}[!ht]
    \centering
    \caption{Relationship between relative error (log scale) and $\beta$ (log scale).}
    \label{fig:relerr_plot}
    \includegraphics[width=0.45\linewidth]{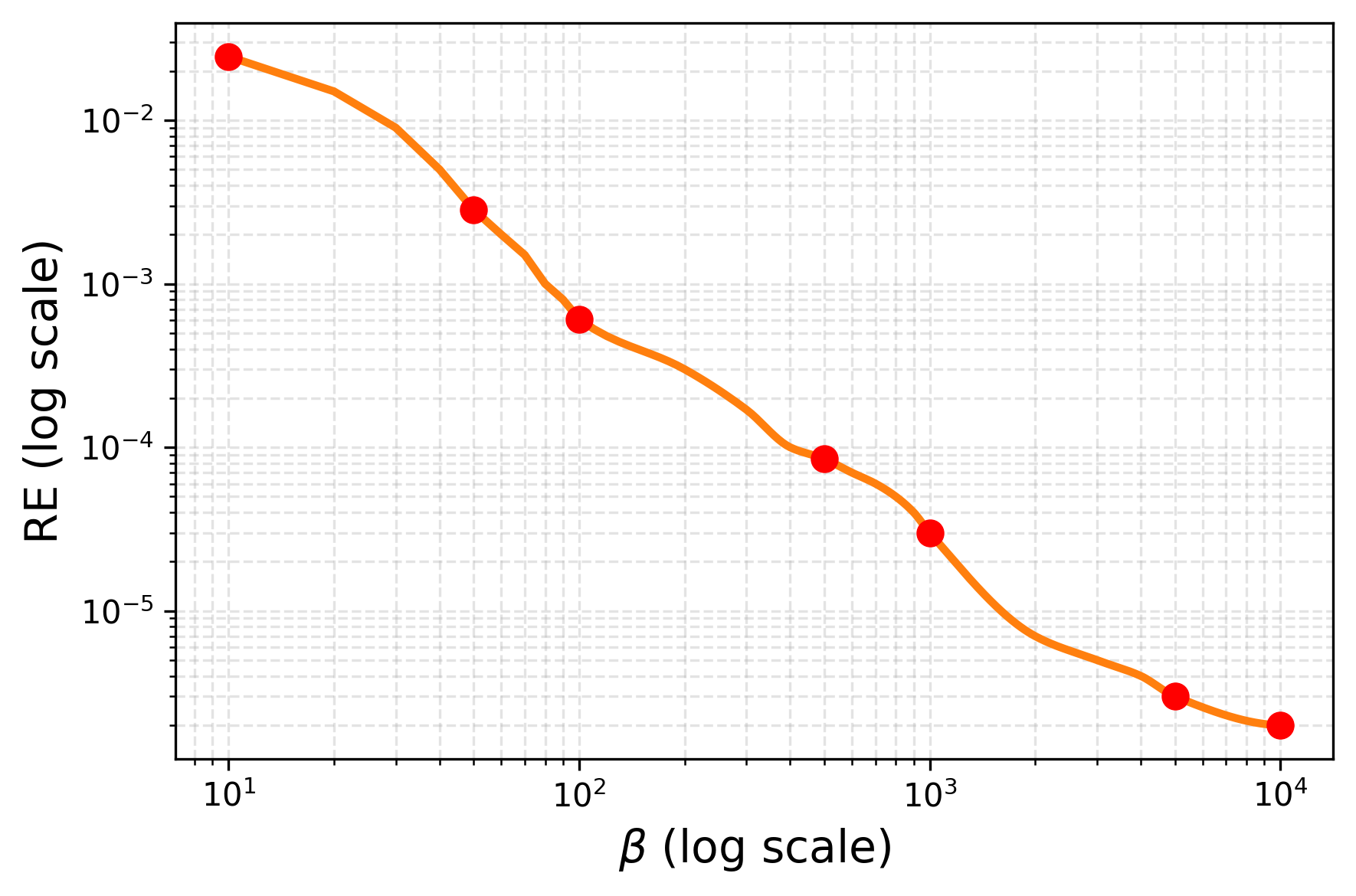}
\end{figure}

\paragraph{Robustness check for multidimensional case.}
We challenge the martingale benchmark by increasing the state dimension to $d\in\{2,3,4,5\}$, or equivalently, $2d\in\{4,6,8,10\}$ jointly observed assets, while keeping the same horizon $N=5$. The quadratic tracking cost is $c_n(y,y')=\|y-y'\|_2^2$, and we focus on pricing accuracy via $\widehat{p(\xi)}$ and $\mathtt{RE}$ relative to the known optimum $\underline p(\xi)$.

Across all dimensions, $\widehat{p(\xi)}$ remains close to $\underline p(\xi)$ with small $\mathtt{RE}$, as shown in Table~\ref{tab:martingale_dim_cost}. Although $\mathtt{RE}$ increases with dimension, the estimates remain reasonably accurate, indicating that the learned coupling continues to deliver reliable pricing approximations in higher-dimensional settings.

\begingroup
\renewcommand{\arraystretch}{0.9}
\setlength{\extrarowheight}{0pt}
\begin{table}[!ht]
\centering
\small
\caption{
Accuracy of robust subhedging prices for varying numbers of assets (standard deviations in parentheses).
\label{tab:martingale_dim_cost}}
{
\begin{tabular*}{\textwidth}{@{\extracolsep{\fill}}cccc@{}}
\toprule
\textbf{Number of assets} & $\underline p(\xi)$ & $\widehat{p(\xi)}$ & $\mathtt{RE}(\%)$ \\
\midrule
4 & 17.5000 & 17.6385 (0.0673) & 0.7913 (0.3846) \\
6 & 26.2500 & 26.5427 (0.1526) & 1.1149 (0.5815) \\
8 & 35.0000 & 36.1372 (0.1815) & 3.2491 (0.5185) \\
10 & 43.7500 & 45.3549 (0.2335) & 3.6685 (0.5337) \\
\bottomrule
\end{tabular*}
}
{}
\end{table}
\endgroup

\subsection{Time series statistical downscaling}
\label{sec: statistical downscaling}
We apply Algorithm~\ref{alg:kl-pg} to a time series statistical downscaling problem motivated by mixed-frequency data commonly encountered in operations research---such as facility-level demand logs versus widely available aggregate forecasts, fine-grained local weather measurements versus coarse climate forecasts, or intraday asset returns versus quarterly accounting disclosures. We evaluate the generated samples in terms of (i) marginal fidelity over time and (ii) temporal dependence, quantified by adjacent-step correlations.

\subsubsection{Framework}

In many applications, high-resolution time series $\pmb X$ (e.g., local logs and measurements) are scarce, while low-resolution data $\pmb Y$ (e.g., aggregate forecasts) are abundant. The goal of statistical downscaling is to generate realistic high-resolution trajectories from low-resolution observations.

To connect these two scales, we introduce a designed projection space $\pmb Y'$, obtained by applying a known downsampling operator $C$ to $\pmb X$. This yields the forward modeling pipeline where $\pmb Y'$ is a convenient but biased low-resolution representation induced by the high-resolution data, and the bi-causal OT step corrects this bias by aligning $\pmb Y'$ with the observed low-resolution distribution $\pmb Y$.

While this pipeline describes how the distributions are related, our ultimate goal is the inverse problem: to generate high-resolution trajectories $\pmb X$ from low-resolution observations $\pmb Y$. The role of the intermediate space $\pmb Y'$ is therefore to provide a structured bridge that links the scarce high-resolution data to the abundant low-resolution data.
See below for an illustration.

\begin{center}
\begin{tikzpicture}[>=stealth]

    \node (X)  at (0,0)   {$\pmb X$};
    \node (Yp) at (3.5,0) {$\pmb Y'$};
    \node (Y)  at (7,0) {$\pmb Y$};

    \draw[->, thick] (X) -- node[above] {$C$} (Yp);
    \draw[->, thick] (Yp) -- node[above] {bi-causal OT} node[below] {debias} (Y);

    \draw[->, thick]
        (Y) -- (7,-1.4)
            -- node[below=2pt] {generation} (0,-1.4)
            -- (X);
\end{tikzpicture}
\vspace{-10pt}
\end{center}

A related two-stage framework was proposed by \citet{wan2024debias}, which first performs debiasing and then conditional generation. However, their framework is designed for non-sequential settings. In our setting, the data are time series, and the debiasing step must preserve temporal dependence and information flow. This makes the debiasing problem naturally a bi-causal transport problem.

Formally, let $\pmb X = (X_0,\dots,X_{N_x}) \sim \mu_X$ denote the high-fidelity high-resolution (HFHR) process, which is the ultimate target of modeling, and let $\pmb Y = (Y_0,\dots,Y_{N_y}) \sim \mu_Y$ denote the observed low-fidelity low-resolution (LFLR) process, where $N_y \ll N_x$. We introduce a {\it known} linear downsampling operator $C:\mathcal X \to \mathcal Y$,
and define the projected high-fidelity low-resolution (HFLR) process
\[
\pmb Y' = (Y_0',\dots,Y_{N_y}') \sim \mu_{Y'} := C_\# \mu_X.
\]

The process $\pmb Y'$ is not intended to match the observed low-resolution data directly. Rather, it is a tractable representation of the high-resolution process on which the debiasing step can be performed. Because the projection map $C$ is imposed by design rather than derived from the true data-generating mechanism, the induced law $\mu_{Y'}$ is generally biased relative to the observed low-resolution law $\mu_Y$.

Accordingly, we propose the following two-step approach, with technical details provided in Appendix~\ref{sec:details_downscaling}:
\begin{enumerate}
    \item \emph{Debiasing via bi-causal OT.}
We learn a bi-causal transport map $T:\mathcal Y \to \mathcal Y$ such that $T_\# \mu_{Y'} = \mu_Y$, while preserving the temporal information structure. This step corrects the bias introduced by the projection space $\pmb Y'$ and aligns the projected low-resolution dynamics with the observed low-resolution data.

\item \emph{High-resolution conditional generation.}
After debiasing, we use the resulting low-resolution representation to generate an HFHR trajectory $\widehat{\pmb X}$ through a conditional generative model.
\end{enumerate}

\subsubsection{Numerical experiment}
We study a regime with limited HFHR data and abundant LFLR data in order to evaluate the practical value of the debiasing step. In such a setting, unconditional high-resolution generation is difficult because the HFHR sample size is small. Direct conditioning on low-resolution observations is also unreliable because these observations are coarse. The experiment is therefore designed to test whether bi-causal OT-based statistical downscaling improves downstream conditional generation.

In the experiment, we compare our method, \textsf{OT+cDfn}, with two baselines. Here, \textsf{OT+cDfn} denotes the OT-adjusted conditional diffusion model, which conditions on the HFLR data obtained from the bi-causal OT step:
\begin{itemize}
    \item \textsf{uDfn}: unconditional diffusion without using low-resolution conditioning.
    \item \textsf{cDfn}: conditional diffusion, which conditions directly on the observed LFLR input.
\end{itemize}

\paragraph{Setup.} We consider $d=2$ processes and model the HFHR data by the Gaussian AR(1) system in \eqref{eq:ar1_simulation}--\eqref{eq:ar1_unimodal}, over a horizon of length $N_x+1=300$. From each HFHR trajectory, we construct two different low-resolution objects. First, we subsample every 20 time steps, starting from the first observation, to obtain an LFLR sequence of length $N_y+1=15$. Second, we average over the same 20-step blocks to obtain the corresponding HFLR sequence.

We use 300 HFHR samples for training, which induce 300 HFLR samples, together with 6000 LFLR samples. At test time, we take 6000 LFLR samples and use the learned bi-causal coupling to obtain the corresponding HFLR samples for each LFLR sample. We generate 10 HFHR samples per conditioning sample and evaluate these generated samples against 20{,}000 HFHR test samples.

We report the distributional and temporal-dependence metrics introduced in Section~\ref{subsec:synthetic_exp_setup} (definitions in Appendix~\ref{sec:gen_set_metrics}). In addition, we report three metrics (definitions in Appendix~\ref{sec:eval_met_sd}):

\begin{itemize}
  \item \textbf{Constraint relative root mean squared error}, $\mathtt{cRMSE}$, which measures how well the generated trajectories satisfy the low-resolution constraint.
  \item \textbf{Wasserstein-1 distance}, $\mathtt{Wass1}$, which quantifies discrepancies in one-step marginals pooled across generated samples, conditionings, and assets. 
  \item \textbf{Kernel-density-estimated KL divergence}, $\mathtt{KLD}$, which measures one-step marginal discrepancy through kernel-density estimates. 
\end{itemize}

\paragraph{Main findings.}
Our method \textsf{OT+cDfn} robustly outperforms the other two methods across all reported metrics. We repeat the experiment over 10 independent runs, and report the mean with standard deviation (SD) in Table~\ref{tab:ar1_metrics_main}. In particular, \textsf{OT+cDfn} achieves the smallest values on all distributional metrics while maintaining zero $\mathtt{cRMSE}$, indicating that it both satisfies the low-resolution constraint and most accurately recovers the HFHR law. Moreover, this result suggests that low-resolution conditioning is effective only after matching with the HFHR law.

\begingroup
\renewcommand{\arraystretch}{0.9}
\setlength{\extrarowheight}{0pt}
\begin{table}[!ht]
\centering
\small
\caption{
Performance metrics per generation method, reporting mean values and standard deviations (in parentheses).
\label{tab:ar1_metrics_main}}
{
\begin{tabular*}{\textwidth}{@{\extracolsep{\fill}}l *{7}{c}@{}}
\toprule
\textbf{Method} & \multicolumn{7}{c}{\textbf{Distributional Metrics}} \\
\cmidrule(lr){2-8}
& $\mathtt{cRMSE}$ & $\mathtt{W2\_{avg}}$ & $\mathtt{KS\_{avg}}$ & $\mathtt{SWD}$ & $\mathtt{MMD^2}$ & $\mathtt{Wass1}$ & $\mathtt{KLD}$ \\
\midrule
\textsf{uDfn} & 0.2566 & 8.4465 & 0.0691 & 2.5145 & 0.0024 & 1.2662 & 0.1945 \\
& (0.0001) & (0.0290) & (0.0021) & (0.0074) & (0.0000) & (0.0336) & (0.0008) \\
\midrule
\textsf{cDfn} & 0.0000 & 8.1402 & 0.0812 & 2.7195 & 0.0037 & 1.3163 & 0.5818 \\
& (0.0000) & (0.3496) & (0.0043) & (0.0392) & (0.0001) & (0.1760) & (0.0121) \\
\midrule
\textsf{OT+cDfn} & 0.0000 & 2.8679 & 0.0416 & 1.4133 & 0.0007 & 0.7895 & 0.0610 \\
& (0.0000) & (0.0520) & (0.0008) & (0.0102) & (0.0000) & (0.0185) & (0.0012) \\
\bottomrule
\end{tabular*}
}
{}
\end{table}
\endgroup

The \textsf{OT+cDfn} method improves the estimation of marginal distributions. Figure~\ref{fig:marginal_densities_d1_and_d2} compares the marginal densities across time for the three methods. Without conditioning, the \textsf{uDfn} spreads probability mass too broadly and generates heavier tails; the \textsf{cDfn} partially sharpens the marginals, but introduces systematic distortion because the conditioning signal is drawn from the wrong coarse-scale law. After OT-based debiasing, \textsf{OT+cDfn} allocates probability mass more accurately across the support and tracks the target marginals more closely over time. These results further verify that debiasing through bi-causal OT indeed contributes to downstream marginal reconstruction of the HFHR process.

\begin{figure}[!ht]
    \centering
    \caption{Marginal distributions across time for the downscaling setting ($d=2$, $N_x=299$).}
    \label{fig:marginal_densities_d1_and_d2}
    
    \subfigure[$n=60$]{
        \includegraphics[width=0.48\textwidth]{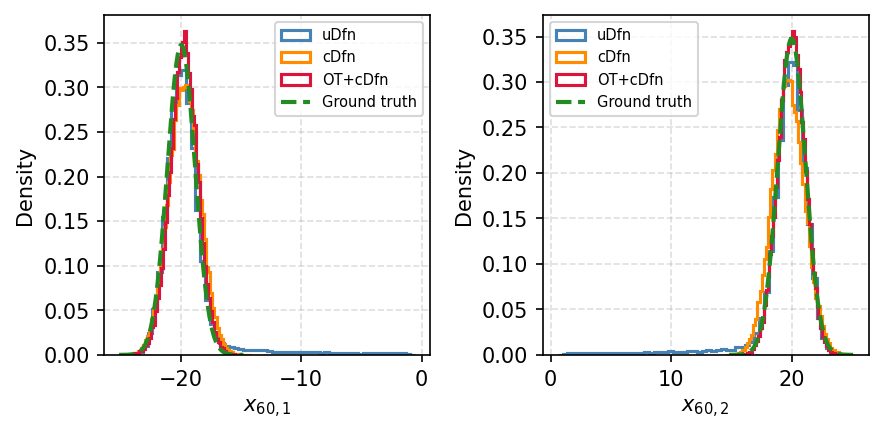}
    }
    \hfill
    \subfigure[$n=140$]{
        \includegraphics[width=0.48\textwidth]{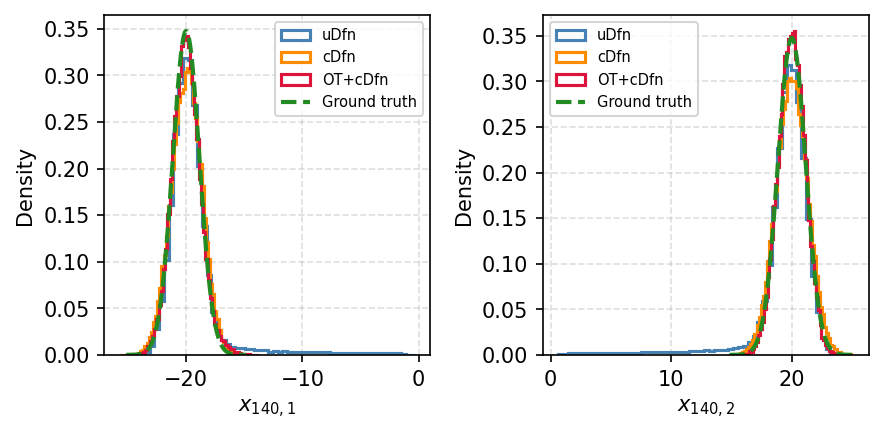}
    }

        \subfigure[$n=220$]{
        \includegraphics[width=0.4\textwidth]{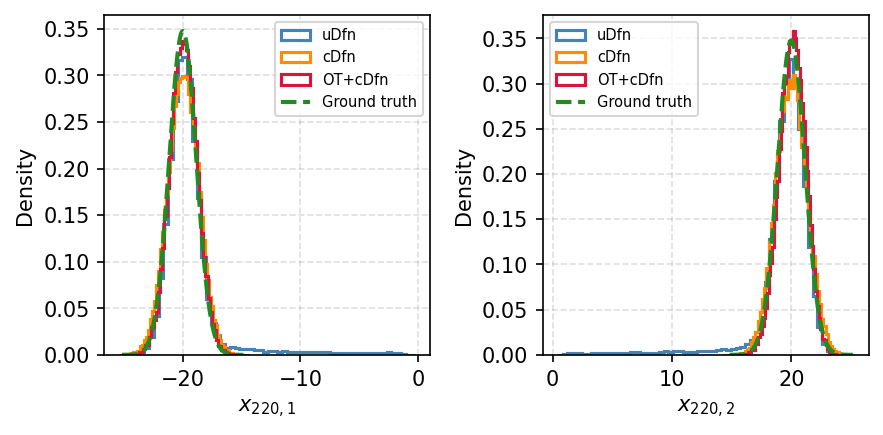}
    }
\end{figure}

The \textsf{OT+cDfn} method shows improvements in estimated temporal correlations, consistently tracking the ground truth and more accurately reproducing the pattern of adjacent-step correlations. Figure~\ref{fig:sd_ar_corr} reports adjacent-step correlations for different methods. The \textsf{uDfn} benchmark systematically underestimates the autocorrelation of the AR process. In contrast, the \textsf{cDfn} conditions on coarse LFLR data and overshoots the target correlations. Once we use the HFLR time series, the correlations generated by \textsf{OT+cDfn} are substantially closer to the ground truth. Therefore, debiasing contributes not only to correcting marginals but also to improving the temporal structure learned by conditional diffusion models.

\begin{figure}[!ht]
    \centering
    \caption{Adjacent-step correlations over $n=1, 10, 20, \dots,N_x$ ($d=2$, $N_x=299$).}
    \label{fig:sd_ar_corr}
    \includegraphics[width=0.5\linewidth]{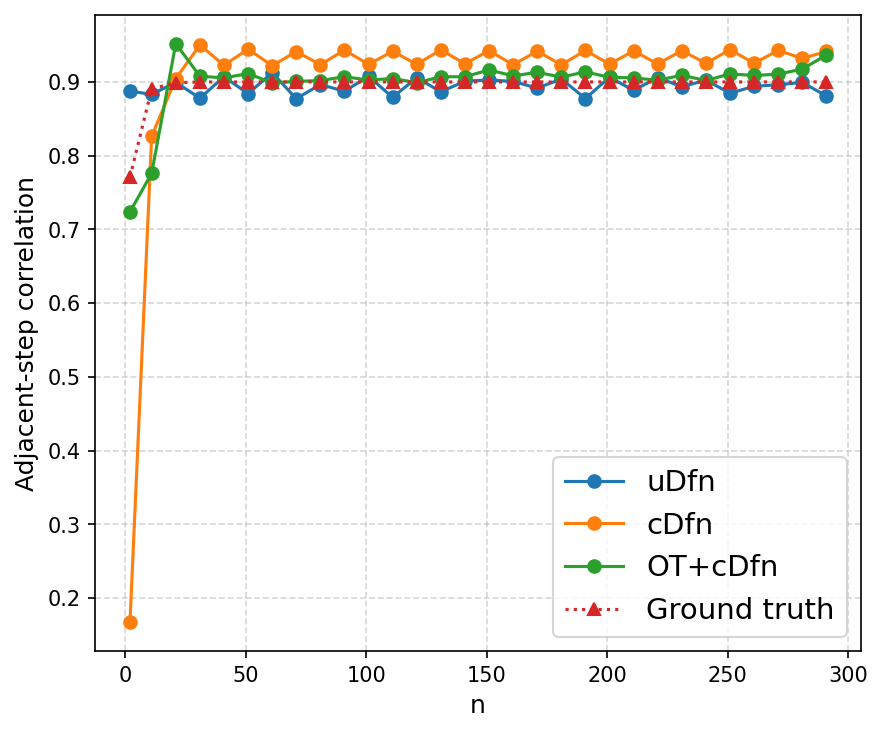}
\end{figure}

\section{Discussion}

This paper develops a scalable computational framework for bi-causal OT under general marginals. By combining a KL-penalized relaxation with the recursive structure of bi-causal OT, we obtain an end-to-end stochastic-optimization method that approximates the original path-space problem without abandoning the underlying information constraints. On the theoretical side, we establish dynamic-programming principles for both the constrained and relaxed formulations, prove \(\Gamma\)-convergence back to the original problem, derive explicit policy-gradient formulas, and provide nonasymptotic optimization guarantees for our algorithm. On the numerical side, we show that the method preserves both marginal laws and temporal dependence, and that it performs well in robust hedging and time series statistical downscaling.

These results make adapted transport computationally feasible in operations research (OR) settings where sequential information matters. Many OR problems involve stochastic processes, multistage decisions, and filtration constraints, including scenario generation and reduction, multistage stochastic optimization, robust finance, and data-driven approximation of dynamical systems. In such settings, respecting information flow is essential for both modeling fidelity and decision quality. Our results suggest that bi-causal OT can serve not only as a structural notion of distance between stochastic processes, but also as a practical computational primitive inside larger optimization and uncertainty-quantification pipelines.

Several promising directions build naturally on the present framework. First, extending the analysis beyond the Markov setting to incorporate path-dependent costs and non-Markovian marginals would further expand its scope to important applications such as American option pricing and inventory management with demand memory. In this setting, recurrent or attention-based architectures offer a flexible modeling route beyond the one-step flow parameterization, together with interesting new questions for convergence analysis. Second, the current framework also opens the door to further advances in both computation and statistical theory, including finite-sample guarantees, adaptive penalty parameters, and accelerated training methods for long-horizon problems.

Overall, the framework proposed here narrows the gap between the theory of adapted transport and the computational demands of modern OR practice. It provides a basis for bringing bi-causal OT to a broader class of sequential modeling and uncertainty quantification problems.

\theendnotes
\makeatletter

\singlespacing
\bibliography{references}


\newpage

\setlength{\baselineskip}{1.5\baselineskip}
\onehalfspacing
\counterwithin{figure}{section}
\makeatletter
\renewcommand\p@subfigure{\thefigure}
\makeatother
\counterwithin{equation}{section}
\counterwithin{table}{section}
\counterwithin{theorem}{section}
\counterwithin{proposition}{section}
\counterwithin{definition}{section}
\counterwithin{example}{section}
\counterwithin{lemma}{section}
\counterwithin{remark}{section}
\counterwithin{assumption}{section}

\appendix
\allowdisplaybreaks
\section{Omitted proofs in Section~\ref{sec: bicausal ot}}
\label{sec: omitted proofs of Section bicausal ot}
In this section, we provide the formal proofs of Lemma~\ref{lem:mbc-equivalence} and Lemma~\ref{lem:dpp}.

\subsection{Proof of Lemma~\ref{lem:mbc-equivalence}}
\label{subsec: proof of lemma mbc equivalence}

\begin{proof}{Proof.}
It suffices to prove $\Pi_{\textrm{bc}}\subset\mathcal M_{\textrm{bc}}$ and $\mathcal M_{\textrm{bc}}\subset\Pi_{\rm bc}$.\\
\noindent\underline{Step 1: $\textrm{ Markov }\Pi_{\textrm{bc}} \subseteq \mathcal M_{\textrm{bc}}$.}
Let $\pi\in\Pi_{\textrm{bc}}(\mu,\mu')$ be Markov and consider any $Q_n^\pi$ defined in \eqref{eq:coupling-kernel}. We only need to verify $Q_n^{\pi}\in\Pi_M(T_n,T'_{n})$. The transition property in \eqref{eq:markov-bicausal-coupling} and the definition of $Q_n^{\pi}$ in \eqref{eq:coupling-kernel} imply that for any $y,y'\in S$ and any $A, B\in\mathcal S$, the first and second marginal distributions of $Q^\pi_n$ satisfy that
\begin{align*}
    Q_n^{\pi,(1)}(A \,|\, y,y')&=\mathbb P_{\pi}\big(Y_n\in A, Y'_n\in \mathcal S \,\big|\, Y_{n-1}=y, Y'_{n-1}=y'\big)=T_n(A \,|\, y),\\
    Q_n^{\pi,(2)}(B \,|\, y,y')&=\mathbb P_{\pi}\big(Y_n\in \mathcal S, Y'_n\in B \,\big|\, Y_{n-1}=y, Y'_{n-1}=y'\big)=T_n'(B \,|\, y').
\end{align*}
Thus, we obtain $Q_n^{\pi}\in\Pi_M(T_n,T'_{n})$.\\
\noindent\underline{Step 2: $\mathcal M_{\textrm{bc}} \subseteq \textrm{ Markov }\Pi_{\textrm{bc}}$.}
Fix $\pi_0\in\Pi(\mu_0,\mu_0')$ and take $Q_n\in\Pi_M(T_n,T_n')$ for $n \in [N]$. 
For any $\pi_{\pi_0,\pmb Q}$ defined in \eqref{eq:markov-coupling}, by construction, the marginals evolve via $T_n$ and $T_n'$, respectively; and $(Y_n,Y_n')$ is Markov in $(Y_{n-1},Y'_{n-1})$ with transition kernel $Q_n$. Moreover, conditioning on past filtrations satisfies the bi-causal constraints since $Q_n$ uses only $(y_{n-1},y'_{n-1})$ and its marginals match $T_n,T_n'$.
Thus $\pi_{\pi_0,\pmb Q}\in\Pi_{\textrm{bc}}(\mu,\mu')$. \hfill\Halmos
\end{proof}

\subsection{Proof of Lemma~\ref{lem:dpp}}
\label{subsec: proof of lemma dpp}

\begin{proof}{Proof.}
For notational simplicity, denote $\pmb\Pi=\Pi(\mu,\mu')$, $\pmb\Pi_{\textrm{bc}}=\Pi_{\textrm{bc}}(\mu,\mu')$ and denote $\Pi_n=\Pi_M(T_n,T_n')$ for all $n\in[N]$.

\noindent\underline{Part 1: $V_n = \Gamma_n[V_{n+1}]$.}
We will prove $V_n \geq \Gamma_n[V_{n+1}]$ and $V_n \leq \Gamma_n[V_{n+1}]$ respectively.

\begin{enumerate}
    \item For any $\pi\in\pmb\Pi_{\textrm{bc}}$, we have $Q_{n+1}^\pi\in\Pi_{n+1}$ for all $n \in [N-1]_0$, and
    \begin{align*}
        J_n^\pi=c_n+Q_{n+1}^\pi[J_{n+1}^\pi] \geq c_n+Q_{n+1}^\pi[V_{n+1}] = \Gamma_n^{Q_{n+1}^{\pi}}[V_{n+1}] \geq\Gamma_n[V_{n+1}].
    \end{align*}
    Here, the first inequality follows from $J_{n+1}^{\pi} \geq V_{n+1}$ and the monotonicity of operator $Q_{n+1}^{\pi}$, and the last inequality follows from \eqref{eq:bellman-op-n-2}. Then, by the definition of $V_n$ in \eqref{eq:opt-cost-n} as the infimum over $\pi\in\pmb{\Pi}_{\textrm{bc}}$, we obtain $V_n \geq \Gamma_n[V_{n+1}]$.

    \item For any $\epsilon > 0$, there exist kernels $Q_{1}^{\epsilon} \in \Pi_1, \dots, Q_{N}^{\epsilon }\in \Pi_N$ such that for any $n \in [N-1]_0$,
    \begin{equation}
    \label{eq:lemma_4_1_part1_eps_kernel}
    \Gamma_n[V_{n+1}] \leq c_n+Q_{n+1}^\epsilon[V_{n+1}] \leq \Gamma_n[V_{n+1}]+\epsilon\,2^{-(n+1)}
    \end{equation}
    Then, for any $\pi_{0} \in \Pi(\mu_0,\mu_0')$, we construct a Markov bi-causal coupling 
    \begin{align}
    \label{eq:lemma_4_1_part1_pi_eps}
    \pi^\epsilon:=\pi_{\pi_0,\pmb Q^\epsilon} \textrm{ with } \pmb Q^\epsilon=(Q_1^\epsilon,\dots,Q_N^\epsilon).
    \end{align}
    Next, we prove by backward induction that for $n \in [N-1]_0$,
    \begin{equation}
    \label{eq:lemma_4_1_part1_ind_bound}
    J_n^{\pi^\epsilon} \leq \Gamma_n[V_{n+1}] + \epsilon\sum_{k=n}^{N-1}2^{-(k+1)}.
    \end{equation}
    When $n=N-1$, by the terminal condition in \eqref{eq:terminal-condition}, the equivalent expression of $J_n^{\pi}$ in \eqref{eq:cost-n} and the inequality of $Q_{n+1}^{\epsilon}$ in \eqref{eq:lemma_4_1_part1_eps_kernel}, we have
    \begin{align*}
        J_{N-1}^{\pi^\epsilon}=c_{N-1}+Q_N^\epsilon[J_N^{\pi^\epsilon}]
        = c_{N-1}+Q_N^\epsilon[V_N] \leq \Gamma_{N-1}[V_N]+\epsilon 2^{-N}.
    \end{align*}
    Assume that \eqref{eq:lemma_4_1_part1_ind_bound} holds for $n \in [N-2]_0$. Then, by the monotonicity of $Q_{n+1}^\epsilon$, the above induction hypothesis for $n+1$ and the second inequality in \eqref{eq:lemma_4_1_part1_eps_kernel}, we deduce that
    \begin{align*}
    J_n^{\pi^\epsilon}
    =c_n+Q_{n+1}^\epsilon[J_{n+1}^{\pi^\epsilon}]
    &\leq c_n+Q_{n+1}^\epsilon\Big[\Gamma_{n+1}[V_{n+2}]+\epsilon\sum_{k=n+1}^{N-1}2^{-(k+1)}\Big] \nonumber \\
    &\leq c_n+Q_{n+1}^\epsilon[V_{n+1}]+\epsilon\sum_{k=n+1}^{N-1}2^{-(k+1)} \nonumber \\
    &\leq \Gamma_n[V_{n+1}]+\epsilon\sum_{k=n}^{N-1}2^{-(k+1)}.
    \end{align*}
    That is, \eqref{eq:lemma_4_1_part1_ind_bound} holds for all $n \in [N-1]_0$.
    
    Since $\sum_{k=n}^{N-1}2^{-(k+1)}<1$, we have
    $J_n^{\pi^\epsilon} \leq \Gamma_n[V_{n+1}]+\epsilon$. Combined with the optimality of $V_n$, we derive 
    $$
    V_n \leq J_n^{\pi^\epsilon} \leq \Gamma_n[V_{n+1}]+\epsilon.
    $$
    Letting $\epsilon\downarrow0$ yields $V_n \leq \Gamma_n[V_{n+1}]$.
\end{enumerate}

Combining the inequality in both directions, we obtain $V_n = \Gamma_n[V_{n+1}]$.

\vspace{4pt}
\noindent\underline{Part 2: Representations of $W_{\textrm{bc}}$.}
First, notice that for any $\pi\in\pmb\Pi_{\textrm{bc}}$, using the tower property and the expression of $J^\pi_n$ in \eqref{eq:cost-n}, we have $$
\E_{\pi}\left[\sum_{n=0}^N c_n(Y_n, Y'_n)\right]=\mathbb E_{\pi}\left[\mathbb E_{\pi}\Big[\sum_{n=0}^Nc_n(Y_n,Y_n') \,\Big|\, Y_0,Y_0'\Big]\right] =\E_{(\pmb Y,\pmb Y')\sim\pi}[J_0^{\pi}(Y_0, Y'_0)],
$$
and therefore, \eqref{eq:bicausal-ot} is equivalent to $W_{\textrm{bc}}=\inf_{\pi\in\pmb\Pi_{\textrm{bc}}}\E_{\pi}[J_0^{\pi}(Y_0, Y'_0)]$. To show the first equivalent representation for $W_{\textrm{bc}}$, we combine the following two results.
\begin{enumerate}
    \item 
    For any $\pi \in \pmb\Pi_{\textrm{bc}}$ with $\pi_0 \in \Pi(\mu_0,\mu_0')$, $J_0^{\pi} \geq V_0$ and thus
    \begin{align*}
        \mathbb E_{\pi}\left[J_0^\pi(Y_0,Y_0')\right]
        &\ge \mathbb E_{\pi}\left[V_0(Y_0,Y_0')\right] =\mathbb E_{\pi_0}\left[V_0(Y_0,Y_0')\right] \\&\geq\inf_{\pi_0\in\Pi(\mu_0,\mu_0')}\mathbb E_{\pi_0}\left[V_0(Y_0,Y_0')\right].
    \end{align*}
    The definition of $W_{\textrm{bc}}$ in \eqref{eq:bicausal-ot} as the infimum over $\pi\in\pmb \Pi_{\textrm{bc}}$ yields 
    $
    W_{\textrm{bc}}(\mu,\mu') \geq \underset{\pi_0\in\Pi(\mu_0,\mu_0')}{\inf}\E_{\pi_0}[V_0]
    $.

    \item For any $\epsilon>0$, there exists $\gamma_0^{\epsilon} \in \Pi(\mu_0,\mu_0')$ such that 
    \begin{align}
    \label{eq:lemma_4_1_part2_gamma_0_eps}
        \E_{\gamma_0^{\epsilon}}[V_0] \leq \inf_{\pi_0 \in \Pi(\mu_0,\mu'_0)}\E_{\pi_0}[V_0] + \epsilon.
    \end{align}
    Using $\pmb Q^\varepsilon$ defined in \eqref{eq:lemma_4_1_part1_pi_eps}, we construct a Markov bi-causal coupling
    $\gamma^\epsilon :=\pi_{\gamma_0^{\epsilon},\pmb Q^\varepsilon}\in\mathcal M_{\textrm{bc}}$. Notice that \eqref{eq:lemma_4_1_part1_ind_bound} holds for any initial coupling in $\Pi(\mu_0,\mu_0')$. 
    Therefore, Part I and \eqref{eq:lemma_4_1_part1_ind_bound} imply that $V_0\leq J_0^{\gamma^\epsilon}<V_0+\epsilon$, and using \eqref{eq:lemma_4_1_part2_gamma_0_eps}, we have
    \begin{align*}
        \E_{\gamma^\epsilon}[J_0^{\gamma^\epsilon}] \leq \E_{\gamma^{\epsilon}}[V_0]+\epsilon =\E_{\gamma_0^\epsilon}[V_0]+\epsilon\leq \inf_{\pi_0 \in \Pi(\mu_0,\mu'_0)}\E_{\pi_0}[V_0]+2\epsilon.
    \end{align*}
    Therefore, by the definition of $W_{\textrm{bc}}$ in \eqref{eq:bicausal-ot} as the infimum over $\pi\in\pmb \Pi_{\textrm{bc}}$ and the arbitrarily chosen $\epsilon$, we obtain $W_{\textrm{bc}}(\mu,\mu') \leq \inf_{\pi_0\in\Pi(\mu_0,\mu_0')}\E_{\pi_0}[V_0]$.
\end{enumerate}
To establish the second equivalent representation, notice that $\gamma^\epsilon\in \mathcal M_{\textrm{bc}}(\mu,\mu')\subset\pmb\Pi_{\textrm{bc}}$ and $W_{\textrm{bc}}(\mu,\mu') = \inf_{\pi_0\in\Pi(\mu_0,\mu_0')}\E_{\pi_0}[V_0]$, we have 
\begin{align}
W_{\textrm{bc}}(\mu,\mu')
&=\inf_{\pi\in\pmb\Pi_{\textrm{bc}}}\mathbb E_{\pi}\left[\sum_{n=0}^Nc_n(Y_n,Y_n')\right] 
\leq \inf_{\pi\in\mathcal M_{\textrm{bc}}(\mu,\mu')}\mathbb E_{\pi}\left[J_0^{\pi}(Y_0, Y'_0)\right] \nonumber \\
&\leq \mathbb E_{\gamma^\epsilon}[J_0^{\pi^\epsilon}(Y_0,Y_0')] \leq W_{\textrm{bc}}(\mu,\mu')+2\epsilon.
\label{eq:lemma_4_1_part2_ineq3}
\end{align}
Since $\epsilon>0$ is arbitrary, all inequalities in \eqref{eq:lemma_4_1_part2_ineq3} become equalities, and hence
\begin{align*}
W_{\textrm{bc}}(\mu,\mu')=\inf_{\pi_0\in\Pi(\mu_0,\mu_0')}\E_{\pi_0}[V_0(Y_0,Y_0')]=\inf_{\pi\in\mathcal{M}_{\textrm{bc}}(\mu,\mu')}\mathbb E_{\pi}\left[\sum_{n=0}^Nc_n(Y_n,Y_n')\right]. \tag*{\Halmos}
\end{align*}
\end{proof}

\section{Omitted proofs in Section~\ref{sec: relaxed formulation}}
In this section, we provide proofs of Theorem~\ref{thm:gamma-convergence} and Lemma~\ref{lem:dpp2}.

\subsection{Proof of Theorem~\ref{thm:gamma-convergence}}
\label{subsec: proof of theorem gamma convergence}

\begin{proof}{Proof.}
In the following proof, we work on the $\mathcal P(\mathcal{Y}\times \mathcal{Y})$ with the weak convergence topology, and omit the ``weak'' thereafter. Observe that for any $0<\beta\leq\beta'$, 
\begin{eqnarray*}
    \inf_{\pi\in\overline{\Pi}_{\rm rel}}F_\beta(\pi)\leq \inf_{\pi\in\overline{\Pi}_{\rm rel}}F_{\beta'}(\pi)\leq \inf_{\pi \in \overline{\mathcal M}_{\textrm{bc}}(\mu,\mu')}F(\pi)
\end{eqnarray*}

\noindent \underline{Step 1: Existence of minimizers over the feasible limit set $\overline{\mathcal M}_{\textrm{bc}}(\mu,\mu')$.}
To prove that the limit (original) problem in \eqref{eq:bicausal-ot} admits a minimizer over $\overline{\mathcal M}_{\textrm{bc}}(\mu,\mu')$, it suffices to verify the following properties,
\begin{enumerate}
    \item[(i)] lower semi-continuity of $F(\pi)$ on $\overline{\mathcal M}_{\textrm{bc}}(\mu,\mu')$, and
    \item[(ii)] compactness of the feasible limit set $\overline{\mathcal M}_{\textrm{bc}}(\mu,\mu')$.
\end{enumerate}

To prove lower semi-continuity, it suffices to show that, for any $\{\pi^{(m)}\}_{m \ge 1} \subset \overline{\mathcal M}_{\textrm{bc}}(\mu,\mu')$ such that $\pi^{(m)}\rightharpoonup\pi\in\overline{\mathcal M}_{\textrm{bc}}(\mu,\mu')$, \[\liminf_{m\to\infty}F(\pi^{(m)})\geq F(\pi).\]
From Assumption~\ref{asm:lsc_cost}, we know that $C=\sum_{n=0}^Nc_n$ is lower semi-continuous and nonnegative. Therefore, for any $x\geq0$, $C^{-1}((x,+\infty))$ is open in $\mathcal Y\times\mathcal Y$ and $\int Cd\eta=\int_0^\infty\eta\left(C^{-1}((x,+\infty))\right)\dd x$ for any $\eta\in\mathcal P(\mathcal Y\times\mathcal Y)$. Hence, by Portmanteau theorem \cite[Theorem~2.1]{billingsley2013convergence} and Fatou's lemma, 
$$
\liminf_{m\to\infty}\int \sum_{n=0}^N c_n\,\dd\pi^{(m)} \ge \int \sum_{n=0}^N c_n\,\dd\pi.
$$
For the KL divergence term, applying the Donsker–Varadhan variational  (see, for instance, Lemma~1.4.3 in \cite{DupuisEllis1997} and Proposition~1.4.2 in \cite{DemboZeitouni1998}), we have  
\begin{align}
\label{eq:def_dv_formula}
    \mathrm{KL}(\mu, \nu)=\sup_{\varphi\in C_b(\mathcal Y)}\Big\{\textstyle\int \varphi\,\dd\mu-\log\int e^\varphi\,\dd\nu\Big\}.
\end{align}
For each $\varphi\in C_b(\mathcal Y)$, the mapping $\nu\mapsto\int\varphi d\mu-\log{\int e^{\varphi}d\nu}$ is continuous, then, by taking the supremum, $\mathrm{KL}(\mu, \cdot)$ is lower semi-continuous, and hence,
\[
\liminf_{m\to+\infty}\mathrm{KL}^{(1)}(\mu,\pi^{(m)})\geq \mathrm{KL}^{(1)}(\mu,\pi), \quad \liminf_{m\to+\infty}\mathrm{KL}^{(2)}(\mu',\pi^{(m)})\geq \mathrm{KL}^{(2)}(\mu',\pi).
\]
If $\liminf_{m\to+\infty}\mathrm{KL}^{(1)}(\mu,\pi^{(m)})>0$ or $\liminf_{m\to+\infty}\mathrm{KL}^{(2)}(\mu',\pi^{(m)})>0$, then $$
\liminf_{m\to+\infty}\iota\big(\mathrm{KL}^{(1)}(\mu,\pi^{(m)})\big)+\iota\big(\mathrm{KL}^{(2)}(\mu',\pi^{(m)})\big)=+\infty.
$$
Otherwise, $\mathrm{KL}^{(1)}(\mu,\pi)=\mathrm{KL}^{(2)}(\mu',\pi)=0$, and hence $\iota\big(\mathrm{KL}^{(1)}(\mu,\pi)\big)+\iota\big(\mathrm{KL}^{(2)}(\mu',\pi)\big)=0$. In either case,
\[
\liminf_{m\to+\infty}\iota\big(\mathrm{KL}^{(1)}(\mu,\pi^{(m)})\big)+\iota\big(\mathrm{KL}^{(2)}(\mu',\pi^{(m)})\big)\geq \iota\big(\mathrm{KL}^{(1)}(\mu,\pi)\big)+\iota\big(\mathrm{KL}^{(2)}(\mu',\pi)\big).
\]
Therefore, we can conclude that $F(\pi)$ is lower semi-continuous.

To prove compactness, notice that $\mathcal Y\cong \mathbb R^{d(N+1)}$ is Polish, therefore $\mu,\mu'\in\mathcal P(\mathcal Y)$ are tight, respectively, that is, for any $\epsilon>0$, there exist compact sets $K,K'\subset\mathcal Y$ such that
$\mu(K)\ge 1-\epsilon/2$ and $\mu'(K')\ge 1-\epsilon/2$.
Then, for any $\pi\in{\mathcal M}_{\textrm{bc}}(\mu,\mu')$, we have
\[
\pi\big((K\times K')^\complement\big)
\le \pi(K^\complement\times\mathcal Y)+\pi(\mathcal Y\times (K')^\complement)
= \mu(K^\complement)+\mu'((K')^\complement)\le \epsilon,
\]
where the first inequality follows from the fact that $(K\times K')^\complement\subset \left(K^\complement\times\mathcal Y\right)\bigcup\left(\mathcal Y\times (K')^\complement\right)$. Thus, ${\mathcal M}_{\mathrm{bc}}(\mu,\mu')$ is tight in $\mathcal P(\mathcal Y\times\mathcal Y)$ and by Prokhorov's theorem \cite[Theorem~5.1]{billingsley2013convergence}, $\overline{\mathcal M}_{\mathrm{bc}}(\mu,\mu')$ is compact in $\mathcal P(\mathcal Y\times\mathcal Y)$. 

Finally, notice that Assumption~\ref{asm:lsc_cost} guarantees non-degeneracy of the following problem:
$$
\inf_{\pi\in\overline{\mathcal M}_{\textrm{bc}}(\mu,\mu')}F(\pi)\leq\inf_{\pi\in\mathcal M_{\rm bc}(\mu,\mu')}F(\pi) = \inf_{\pi\in{\mathcal M}_{\textrm{bc}}(\mu,\mu')}\int \sum_{n=0}^{N}c_n\,d\pi<+\infty.
$$
By compactness of $\overline{\mathcal M_{\rm bc}}(\mu,\mu')$, for any minimizing sequence $\{\pi^{(m)}\}_{m \geq 1}\subset\overline{\mathcal M}_{\textrm{bc}}(\mu,\mu')$ with $F(\pi^{(m)})\downarrow \inf_{\pi \in \overline{\mathcal M}_{\textrm{bc}}(\mu,\mu')}F(\pi)$, it admits a subsequential limit, $\pi^{(m_k)}\overset{k\uparrow\infty}{\rightharpoonup}\overline{\pi}\in\overline{\mathcal M}_{\textrm{bc}}(\mu,\mu')$. Immediately, the lower semi-continuity of $F$ implies that $\bar\pi$ is indeed a minimizer,
\begin{equation}
\label{eq:finite_F_min}
    \inf_{\pi \in \overline{\mathcal M}_{\textrm{bc}}(\mu,\mu')}F(\pi) \le\ F(\overline{\pi})\ \le\ \liminf_{k\to\infty}F(\pi^{(m_k)})\ =\ \inf_{\pi \in \overline{\mathcal M}_{\textrm{bc}}(\mu,\mu')}F(\pi).
\end{equation}

\vspace{10pt}

\noindent \underline{Step 2: $\Gamma$-liminf sequences and lower bound.} Consider $\{\pi^\beta\}_{\beta>0}$ in $\overline{\Pi}_{\textrm{rel}}$ such that $\pi^{\beta}\overset{\beta\uparrow\infty}{\rightharpoonup} \pi$ for some $\pi\in\mathcal P(\mathcal Y\times\mathcal Y)$. 
By the continuous mapping theorem \cite[Theorem~2.7]{billingsley2013convergence}, weak convergence applies to the marginals as well. Similar to the arguments in Step 1, we have
$$
\liminf_{\beta\to\infty}\int \sum_{n=0}^{N} c_n \,\dd\pi^{\beta}\, \ge\, \int \sum_{n=0}^{N} c_n \,\dd\pi,
$$
and
\begin{equation}
\begin{split}
\label{eq:theorem4_3_step2_kl_ineq}
&\liminf_{\beta\to\infty}\mathrm{KL}^{(1)}\big(\mu, \pi^\beta\big) \geq \mathrm{KL}^{(1)}\big(\mu, \pi\big),\qquad \liminf_{\beta\to\infty}\mathrm{KL}^{(2)}\big(\mu', \pi^\beta\big) \geq \mathrm{KL}^{(2)}\big(\mu', \pi\big).        
\end{split}
\end{equation}

If $\pi\in\overline{\mathcal{M}}_{\mathrm{bc}}(\mu,\mu')$, then we have $\mathrm{KL}^{(1)}\big(\mu, \pi\big) = \mathrm{KL}^{(2)}\big(\mu', \pi\big) = 0$. Immediately, it holds that
$$
\liminf_{\beta\to\infty}F_\beta(\pi^{\beta}) \geq \int \sum_{n=0}^{N} c_n \,\dd\pi\ =\ F(\pi).
$$

If $\pi\notin\overline{\mathcal{M}}_{\mathrm{bc}}(\mu,\mu')$, then we have $\mathrm{KL}^{(1)}\big(\mu, \pi\big) > 0$ or $\mathrm{KL}^{(2)}\big(\mu', \pi\big) > 0$. Combined with \eqref{eq:theorem4_3_step2_kl_ineq}, we have
$$
\liminf_{\beta \to \infty} \big[\beta \mathrm{KL}^{(1)}\big(\mu, \pi^\beta\big) + \beta \mathrm{KL}^{(2)}\big(\mu', \pi^\beta\big)\big]>\lim_{\beta\to\infty}\beta \cdot \frac{\mathrm{KL}^{(1)}\big(\mu, \pi\big)+\mathrm{KL}^{(2)}\big(\mu', \pi\big)}{2}= +\infty.
$$
By the definition of $F_{\beta}$ in \eqref{eq:def_F_beta}, we have
$\liminf_{\beta\to\infty}F_\beta(\pi^{\beta})=+\infty \geq F(\pi)$. 

Therefore, in summary, we obtain the following inequality
\begin{align}
\label{eq:theorem4_3_step2_bound}
\liminf_{\beta\to\infty}F_\beta(\pi^{\beta}) \geq F(\pi).
\end{align}

\noindent \underline{Step 3: Recovery sequences and upper bound.}
For any $\pi\in\overline{\mathcal{M}}_{\mathrm{bc}}(\mu,\mu')$, we take the constant sequence $\pi^{\beta}\equiv\pi$. Then, the KL divergence terms vanish and we obtain $F_{\beta}(\pi^{\beta}) = F(\pi)$. Therefore, we have $\limsup_{\beta\to\infty}F_\beta(\pi^{\beta}) = F(\pi)$. 

For any $\pi\notin\overline{\mathcal{M}}_{\mathrm{bc}}(\mu,\mu')$, then $F(\pi)=+\infty$ and the result is trivial.

Hence, we deduce
\begin{align}
\label{eq:theorem4_3_step3_bound}
\limsup_{\beta\to\infty}F_\beta(\pi^{\beta}) \leq F(\pi).
\end{align}

{\noindent \underline{Step 4: Compactness of almost minimizers and convergence.} Inspired by Theorem~3.2 of \cite{thorpe2018deep}, we consider $\{\pi^{\beta}\}_{\beta > 0}$, a sequence of almost minimizers of $F_{\beta}$ defined in \eqref{eq:def_almost_minimizer}. By the definition of $\pi^{\beta}$ in \eqref{eq:def_almost_minimizer} and the fact that $\sum_{n=0}^N c_n \geq 0$ in Assumption~\ref{asm:lsc_cost}, we have
\begin{align*}
    \mathrm{KL}^{(1)}\big(\mu, \pi^\beta\big) + \mathrm{KL}^{(2)}\big(\mu', \pi^\beta\big)
    &\leq \frac{\inf_{\pi \in \overline{\Pi}_{\mathrm{rel}}} F_\beta(\pi)+\epsilon_\beta}{\beta} \nonumber \\
    &\leq \frac{\inf_{\pi \in \overline{\mathcal{M}}_{\mathrm{bc}}(\mu,\mu')} F_\beta(\pi)+\epsilon_\beta}{\beta} \nonumber \\
    &= \frac{\inf_{\pi \in \overline{\mathcal{M}}_{\mathrm{bc}}(\mu,\mu')} F(\pi)+\epsilon_\beta} {\beta} \xrightarrow[]{\beta\to\infty} 0,
\end{align*}
where the second inequality holds since $\overline{\mathcal{M}}_{\mathrm{bc}}(\mu,\mu') \subset \overline{\Pi}_{\mathrm{rel}}$; the last equality follows from the fact that $\mathrm{KL}^{(1)}\big(\mu, \pi\big)+\mathrm{KL}^{(2)}\big(\mu', \pi\big) = 0$ for any $\pi \in \overline{\mathcal{M}}_{\mathrm{bc}}$.
Furthermore, by Pinsker’s inequality \cite[Lemma 2.5]{tsy2009nonparametric}, we obtain
\begin{align}
&\mathrm{TV}\big(\mu, \pi^{\beta}(\cdot \times \mathcal{Y})\big) \leq \sqrt{\frac{1}{2}\,\mathrm{KL}^{(1)}\big(\mu, \pi^\beta\big)}\ \xrightarrow[]{\beta\to\infty}\ 0, \label{eq:theorem_4_3_step4_tv1_limit} \\
&\mathrm{TV}(\mu', \pi^{\beta}(\mathcal{Y} \times \cdot)\big) \leq \sqrt{\frac{1}{2}\, \mathrm{KL}^{(2)}\big(\mu', \pi^\beta\big)} \ \xrightarrow[]{\beta\to\infty}\ 0.
\label{eq:theorem_4_3_step4_tv2_limit}
\end{align}
Then, given any $\epsilon > 0$, we choose compact subsets $K,K'\subset\mathcal Y$ with $\mu(K) \geq 1-\epsilon/4$ and $\mu'(K') \geq 1-\epsilon/4$. Thanks to \eqref{eq:theorem_4_3_step4_tv1_limit} and \eqref{eq:theorem_4_3_step4_tv2_limit}, we can choose $\beta$ sufficiently large such that $\mathrm{TV}\big(\mu, \pi^{\beta}(\cdot \times \mathcal{Y})\big) \leq \epsilon/4$ and $\mathrm{TV}\big(\mu', \pi^{\beta}(\mathcal{Y} \times \cdot)\big) \leq \epsilon/4$. This further implies that 
\begin{align}
&\pi^{\beta}(K \times \mathcal{Y}) \geq \mu(K) - \mathrm{TV}\big(\mu, \pi^{\beta}(\cdot \times \mathcal{Y})\big) \geq 1-\epsilon/2, \label{eq:theorem4_3_step4_pi_tight1} \\
&\pi^{\beta}(\mathcal{Y} \times K') \geq \mu'(K') - \mathrm{TV}\big(\mu', \pi^{\beta}(\mathcal{Y} \times \cdot)\big) \geq 1-\epsilon/2.
\label{eq:theorem4_3_step4_pi_tight2}
\end{align}
Hence, the marginals of $\pi^{\beta}$ are uniformly tight. By the properties of set intersection and the tightness of marginals in \eqref{eq:theorem4_3_step4_pi_tight1} and \eqref{eq:theorem4_3_step4_pi_tight2}, we have
$$
\pi^{\beta}\big((K\times K')^\complement\big) \leq \pi^{\beta}(K^\complement \times \mathcal{Y})+\pi^{\beta}(\mathcal{Y} \times (K')^\complement) \leq \epsilon,
$$
so $\{\pi^\beta\}$ is uniformly tight in $\mathcal P(\mathcal Y\times\mathcal Y)$. 
By Prokhorov’s theorem \cite[Theorem~5.1]{billingsley2013convergence}, uniform tightness yields relative compactness: there exists a limit $\pi^{*}$ and a subsequence such that $\pi^\beta\rightharpoonup \pi^*$. Here, $\pi^\ast\in\overline{\Pi}_{\mathrm{rel}}$ since the feasible set $\overline{\Pi}_{\mathrm{rel}}$ is closed.

Finally, we conclude by a standard $\Gamma$-convergence argument. 
Now, we have verified the conditions needed in \cite[Theorem~3.2]{thorpe2018deep}: the relative compactness of $\{\pi^{\beta}\}_{\beta > 0}$ and the $\Gamma$-convergence induced by the liminf and limsup bounds \eqref{eq:theorem4_3_step2_bound} and \eqref{eq:theorem4_3_step3_bound}. Then, applying \cite[Theorem~3.2]{thorpe2018deep} to $\Omega=\overline{\Pi}_{\mathrm{rel}}$, $\mathcal{E}_n=F_{\beta_n}$ (with respect to $F$), and $u_n=\pi^{\beta_n}$ (with respect to $\pi^\ast$) for any sequence $\beta_n\to+\infty$, we deduce that
$$
\inf_{\pi\in\overline{\Pi}_{\mathrm{rel}}}F_{\beta_n}(\pi)\ \longrightarrow\ \min_{\pi\in\overline{\Pi}_{\mathrm{rel}}}F(\pi) \ \textrm{ and } \
\pi^{\beta_n}\rightharpoonup \pi^\ast\in \underset{\pi\in\overline{\Pi}_{\mathrm{rel}}}{\operatorname{argmin}} \ F(\pi), \quad \textrm{ as } \beta_n \to +\infty.
$$
Using the fact that $\overline{\mathcal{M}}_{\mathrm{bc}} \subset \overline{\Pi}_{\mathrm{rel}}$ and the existence of the minimizers for $F(\pi)$ in \eqref{eq:finite_F_min}, we have
$$
\min_{\pi\in\overline{\Pi}_{\mathrm{rel}}}F(\pi) \le \min_{\pi\in\overline{\mathcal{M}}_{\mathrm{bc}}}F(\pi) < +\infty.
$$
On the other hand, notice that $F(\pi)$ is finite if and only if
$$
\iota\Big(\mathrm{KL}^{(1)}\big(\mu, \pi\big)\Big) = \iota\Big(\mathrm{KL}^{(2)}\big(\mu', \pi\big)\Big) = 0,\quad \textrm{ i.e.,}\quad  \pi \in \overline{\mathcal{M}}_{\mathrm{bc}}.
$$
Consequently, any minimizer $\pi^*$ of $F(\pi)$ must belong to $\overline{\mathcal{M}}_{\mathrm{bc}}$, which gives the desired results.
}
\hfill\Halmos
\end{proof}

\subsection{Proof of Lemma~\ref{lem:dpp2}}
\label{subsec: proof of lemma dpp2}

\begin{proof}{Proof.}
We prove two parts separately.

\noindent\underline{Part 1: $\widetilde V_n=\widetilde\Gamma_n[\widetilde V_{n+1}]$.}
We will prove $\widetilde V_n \geq \widetilde\Gamma_n[\widetilde V_{n+1}]$ and $\widetilde V_n \leq \widetilde\Gamma_n[\widetilde V_{n+1}]$, respectively.
\begin{enumerate}
    \item For any $\pi\in\Pi_{\mathrm{rel}}$, we consider $Q_{n+1}^{\pi}\in\mathcal U_{\mathrm{rel}}$ for all $n \in [N-1]_0$, and
    \begin{align*}
        \widetilde J_n^\pi
        &=
        c_n
        +\beta\Big[\mathrm{KL}_{n+1}^{\!(1)}(\mu,\pi)+\mathrm{KL}_{n+1}^{\!(2)}(\mu',\pi)\Big]
        +Q_{n+1}^{\pi}[\widetilde J_{n+1}^{\pi}] \nonumber\\
        &\geq
        c_n
        +\beta\Big[\mathrm{KL}_{n+1}^{\!(1)}(\mu,\pi)+\mathrm{KL}_{n+1}^{\!(2)}(\mu',\pi)\Big]
        +Q_{n+1}^{\pi}[\widetilde V_{n+1}] \nonumber\\
        &=\widetilde\Gamma_n^{Q_{n+1}^{\pi}}[\widetilde V_{n+1}]
        \geq
        \widetilde\Gamma_n[\widetilde V_{n+1}].
    \end{align*}
    Here, the first inequality follows from $\widetilde J_{n+1}^{\pi}\geq \widetilde V_{n+1}$ and the monotonicity of the operator $Q_{n+1}^{\pi}$, and the last inequality follows from the definition of $\widetilde\Gamma_n$ in \eqref{eq:bellman-op-n-2-2}. Then, by the definition of $\widetilde V_n$ in \eqref{eq:relaxed-opt-cost-n} as the infimum over $\pi\in\Pi_{\mathrm{rel}}$, we obtain
    $$
    \widetilde V_n\geq \widetilde\Gamma_n[\widetilde V_{n+1}].
    $$

    \item For any $\epsilon>0$, there exist kernels $\widetilde Q_1^{\epsilon},\dots,\widetilde Q_N^{\epsilon}\in\mathcal U_{\mathrm{rel}}$ such that for any $n \in [N-1]_0$,
    \begin{equation}
    \label{eq:lemma_rel_part1_eps_kernel}
    \widetilde\Gamma_n[\widetilde V_{n+1}]
    \leq
    c_n
    +\beta\Big[\mathrm{KL}_{n+1}^{\!(1)}(\mu,\widetilde\pi^\epsilon)+\mathrm{KL}_{n+1}^{\!(2)}(\mu',\widetilde\pi^\epsilon)\Big]
    +\widetilde Q_{n+1}^{\epsilon}[\widetilde V_{n+1}]
    \leq
    \widetilde\Gamma_n[\widetilde V_{n+1}]+\epsilon\,2^{-(n+1)}.
    \end{equation}
    Then, for any $\widetilde\pi_0\in\mathcal P(S\times S)$, we construct a relaxed Markov policy
    \begin{align}
    \label{eq:lemma_rel_part1_pi_eps}
    \widetilde\pi^\epsilon:=\pi_{\widetilde\pi_0,\pmb{\widetilde Q}^{\epsilon}}
    \qquad\text{with}\qquad
    \pmb{\widetilde Q}^{\epsilon}=(\widetilde Q_1^\epsilon,\dots,\widetilde Q_N^\epsilon).
    \end{align}
    Next, we prove by backward induction that for $n \in [N-1]_0$,
    \begin{equation}
    \label{eq:lemma_rel_part1_ind_bound}
    \widetilde J_n^{\widetilde\pi^\epsilon}
    \leq
    \widetilde\Gamma_n[\widetilde V_{n+1}]
    +
    \epsilon\sum_{k=n}^{N-1}2^{-(k+1)}.
    \end{equation}

    When $n=N-1$, by the terminal condition in \eqref{eq:relaxed-terminal-cost} and the representation of $\widetilde J_n^\pi$ in \eqref{eq:relaxed-cost-n}, we have
    \begin{align*}
        \widetilde J_{N-1}^{\widetilde\pi^\epsilon}
        &=
        c_{N-1}
        +\beta\Big[\mathrm{KL}_{N}^{\!(1)}(\mu,\widetilde\pi^\epsilon)+\mathrm{KL}_{N}^{\!(2)}(\mu',\widetilde\pi^\epsilon)\Big]
        +\widetilde Q_N^\epsilon[\widetilde J_N^{\widetilde\pi^\epsilon}] \nonumber\\
        &=
        c_{N-1}
        +\beta\Big[\mathrm{KL}_{N}^{\!(1)}(\mu,\widetilde\pi^\epsilon)+\mathrm{KL}_{N}^{\!(2)}(\mu',\widetilde\pi^\epsilon)\Big]
        +\widetilde Q_N^\epsilon[\widetilde V_N] \nonumber\\
        &\leq
        \widetilde\Gamma_{N-1}[\widetilde V_N]+\epsilon\,2^{-N},
    \end{align*}
    where the last inequality follows from the inequality in \eqref{eq:lemma_rel_part1_eps_kernel} with $n=N-1$.
    Assume that \eqref{eq:lemma_rel_part1_ind_bound} holds for $n \in [N-2]_0$. Then, by the monotonicity of $\widetilde Q_{n+1}^{\epsilon}$ and the induction hypothesis for $n+1$, we deduce that
    \begin{align*}
    \widetilde J_n^{\widetilde\pi^\epsilon}
    &=
    c_n
    +\beta\Big[\mathrm{KL}_{n+1}^{\!(1)}(\mu,\widetilde\pi^\epsilon)+\mathrm{KL}_{n+1}^{\!(2)}(\mu',\widetilde\pi^\epsilon)\Big]
    +\widetilde Q_{n+1}^{\epsilon}[\widetilde J_{n+1}^{\widetilde\pi^\epsilon}] \\
    &\leq
    c_n
    +\beta\Big[\mathrm{KL}_{n+1}^{\!(1)}(\mu,\widetilde\pi^\epsilon)+\mathrm{KL}_{n+1}^{\!(2)}(\mu',\widetilde\pi^\epsilon)\Big] \\
    &\qquad\qquad
    +\widetilde Q_{n+1}^{\epsilon}\left[
    \widetilde\Gamma_{n+1}[\widetilde V_{n+2}]
    +\epsilon\sum_{k=n+1}^{N-1}2^{-(k+1)}
    \right] \\
    &\leq
    c_n
    +\beta\Big[\mathrm{KL}_{n+1}^{\!(1)}(\mu,\widetilde\pi^\epsilon)+\mathrm{KL}_{n+1}^{\!(2)}(\mu',\widetilde\pi^\epsilon)\Big]
    +\widetilde Q_{n+1}^{\epsilon}[\widetilde V_{n+1}]
    +\epsilon\sum_{k=n+1}^{N-1}2^{-(k+1)} \\
    &\leq
    \widetilde\Gamma_n[\widetilde V_{n+1}]
    +\epsilon\sum_{k=n}^{N-1}2^{-(k+1)},
    \end{align*}
    where the last inequality invokes the second inequality in \eqref{eq:lemma_rel_part1_eps_kernel}.
    
    Now, we have proved that \eqref{eq:lemma_rel_part1_ind_bound} holds for all $n \in [N-1]_0$. Since $\sum_{k=n}^{N-1}2^{-(k+1)}<1$, we further derive
    \begin{equation}
    \label{eq:ineq_tildeJ_bellman}
        \widetilde J_n^{\widetilde\pi^\epsilon}\leq \widetilde\Gamma_n[\widetilde V_{n+1}]+\epsilon.
    \end{equation}
    Combining the inequality \eqref{eq:ineq_tildeJ_bellman} with the optimality of $\widetilde V_n$ defined in \eqref{eq:relaxed-opt-cost-n}, we obtain
    $$
    \widetilde V_n\leq \widetilde J_n^{\widetilde\pi^\epsilon}\leq \widetilde\Gamma_n[\widetilde V_{n+1}]+\epsilon.
    $$
    Letting $\epsilon\downarrow0$ yields $\widetilde V_n\leq \widetilde\Gamma_n[\widetilde V_{n+1}]$.
\end{enumerate}

Combining the inequalities in both directions, we obtain
$$
\widetilde V_n=\widetilde\Gamma_n[\widetilde V_{n+1}].
$$

\vspace{4pt}
\noindent\underline{Part 2: Representation of $W_{\mathrm{rel}}(\mu,\mu')$.}
First, for any $\pi\in\Pi_{\mathrm{rel}}$, using the tower property and the expression of $\widetilde J_n^\pi$ in \eqref{eq:relaxed-cost-n-OG}, we have
\begin{align*}
&\E_{\pi}\left[\sum_{n=0}^{N}c_n(Y_n,Y_n')
+\beta\Big(\mathrm{KL}_{0}^{\!(1)}(\mu,\pi)+\mathrm{KL}_{0}^{\!(2)}(\mu',\pi)\Big)\right] \nonumber\\
&\qquad=
\E_{\pi_0}\left[
\widetilde J_0^\pi(Y_0,Y_0')
+\beta\Big(\mathrm{KL}_{0}^{\!(1)}(\mu,\pi)+\mathrm{KL}_{0}^{\!(2)}(\mu',\pi)\Big)
\right]. \nonumber
\end{align*}
Therefore, \eqref{eq:relaxed-bicausal-ot} is equivalent to
$$
W_{\mathrm{rel}}(\mu,\mu')
=
\inf_{\pi\in\Pi_{\mathrm{rel}}}
\E_{\pi_0}\left[
\widetilde J_0^\pi(Y_0,Y_0')
+\beta\Big(\mathrm{KL}_{0}^{\!(1)}(\mu,\pi)+\mathrm{KL}_{0}^{\!(2)}(\mu',\pi)\Big)
\right].
$$
To show the equivalent representation for $W_{\mathrm{rel}}(\mu,\mu')$, we combine the following two results.

\begin{enumerate}
    \item For any $\pi\in\Pi_{\mathrm{rel}}$, we have $\widetilde J_0^\pi\geq \widetilde V_0$, and thus
    \begin{align*}
        &\E_{\pi_0}\left[
        \widetilde J_0^\pi(Y_0,Y_0')
        +\beta\Big(\mathrm{KL}_{0}^{\!(1)}(\mu,\pi)+\mathrm{KL}_{0}^{\!(2)}(\mu',\pi)\Big)
        \right] \\
        &\qquad\geq
        \E_{\pi_0}\left[
        \widetilde V_0(Y_0,Y_0')
        +\beta\Big(\mathrm{KL}_{0}^{\!(1)}(\mu,\pi)+\mathrm{KL}_{0}^{\!(2)}(\mu',\pi)\Big)
        \right] \\
        &\qquad\geq
        \inf_{\pi_0\in\mathcal P(S\times S)}
        \E_{\pi_0}\left[
        \widetilde V_0(Y_0,Y_0')
        +\beta\Big(\mathrm{KL}_{0}^{\!(1)}(\mu,\pi)+\mathrm{KL}_{0}^{\!(2)}(\mu',\pi)\Big)
        \right].
    \end{align*}
    The definition of $W_{\mathrm{rel}}(\mu,\mu')$ in \eqref{eq:relaxed-bicausal-ot} as the infimum over $\pi\in\Pi_{\mathrm{rel}}$ yields
    $$
    W_{\mathrm{rel}}(\mu,\mu')
    \geq
    \inf_{\pi_0\in\mathcal P(S\times S)}
    \E_{\pi_0}\left[
    \widetilde V_0(Y_0,Y_0')
    +\beta\Big(\mathrm{KL}_{0}^{\!(1)}(\mu,\pi)+\mathrm{KL}_{0}^{\!(2)}(\mu',\pi)\Big)
    \right].
    $$

    \item For any $\epsilon>0$, there exists $\widetilde\gamma_0^\epsilon\in\mathcal P(S\times S)$ such that
    \begin{align}
    \label{eq:lemma_rel_part2_gamma0}
        &\E_{\widetilde\gamma_0^\epsilon}\left[
        \widetilde V_0(Y_0,Y_0')
        +\beta\Big(\mathrm{KL}_{0}^{\!(1)}(\mu,\widetilde\gamma^\epsilon)+\mathrm{KL}_{0}^{\!(2)}(\mu',\widetilde\gamma^\epsilon)\Big)
        \right] \\
        &\qquad\leq
        \inf_{\pi_0\in\mathcal P(S\times S)}
        \E_{\pi_0}\left[
        \widetilde V_0(Y_0,Y_0')
        +\beta\Big(\mathrm{KL}_{0}^{\!(1)}(\mu,\pi)+\mathrm{KL}_{0}^{\!(2)}(\mu',\pi)\Big)
        \right]
        +\epsilon. \nonumber
    \end{align}
    Using $\pmb{\widetilde Q}^{\epsilon}$ defined in \eqref{eq:lemma_rel_part1_pi_eps}, we construct a relaxed Markov policy
    $$
    \widetilde\gamma^\epsilon:=\pi_{\widetilde\gamma_0^\epsilon,\pmb{\widetilde Q}^{\epsilon}}\in\Pi_{\mathrm{rel}}.
    $$
    Notice that \eqref{eq:lemma_rel_part1_ind_bound} holds for any initial coupling in $\mathcal P(S\times S)$. Therefore, Part I and \eqref{eq:lemma_rel_part1_ind_bound} imply that
    $$
    \widetilde V_0\leq \widetilde J_0^{\widetilde\gamma^\epsilon}<\widetilde V_0+\epsilon,
    $$
    which further gives that
    \begin{align*}
        &\E_{\widetilde\gamma_0^\epsilon}\left[
        \widetilde J_0^{\widetilde\gamma^\epsilon}(Y_0,Y_0')
        +\beta\Big(\mathrm{KL}_{0}^{\!(1)}(\mu,\widetilde\gamma^\epsilon)+\mathrm{KL}_{0}^{\!(2)}(\mu',\widetilde\gamma^\epsilon)\Big)
        \right] \\
        &\qquad\leq
        \E_{\widetilde\gamma_0^\epsilon}\left[
        \widetilde V_0(Y_0,Y_0')
        +\beta\Big(\mathrm{KL}_{0}^{\!(1)}(\mu,\widetilde\gamma^\epsilon)+\mathrm{KL}_{0}^{\!(2)}(\mu',\widetilde\gamma^\epsilon)\Big)
        \right]
        +\epsilon \nonumber\\
        &\qquad\leq
        \inf_{\pi_0\in\mathcal P(S\times S)}
        \E_{\pi_0}\left[
        \widetilde V_0(Y_0,Y_0')
        +\beta\Big(\mathrm{KL}_{0}^{\!(1)}(\mu,\pi)+\mathrm{KL}_{0}^{\!(2)}(\mu',\pi)\Big)
        \right]
        +2\epsilon, \nonumber
    \end{align*}
    where the last inequality follows from \eqref{eq:lemma_rel_part2_gamma0}. Therefore, by the definition of $W_{\mathrm{rel}}(\mu,\mu')$ in \eqref{eq:relaxed-bicausal-ot} as the infimum over $\pi\in\Pi_{\mathrm{rel}}$ and the arbitrarily chosen $\epsilon$, we obtain
    $$
    W_{\mathrm{rel}}(\mu,\mu')
    \leq
    \inf_{\pi_0\in\mathcal P(S\times S)}
    \E_{\pi_0}\left[
    \widetilde V_0(Y_0,Y_0')
    +\beta\Big(\mathrm{KL}_{0}^{\!(1)}(\mu,\pi)+\mathrm{KL}_{0}^{\!(2)}(\mu',\pi)\Big)
    \right].
    $$
\end{enumerate}

Combining the two inequalities in both directions, we conclude that
\begin{align*}
W_{\mathrm{rel}}(\mu,\mu')
=
\inf_{\pi_0\in\mathcal P(S\times S)}
\E_{\pi_0}\left[
\widetilde V_0(Y_0,Y_0')
+\beta\Big(\mathrm{KL}_{0}^{\!(1)}(\mu,\pi)+\mathrm{KL}_{0}^{\!(2)}(\mu',\pi)\Big)
\right]. \tag*{\Halmos}
\end{align*}
\end{proof}

\section{Omitted proofs in Section~\ref{sec: algorithm and convergence}}
\label{sec: omitted proofs in section algorithm and convergence}
In this section, we provide the formal proofs of Theorems~\ref{thm:policy-grad} and~\ref{thm:beta-PG-converge}, Lemma~\ref{lem:unbiased-cov}, Corollary~\ref{cor:converge-largebeta}, and other supporting lemmas used in the proof.

\subsection{Proof of Theorem~\ref{thm:policy-grad}}
\label{subsec: proof of theorem policy grad}

\begin{proof}{Proof.}
In the following proof, we treat $J_{\mathrm{val}}(\pmb\theta)$ in~\eqref{eq:Jval-def} and $J_\mathrm{KL}(\pmb\theta)$ in~\eqref{eq:JKL-def} separately. We consider the case with arbitrary $n \in [N]_0$.

\medskip
\noindent\underline{Step 1: Gradient of $J_{\mathrm{val}}(\pmb\theta)$.}
Fix $n \in [N]_0$. For any function $F(\pmb y, \pmb y';\pmb\theta): \mathcal Y \times \mathcal Y\xmapsto{} \R$ that satisfies the condition that the expectation and differentiation can be interchanged, by direct calculations, it holds that
\begin{equation}
\label{eq:LR-general}
\nabla_{\theta_n}\E_{\pi^{\pmb\theta}}[F(\widehat{\pmb Y}, \widehat{\pmb Y}\mkern-1mu';\pmb\theta)] = \E_{\pi^{\pmb\theta}}\big[F(\widehat{\pmb Y}, \widehat{\pmb Y}\mkern-1mu';\pmb\theta)\nabla_{\theta_n}\log q_n^{\theta_n}(\widehat Y_n,\widehat Y_n' \,\big|\, \widehat Y_{n-1},\widehat Y_{n-1}')\big]+ \E_{\pi^{\pmb\theta}}[\nabla_{\theta_n}F(\widehat{\pmb Y}, \widehat{\pmb Y}\mkern-1mu';\pmb\theta)].
\end{equation}
The identity \eqref{eq:LR-general} converts differentiation of an expectation under $\pi^{\pmb\theta}$ into an expectation involving the score of the parameterized kernel.
Notice that $\sum_{k=0}^{N} c_k$ does not depend on the parameter $\theta_n$, which indicates that $\nabla_{\theta_n} \left[\sum_{k=0}^{N} c_k\right] = 0$. 

Therefore, substituting $F(\pmb y, \pmb y'; \pmb \theta)=\sum_{k=0}^{N} [c_k(y_k, y_k')]$ into the identity \eqref{eq:LR-general}, we obtain
\begin{align*}
    &\nabla_{\theta_n}J_{\mathrm{val}}(\pmb\theta)
    =\E_{\pi^{\pmb\theta}}\left[\left(\sum_{k=0}^{N}c_k(\widehat Y_k,\widehat Y_k')\right)\nabla_{\theta_n}\log q_n^{\theta_n}(\widehat Y_n,\widehat Y_n'  \,|\,   \widehat Y_{n-1},\widehat Y_{n-1}')\right] \\
    &\hspace{7pt}=\E_{\pi^{\pmb\theta}}\left[\E_{\pi^{\pmb\theta}}\left[\left(\sum_{k=0}^{N}c_k(\widehat Y_k,\widehat Y_k')\right)\nabla_{\theta_n}\log q_n^{\theta_n}(\widehat Y_n,\widehat Y_n'  \,\big|\,  \widehat Y_{n-1},\widehat Y_{n-1}') \,\biggl|\, \widehat Y_{n-1}, \widehat Y_{n-1}'\right]\right] \\
    &\hspace{7pt}\stackrel{(i)}{=}\E_{\pi^{\pmb\theta}}\left[\E_{\pi^{\pmb\theta}}\left[\sum_{k=0}^{n-1}c_k(\widehat Y_k,\widehat Y_k') \,\big|\, \widehat Y_{n-1},\widehat Y_{n-1}' \right] \E_{\pi^{\pmb\theta}}\left[\nabla_{\theta_n}\log q_n^{\theta_n}(\widehat Y_n,\widehat Y_n'  \,|\,  \widehat Y_{n-1},\widehat Y_{n-1}') \,\biggl|\, \widehat Y_{n-1}, \widehat Y_{n-1}'\right]\right] \\
    &\qquad \qquad +\E_{\pi^{\pmb\theta}}\left[\left(\sum_{k=n}^{N}c_k(\widehat Y_k,\widehat Y_k')\right)\nabla_{\theta_n}\log q_n^{\theta_n}(\widehat Y_n,\widehat Y_n'  \,|\,  \widehat Y_{n-1},\widehat Y_{n-1}')\right] \\
    &\hspace{7pt}=\E_{\pi^{\pmb\theta}}\left[\left(\sum_{k=n}^{N}c_k(\widehat Y_k,\widehat Y_k')\right)\nabla_{\theta_n}\log q_n^{\theta_n}(\widehat Y_n,\widehat Y_n'  \,\big|\,  \widehat Y_{n-1},\widehat Y_{n-1}')\right], 
\end{align*}
where $(i)$ uses the Markov property of $(\widehat{\pmb Y}, \widehat{\pmb Y}\mkern-1mu')$ that $(\widehat{\pmb Y}_{0:n-1}, \widehat{\pmb Y}\mkern-1mu'_{0:n-1})$ and $(\widehat{\pmb Y}_{n:N}, \widehat{\pmb Y}\mkern-1mu'_{n:N})$ are independent conditioning on $(\widehat Y_{n-1},\widehat Y_{n-1}')$; the last equality follows from 
\begin{align}
\label{eq:zero-score-int}
\E_{\pi^{\pmb\theta}}\bigg[\nabla_{\theta_n}\log q_n^{\theta_n}(\widehat Y_n,\widehat Y_n' \,\big|\, \widehat Y_{n-1},\widehat Y_{n-1}')\bigg] &= \E_{\pi^{\pmb\theta}}\Big[
\E_{\pi^{\pmb\theta}}\big[\nabla_{\theta_n}\log q_n^{\theta_n}(\widehat Y_n,\widehat Y_n' \,\big|\, \widehat Y_{n-1},\widehat Y_{n-1}') \,\big|\, \widehat Y_{n-1},\widehat Y_{n-1}'\big]
\Big] \nonumber \\
&= \E_{\pi^{\pmb\theta}} \Big[ \nabla_{\theta_{n}} \int
q_n^{\theta_{n}}(y_n, y_n'
 \,|\,  \widehat Y_{n-1},\widehat Y_{n-1}')\,\dd y_n \dd y_n' \Big]=0.
\end{align}

\noindent\underline{Step 2: Gradient of $J_{\mathrm{KL}}(\pmb \theta)$.}
For any $n \in [N]$, we note that for $\nabla_{\theta_n}J_{\mathrm{KL}}(\pmb \theta)=\nabla_{\theta_n}\mathrm{KL}_{n}^{\!(1)}(\mu, \pi^{\pmb\theta})+\nabla_{\theta_n}\mathrm{KL}_{n}^{\!(2)}(\mu, \pi^{\pmb\theta})$. Due to symmetry, we only consider $\nabla_{\theta_n}\mathrm{KL}_{n}^{\!(1)}(\mu, \pi^{\pmb\theta})$. Consider random variables $(Y_{n-1}, Y_{n-1}', Y_{n}, Y_{n}') \sim \mu \otimes \mu'$. With Assumptions~\ref{asm:abs_cont_density} and~\ref{asm:exchange_expectation_gradient}, as well as the definition of $\mathrm{KL}_{n}^{\!(1)}$ in \eqref{eq:expectation_kl_n_1}, we have
\begin{align*}
\nabla_{\theta_n} 
\mathrm{KL}_{n}^{\!(1)}\big(\mu, \pi^{\pmb\theta}\big) &= \E_{\mu \otimes \mu'}\Big[\nabla_{\theta_n} \log \zeta_n^{\theta_n,(1)}(Y_n \,|\,  Y_{n-1},Y_{n-1}')\Big] \\
&= -\E_{\mu \otimes \mu'}\Big[\nabla_{\theta_n} \log q_n^{\theta_n,(1)}(Y_n \,|\,  Y_{n-1},Y_{n-1}')\Big],
\end{align*}
where the last identity follows from the definition of $\zeta_n^{\theta_n, (1)}$ in \eqref{eq:def_zeta_0} and \eqref{eq:def_zeta_n}. \hfill\Halmos
\end{proof}

\subsection{Proof of Lemma~\ref{lem:unbiased-cov}}
\label{subsec: proof of lemma unbiased cov}

\begin{proof}{Proof.}
For notational simplicity, denote
$\widehat l_{k,n}^{(b)}:=\widehat l_{k,n}\big(\widehat Y_{n-1}^{(b)},\widehat Y_{n-1}'^{(b)}\big)$ for each $n \in [N]_0$ and when there is no confusion, we omit the input random variable $\widehat Y_n^{(b)}$ and $\widehat Y_n'^{(b)}$.

\medskip
\noindent\underline{Part I: Unbiasedness.}
For any $n \in [N]_0$ and any trajectory index $b$, we have
\begin{align}
&\quad \,\,
\E_{\pi^{\pmb\theta_k}}\bigg[\Big(\widehat V_{k,n}^{(b)}-\widehat l_{k,n}^{(b)}\Big)\widehat S_{k,n}^{(b)}-\beta\widehat K_{k,n}^{(b)} \,\Big|\, \pmb\theta_k\bigg] \nonumber \\
&=\E_{\pi^{\pmb\theta_k}}\bigg[\widehat V_{k,n}^{(b)}\widehat S_{k,n}^{(b)}-\beta\widehat K_{k,n}^{(b)} \,\Big|\, \pmb\theta_k\bigg]-\E_{\pi^{\pmb\theta_k}}\bigg[\widehat l_{k,n}^{(b)}\,\E_{\pi^{\pmb\theta_k}}\Big[\widehat S_{k,n}^{(b)} \,\Big|\,  \widehat Y_{n-1}^{(b)},\widehat Y_{n-1}'^{(b)} \Big] \,\Big|\, \pmb\theta_k\bigg] \label{eq:L_mul_S_zero} \\
&=\E_{\pi^{\pmb\theta_k}}\bigg[\widehat V_{k,n}^{(b)}\widehat S_{k,n}^{(b)}-\beta\widehat K_{k,n}^{(b)} \,\Big|\, \pmb\theta_k\bigg], 
\label{eq:unbiased-L}
\end{align}
where the first identity follows from the fact that conditioning on $(\widehat Y_{n-1}^{(b)},\widehat Y_{n-1}'^{(b)})$, $\widehat l_{k,n}$ is independent of $\widehat Y_{n}^{(b)},\widehat Y_{n}'^{(b)},\dots,\widehat Y_{N},\widehat Y_{N}'^{(b)}$; the second equality follows from the identity \eqref{eq:zero-score-int} with $\theta_{k,n}$. By the definitions of $\widehat V_{k,n}^{(b)}$, $\widehat S_{k,n}^{(b)}$, and $\widehat K_{k,n}^{(b)}$ for $i=1,2$ in \eqref{eq:Vhat-def}--\eqref{eq:Khat-def}, using the expectation representation of $\nabla_{\theta_{k,n}}J_{\mathrm{val}}$ and $\nabla_{\theta_{k,n}}J_{\mathrm{KL}}$ in \eqref{eq:grad-Jval-n}--\eqref{eq:grad-JKL-n} with $\theta_n=\theta_{k,n}$, we derive
\begin{equation}
\label{eq:unbiased-VK}
\E_{\pi^{\pmb\theta_k}}\bigg[\widehat V_{k,n}^{(b)}\widehat S_{k,n}^{(b)}-\beta\widehat K_{k,n}^{(b)} \,\Big|\, \pmb\theta_k\bigg] = \nabla_{\theta_{k,n}}J_{\mathrm{val}}(\pmb\theta_k) + \beta\nabla_{\theta_{k,n}}J_{\mathrm{KL}}(\pmb\theta_k).
\end{equation}
Immediately, by the definitions of $\widehat G_{k,\mathrm{val},n}$ and $\widehat G_{k,\mathrm{KL},n}$ in \eqref{eq:agg-Gval-GKL-def}, combining the identities \eqref{eq:unbiased-L} with \eqref{eq:unbiased-VK} and summing over $b \in [B]$, we obtain
\begin{equation}
\label{eq:unbiased-Gval-GKL-n}
\E_{\pi^{\pmb\theta_k}}\big[\widehat G_{k,\mathrm{val},n} + \beta\widehat G_{k,\mathrm{KL},n} \,\big|\, \pmb\theta_k\big]=\,\nabla_{\theta_{k,n}}J_{\mathrm{val}}(\pmb\theta_k)+\,\beta\,\nabla_{\theta_{k,n}}J_{\mathrm{KL}}(\pmb\theta_k).
\end{equation} 
Summing \eqref{eq:unbiased-Gval-GKL-n} over $n \in [N]_0$, we get the desired results in \eqref{eq:unbiased-gk0}.

Recall the definition of $\widehat g_k$ in \eqref{eq:gkhat-def}. Stacking \eqref{eq:unbiased-Gval-GKL-n} over $n$, and using the decomposition $\nabla J_\beta=\nabla J_{\mathrm{val}}+\beta\,\nabla J_{\mathrm{KL}}$ in \eqref{eq:Jbeta-decom} of Theorem~\ref{thm:policy-grad}, we obtain the identity \eqref{eq:unbiased-gk} for $\widehat g_k$.

\medskip
\noindent\underline{Part II: PSD covariance improvement.}
Fix $n \in [N]_0$ and $b \in [B]$. Denote the $n$-th element of $\widehat g_k$ by
$$
\widehat g_{k,n} =\ \widehat G_{k,\mathrm{val},n}+\ \beta\,\widehat G_{k,\mathrm{KL},n} -\ \widehat L_{k,n},
$$
and the corresponding $n$-th element of $\widehat g_k^{0}$ by
$$
\widehat g_{k,n}^{0} =\ \widehat G_{k,\mathrm{val},n}+\ \beta\, \widehat G_{k,\mathrm{KL},n}.
$$
For any deterministic vector $u \in \mathbb{R}^d$ with the same dimension as $\widehat S_{k,n}^{(b)}$, by the formula of $\widehat G_{k,\mathrm{val},n}$, $\widehat G_{k,\mathrm{KL},n}$, and $\widehat L_{k,n}$ in \eqref{eq:Gval-GKL-n-def} and \eqref{eq:agg-control-variate-def}, the variance ($\operatorname{Var}$) of $u^\top \widehat g_{k,n}$ satisfies
\begin{align}
\operatorname{Var}\Big[u^\top \widehat g_{k,n} \,\big|\, \pmb\theta_k\Big] &= \frac{1}{B^2}\sum_{b=1}^{B}\operatorname{Var}\Big[u^\top \Big(\big(\widehat V_{k,n}^{(b)}-\widehat l_{k,n}^{(b)}\big)\widehat S_{k,n}^{(b)}-\beta\widehat K_{k,n}^{(b)}\Big) \,\big|\, \pmb\theta_k\Big] \label{eq:var-gkn} \\
\operatorname{Var}\Big[u^\top \widehat g_{k,n}^{0} \,\big|\, \pmb\theta_k\Big] &= \frac{1}{B^2}\sum_{b=1}^{B}\operatorname{Var}\Big[u^\top \Big(\widehat V_{k,n}^{(b)}\widehat S_{k,n}^{(b)}-\beta\widehat K_{k,n}^{(b)}\Big) \,\big|\, \pmb\theta_k\Big] \label{eq:var-gkn0}.
\end{align}
Here, we use the fact that the trajectories are i.i.d. across $b=1,\dots,B$. Using the independence between trajectories $(\widehat{\pmb Y}^{(b)},\widehat{\pmb Y}'^{(b)})$ and $(\pmb Y^{(b)},\pmb Y'^{(b)})$, and incorporating the identities \eqref{eq:unbiased-L}--\eqref{eq:unbiased-Gval-GKL-n} into \eqref{eq:var-gkn} and \eqref{eq:var-gkn0}, we obtain
\begin{align}
\E_{\pi^{\pmb\theta_k}}\Big[u^\top \widehat g_{k,n} \,\big|\, \pmb\theta_k\Big]
&= \E_{\pi^{\pmb\theta_k}}\Big[u^\top \widehat g_{k,n}^{0} \,\big|\, \pmb\theta_k\Big],\quad \textrm{i.e.,} \quad \E_{\pi^{\pmb\theta_k}}\Big[\widehat g_k \,\big|\, \pmb\theta_k \Big] = \E_{\pi^{\pmb\theta_k}}\Big[\widehat g_k^{0} \, \big| \, \pmb\theta_k \Big]. \label{eq:expectation_equal}
\end{align}
Notice that the population weighted least-squares
\begin{align*}
&\widehat w_{k,n}\in\arg\min_{w}\ \frac{1}{B}\sum_{b=1}^{B} \E_{\pi^{\pmb\theta_k}} \Big\|\big(\widetilde V_{k,n}^{(b)}-w^\top\phi_{k,n}(\widetilde Y_{n-1}^{(b)},\widetilde Y_{n-1}'^{(b)})\big)\widetilde S_{k,n}^{(b)}\Big\|^2, \label{eq:wls-val}
\end{align*}
admits an optimal solution as
\begin{align*}
l_{k,n}^{*}(\widetilde Y_{n-1}^{(b)},\widetilde Y_{n-1}'^{(b)}) =  \frac{\sum_{b=1}^{B}\E\Big[\widetilde V_{k,0}^{(b)}\big\|\widetilde S_{k,0}^{(b)}\big\|_2^2 \,\Big|\, \widetilde Y_{n-1}^{(b)},\widetilde Y_{n-1}'^{(b)}\Big]}{\sum_{b=1}^{B}\E\Big[\big\|\widetilde S_{k,0}^{(b)}\big\|_2^2 \,\Big|\, \widetilde Y_{n-1}^{(b)},\widetilde Y_{n-1}'^{(b)}\Big]}.
\end{align*}
Using the fact that $(\widehat {\pmb Y}^{(b)}, \widehat {\pmb Y}'^{(b)})$ and $(\widetilde {\pmb Y}^{(b)}, \widetilde {\pmb Y}'^{(b)})$ are identically distributed, we obtain that the estimator $\widehat l_{k,n}$ is consistent for $l_{k,n}^\star$ in the weighted $L^2$-distance induced by \(\|\widehat S_{k,n}\|_2^2\), namely,
\begin{equation}
\label{eq:weighted-l2-consistency}
\lim_{B \to \infty}\mathbb E\Big[
\big(
\widehat l_{k,n}-l_{k,n}^\star
\big)^2
\|\widehat S_{k,n}\|_2^2
\Big] = 0.
\end{equation}
Therefore, combining the identities \eqref{eq:var-gkn}--\eqref{eq:expectation_equal} with the asymptotic convergence of $\widehat l_{k,n}$ defined in \eqref{eq:weighted-l2-consistency}, we obtain
\begin{equation*}
\operatorname{Var}\Big[u^\top \widehat g_{k,n} \,\big|\, \pmb\theta_k\Big] \leq \operatorname{Var}\Big[u^\top \widehat g_{k,n}^{0} \,\big|\, \pmb\theta_k\Big], \quad \textrm{ as } B \to \infty.
\end{equation*}
Then, the arbitrariness $u$ gives the variance induction of $\widehat g_{k,n}$.

Finally, we write the covariance of $\widehat g_k$ in block form
\begin{align}
\label{eq:cov-gk}
\operatorname{Cov}\big(\widehat g_k \,\big|\, \pmb\theta_k \big) = \begin{bmatrix}
\operatorname{Cov}\big(\widehat g_{k,0} \,\big|\, \pmb\theta_k \big) & \cdots & \operatorname{Cov}\big(\widehat g_{k,0},\widehat g_{k,N} \,\big|\, \pmb\theta_k \big)\\
\vdots & \ddots & \vdots\\
\operatorname{Cov}\big(\widehat g_{k,N},\widehat g_{k,0} \,\big|\, \pmb\theta_k \big) & \cdots & \operatorname{Cov}\big(\widehat g_{k,N} \,\big|\, \pmb\theta_k \big)
\end{bmatrix}.
\end{align}
Notice that for each $n \in [N]_0$, \eqref{eq:L_mul_S_zero} gives
$$
\E\Big[\widehat l_{k,n}^{(b)} \widehat S_{k,n}^{(b)} \,\big|\, \widehat Y_{n-1}^{(b)}, \widehat Y_{n-1}'^{(b)}\Big] = 0,
$$
which further indicates that
\begin{equation}
\begin{split}
\label{eq:cond_expectation_equal}
    &\quad \,\, \E_{\pi^{\pmb\theta_k}}\Big[\widehat g_{k,n} \,\big|\, \pmb\theta_k, \widehat Y_{n-1}^{(1)}, \widehat Y_{n-1}'^{(1)},\widetilde Y_{n-1}^{(1)}, \widetilde Y_{n-1}'^{(1)},\dots, \widehat Y_{n-1}^{(B)}, \widehat Y_{n-1}'^{(B)}, \widetilde Y_{n-1}^{(B)}, \widetilde Y_{n-1}'^{(B)} \Big] \\
    &= \E_{\pi^{\pmb\theta_k}}\Big[\widehat g_{k,n}^{0} \,\big|\, \pmb\theta_k, \widehat Y_{n-1}^{(1)}, \widehat Y_{n-1}'^{(1)},\widetilde Y_{n-1}^{(1)}, \widetilde Y_{n-1}'^{(1)},\dots, \widehat Y_{n-1}^{(B)}, \widehat Y_{n-1}'^{(B)}, \widetilde Y_{n-1}^{(B)}, \widetilde Y_{n-1}'^{(B)} \Big].
\end{split}
\end{equation}
Then, for any $n \neq m$, using the tower property and the identities \eqref{eq:expectation_equal} and \eqref{eq:cond_expectation_equal}, we have
\begin{equation}
\label{eq:cov_nondiag_equal}
    \operatorname{Cov}\big(\widehat g_{k,n},\widehat g_{k,m} \,\big|\, \pmb\theta_k \big) \, =\, \operatorname{Cov}\big(\widehat g_{k,n}^{0},\widehat g_{k,m}^{0} \,\big|\, \pmb\theta_k \big).
\end{equation}
Therefore, using the formula of the covariance in \eqref{eq:cov-gk} (replacing each $\widehat g_{k,n}$ with $\widehat g_{k,n}^{0}$ to obtain $\operatorname{Cov}(\widehat g_k^{0} \,|\, \pmb\theta_k )$) and the property of the non-diag elements in \eqref{eq:cov_nondiag_equal} we obtain
\begin{align*}
&\quad\,\, \operatorname{Cov}\big(\widehat g_k \,\big|\, \pmb\theta_k \big)-\operatorname{Cov}\big(\widehat g_k^{0} \,\big|\, \pmb\theta_k \big) \\
&=\operatorname{diag}\Big(\operatorname{Cov}(\widehat g_{k,0} \,\big|\, \pmb\theta_k \big)-\operatorname{Cov}\big(\widehat g_{k,0}^{0} \,\big|\, \pmb\theta_k \big), \dots, \operatorname{Cov}\big(\widehat g_{k,N} \,\big|\, \pmb\theta_k \big)-\operatorname{Cov}\big(\widehat g_{k,N}^{0} \,\big|\, \pmb\theta_k \big)\Big) \preceq 0. \tag*{\Halmos}
\end{align*}

\end{proof}

\subsection{Proof of Theorem~\ref{thm:beta-PG-converge}}
\label{subsec: proof of theorem beta pg converge}

\begin{proof}{Proof.}
We first prove the average regret bound and then the last-iteration regret bound.

\noindent\underline{Part I: Average regret bound.}
For each $k \geq 1$, taking expectation with respect to $\pmb\theta_k$ in \eqref{eq:onestep_PG} of Lemma~\ref{lem:onestep-PG}, we have
\begin{align*}
    \mathrm{Err}_{k+1}(\beta) \leq \big(1-\mu(\beta)\eta_k\big)\,\mathrm{Err}_k(\beta)\ +\ \frac{L(\beta)\,\sigma^2(\beta)}{2B}\,\eta_k^2.
\end{align*}
Substituting $a_k = \mathrm{Err}_k$, $\rho_k = \mu(\beta)\eta_k$, and $b_k = L(\beta)\,\sigma^2(\beta)\,\eta_k^2 / (2B)$ into the inequality \eqref{eq:gronwall-ineq} in Lemma~\ref{lem:gronwall-err}, we obtain
\begin{align}
    \mathrm{Err}_{k+1}(\beta) &\leq \Big(\prod_{s=1}^{k}(1-\mu(\beta)\eta_s)\Big)\mathrm{Err}_1(\beta)\ +\ \frac{L(\beta)\,\sigma^2(\beta)}{2B}\sum_{t=1}^{k}\Big(\eta_t^2\prod_{s=t+1}^{k}(1-\mu(\beta)\eta_s)\Big). \label{eq:bound-k-regret}
\end{align}
Averaging the inequality \eqref{eq:bound-k-regret} over $k \in [K-1]_0$ and rearranging the terms yields
\begin{equation*}
\label{eq:bound-avg-reg}
\overline{\mathrm{Reg}}_K(\beta) \leq
\frac{\mathrm{Err}_1(\beta)}{K}\sum_{k=1}^{K}\prod_{s=1}^{k-1}\big(1-\mu(\beta)\eta_s\big)
\ +\ \frac{L(\beta)\sigma^2(\beta)}{2 B K}\sum_{k=1}^{K-1}\eta_k^2\sum_{t=k}^{K-1}\prod_{s=k+1}^{t}\big(1-\mu(\beta)\eta_s\big).
\end{equation*}
Immediately, we obtain the desired result in \eqref{eq:bound-avg-regret-monotone}.

\medskip
\noindent\underline{Part II: Last–iterate proximity.}
By the decomposition of $J_\beta(\pmb\theta)$ in \eqref{eq:Jbeta-decom}, for any $\pmb\theta \in \mathcal A_{\mathrm{KL}}$, we have
\begin{equation}
\label{eq:prox-sandwich}
J_\beta(\pmb\theta)-\inf_{\pmb\theta\in\mathcal A_{\mathrm{KL}}}J_{\mathrm{val}}(\pmb\theta)\ \leq J_\beta(\pmb\theta)-J_\beta^{*}.
\end{equation}
Applying the inequalities $1-\mu(\beta)\eta_s \leq \exp(-\mu(\beta)\eta_s)$ to \eqref{eq:bound-k-regret} with $k=K-1$ and replacing the summation index $s$ with $k$, we obtain
\begin{equation}
\label{eq:bound-K-regret}
    \mathrm{Err}_K(\beta) \leq \Big(\prod_{k=1}^{K}(1-\mu(\beta)\eta_k)\Big) \cdot \mathrm{Err}_1(\beta)\ +\ \frac{L(\beta)\,\sigma^2(\beta)}{2B}\sum_{k=1}^{K-1}\Big(\eta_k^2 \big(\prod_{t=k+1}^{K-1}(1-\mu(\beta)\eta_t)\big) \Big).
\end{equation}
Combining \eqref{eq:bound-K-regret} with \eqref{eq:prox-sandwich} yields \eqref{eq:bound-last-iterate-prox}. \hfill\Halmos
\end{proof}

\subsection{Proof of Corollary~\ref{cor:converge-largebeta}}
\label{subsec: proof of corollary converge-largebeta}

\begin{proof}{Proof.}
By the definition of $K_0(\beta,\epsilon)$ and $B \geq B_0(\beta,\epsilon,K)$ in \eqref{eq:K0-choice} and \eqref{eq:B0-choice}, we have
$$
\exp\Big(-\mu(\beta)\sum_{k=1}^{K-1}\eta_k\Big) \cdot \mathrm{Err}_1(\beta) \leq \frac{\epsilon}{2},\ \textrm{ and }\ \ \frac{L(\beta)\,\sigma(\beta)^2}{2B}\,\sum_{k=1}^{K}\eta_k^2 \leq \frac{\epsilon}{2}.
$$
Meanwhile, notice that $\exp\big(\sum_{t=k+1}^{K-1}-\mu(\beta)\eta_t\big) \leq 1$. Combining the above results with the inequality \eqref{eq:bound-last-iterate-prox} in Theorem~\ref{thm:beta-PG-converge}, we obtain the desired result.
\hfill\Halmos
\end{proof}

\subsection{Supporting lemmas and proofs}
\label{subsec: supporting lemmas and proofs in section algorithm and convergence}

The proof of Theorem~\ref{thm:beta-PG-converge} relies on the following technical Lemmas \ref{lem:beta-noise-gk}--\ref{lem:gronwall-err}. The first two lemmas quantify the second moment of $\widehat g_k$ in \eqref{eq:gkhat-def} and the one-step error for the policy gradient iterate.
\begin{lemma}
\label{lem:beta-noise-gk}
Suppose that Assumption~\ref{asm:beta-smooth-PL-var} holds. Then, for any $\beta \geq 0$, we have
\begin{equation}
\label{eq:sigma-beta-gk}
\E_{\pi^{\pmb\theta_k}}\Big[\big\|\widehat g_k-\nabla J_{\beta}(\pmb\theta_k)\big\|^2\,\big|\,\pmb\theta_k\Big]
\leq \frac{\big(\sigma_{\mathrm{val}}+\beta\,\sigma_{\mathrm{KL}}\big)^2}{B}.
\end{equation}
\end{lemma}

\begin{proof}{Proof.}
Inspired by the variance reduction in \eqref{eq:cov-composite-psd} of Lemma~\ref{lem:unbiased-cov}, it suffices to bound the variance of
$$
\widehat g_k^{0} = \frac{1}{B}\widehat G_{k,\mathrm{val}} + \frac{\beta}{B}\widehat G_{k,\mathrm{KL}}.
$$
Using the representation of $\nabla J_{\beta}(\pmb\theta_k)$ in \eqref{eq:Jbeta-grad-decom} and using the unbiasedness of $\widehat g_{k}^{0}$ in \eqref{eq:unbiased-gk0}--\eqref{eq:unbiased-gk} and \eqref{eq:expectation_equal}, we deduce
\begin{align}
\label{eq:gk-subtract-nablaJ}
\widehat g_k^{0}-\nabla J_\beta(\pmb\theta_k)
=\frac{1}{B}\Big(\widehat G_{k,\mathrm{val}}-\E[\widehat G_{k,\mathrm{val}}\mid\pmb\theta_k]\Big)
+\frac{\beta}{B}\Big(\widehat G_{k,\mathrm{KL}}-\E[\widehat G_{k,\mathrm{KL}}\mid\pmb\theta_k]\Big).
\end{align}
Using the inequality \eqref{eq:cov-composite-psd} and applying the conditional Minkowski inequality to \eqref{eq:gk-subtract-nablaJ}, we have
\begin{align}
    &\quad \,\, \E_{\pi^{\pmb\theta_k}}\Big[\big\|\widehat g_k-\nabla J_\beta(\pmb\theta_k)\big\|^2 \,\big|\, \pmb\theta_k\Big] \nonumber \\
    &\leq \E_{\pi^{\pmb\theta_k}}\Big[\big\|\widehat g_k^{0}-\nabla J_\beta(\pmb\theta_k)\big\|^2 \,\big|\, \pmb\theta_k\Big] \nonumber \\
    &\leq \Bigg[\frac{\Big(\E_{\pi^{\pmb\theta_k}}\Big[\Big\|\widehat G_{k,\mathrm{val}}-\E\big[\widehat G_{k,\mathrm{val}} \,\big|\, \pmb\theta_k\big]\Big\|^2 \,\big|\, \pmb\theta_k\Big]\Big)^{\frac{1}{2}}}{B} + \frac{\beta\Big(\E_{\pi^{\pmb\theta_k}}\Big[\Big\|\widehat G_{k,\mathrm{KL}}-\E\big[\widehat G_{k,\mathrm{KL}} \,\big|\, \pmb\theta_k\big]\Big\|^2 \,\big|\, \pmb\theta_k\Big]\Big)^{\frac{1}{2}}}{B} \Bigg]^2 \nonumber \\
    &\leq \bigg[\frac{1}{B}\Big(\sigma_{\mathrm{val}}\sqrt{B}+\beta\,\sigma_{\mathrm{KL}}\sqrt{B}\Big)\bigg]^2 = \frac{\big(\sigma_{\mathrm{val}}+\beta\,\sigma_{\mathrm{KL}}\big)^2}{B}, \nonumber
\end{align}
where the last inequality follows from the bounds \eqref{eq:component-variance-bounds-G} in Assumption~\ref{asm:beta-smooth-PL-var}. \hfill\Halmos
\end{proof}

\begin{lemma}[One-step error]
\label{lem:onestep-PG}
Suppose that Assumption~\ref{asm:beta-smooth-PL-var} holds and step sizes $\{\eta_k\}_{k \ge 1}$ satisfy $\eta_k \le \min\{1/L(\beta), 1/\mu(\beta)\}$ for each iteration $k \ge 1$. Denote $\sigma(\beta) =\sigma_{\mathrm{val}}+\beta\,\sigma_{\mathrm{KL}}$.
Then, the update $\pmb\theta_{k+1}=\pmb\theta_k - \eta_k \widehat g_k$ yields the following result.
\begin{equation}
\label{eq:onestep_PG}
    \E_{\pi^{\pmb\theta_k}}\Big[J_\beta(\pmb\theta_{k+1})-J_\beta^{*}\,\big|\,\pmb\theta_k\Big]
    \leq (1-\mu(\beta)\eta_k)\,\big(J_\beta(\pmb\theta_k)-J_\beta^{*}\big)\ +\ \frac{L(\beta)\,\sigma^2(\beta)\,\eta_k^2}{2B}.
\end{equation}
\end{lemma}

\begin{proof}{Proof.}
By $L(\beta)$-smoothness \eqref{eq:asm-smooth} in Assumption~\ref{asm:beta-smooth-PL-var}, we have
$$
J_\beta(\pmb\theta_{k+1})
\leq J_\beta(\pmb\theta_k)
-\eta_k\langle \nabla J_\beta(\pmb\theta_k),\widehat g_k\rangle
+\frac{L(\beta)}{2}\eta_k^2\|\widehat g_k\|^2.
$$
Taking conditional expectation and using unbiasedness of $\widehat g_k$ in \eqref{eq:unbiased-gk} yields
\begin{align}
    \E_{\pi^{\pmb\theta_k}}\Big[J_\beta(\pmb\theta_{k+1})\,\big|\,\pmb\theta_k\Big] &\leq J_\beta(\pmb\theta_k)
    -\Big(\eta_k-\frac{L(\beta)}{2}\eta_k^2\Big)\big\|\nabla J_\beta(\pmb\theta_k)\big\|^2
    +\frac{L(\beta)}{2}\eta_k^2\,\E_{\pi^{\pmb\theta_k}}\Big[\big\|\widehat g_k-\nabla J_\beta(\pmb\theta_k)\big\|^2\,\big|\,\pmb\theta_k\Big]\nonumber \\
    &\stackrel{(i)}{\leq} J_\beta(\pmb\theta_k)
    -\mu(\beta)\eta_k\big\|\nabla J_\beta(\pmb\theta_k)\big\|^2
    +\frac{L(\beta)}{2}\eta_k^2\,\E\Big[\big\|\widehat g_k-\nabla J_\beta(\pmb\theta_k)\big\|^2\,\big|\,\pmb\theta_k\Big]\nonumber \\
    &\leq J_\beta(\pmb\theta_k)
    -\mu(\beta)\eta_k\big(J_\beta(\pmb\theta_k)-J_\beta^{*}\big)
    +\frac{L(\beta)\sigma^2(\beta)\eta_k^2}{2B}, \label{eq:temp-onestep-error}
\end{align}
where $(i)$ uses $\eta_k - L(\beta)\eta_k^2 / 2 \geq \eta_k / 2$ for $\eta_k\le 1/L(\beta)$ and the last inequality follows from the PL condition \eqref{eq:asm-PL-condition} in Assumption~\ref{asm:beta-smooth-PL-var} and bounded second moments of $\widehat g_k$ \eqref{eq:sigma-beta-gk} in Lemma \ref{lem:beta-noise-gk}. 

Finally, subtracting $J_\beta^{*}$ from \eqref{eq:temp-onestep-error} and combining the terms, we obtain the desired result. \hfill\Halmos
\end{proof}

We also use the following Gr\"onwall inequality to control the iteration error via recursions.

\begin{lemma}[Discrete Grönwall inequality with tail products]
\label{lem:gronwall-err}
Let $\{a_k\}_{k\ge1}$, $\{\rho_k\}_{k\ge1}$, and $\{b_k\}_{k\ge1}$ be nonnegative sequences with $\rho_k\in[0,1]$ for all $k$. Suppose the following one–step recursion holds
\begin{equation}
\label{eq:gronwall-onestep}
a_{k+1}\, \le\, (1-\rho_k)\,a_k\, +\, b_k,\qquad \forall\,k \geq 1.
\end{equation}
For any $t, k \ge 1$, denote tail products 
\begin{equation}
\label{eq:def_s_t_K}
s_{t:K}=\prod_{s=t}^{K}(1-\rho_s)\quad \textrm{if} \quad t \le K, \quad \textrm{else} \quad 1.
\end{equation} Then, for every $K \geq 1$, we have
\begin{align}
\label{eq:gronwall-ineq}
a_{K+1}\, \le\, s_{1:K}\,a_1\, +\, \sum_{t=1}^{K} b_t\,s_{t+1:K}.
\end{align}
\end{lemma}

\begin{proof}{Proof.}
We prove \eqref{eq:gronwall-ineq} by induction on $k$. For the base case $k=1$, it is trivial by \eqref{eq:gronwall-onestep}

Now, given $a_{k+1}\leq s_{1:k}a_1+\sum_{t=1}^{k} b_t\,s_{t+1:k}$ for some $k \geq 1$, then by induction, we have
\begin{align*}
    a_{k+2} &\leq (1-\rho_{k+1})a_{k+1}+b_{k+1} \\
    &\leq (1-\rho_{k+1})\Big(s_{1:k}a_1+\sum_{t=1}^{k} b_t\,s_{t+1:k}\Big)+b_{k+1} \\
    &= (1-\rho_{k+1})s_{1:k}a_1  +\Big(b_{k+1}+(1-\rho_{k+1})\sum_{t=1}^{k} b_t\,s_{t+1:k}\Big) \\
    &= s_{1:(k+1)}a_1 + \sum_{t=1}^{k+1} b_t\,s_{t+1:k+1},
\end{align*}
where the last equality follows from the recursive definitions of $s_{t:k+1}$ in \eqref{eq:def_s_t_K}. \hfill\Halmos
\end{proof}

\section{Additional results of synthetic experiments}
\label{sec: additional results of synthetic experiments}

For completeness, we report an additional synthetic setting with lower dimension and a shorter horizon, namely $(d=2, N=3)$. This case leads to the same qualitative conclusion as in the main text: the learned generator matches the ground-truth marginals closely and preserves the adjacent-step dependence structure in both unimodal and bimodal regimes.

The generated marginals align closely with the ground truth across all time steps. Figures~\ref{fig:unimodal_d2_N3_distributions} and \ref{fig:bimodal_d2_N3_distributions} compare the empirical marginal distributions at times $n\in\{0,1,2,3\}$ obtained from the learned normalizing-flow generator with those from direct sampling of the ground truth, showing that the method captures both Gaussian and non-Gaussian marginal features in this simpler setting.

\begin{figure}[!ht]
    \caption{Marginal distributions across different time steps for the unimodal setting ($d=2$, $N=3$).
    \label{fig:unimodal_d2_N3_distributions}}
    \centering
    \subfigure[$n=0$]{
        \includegraphics[width=0.48\textwidth]{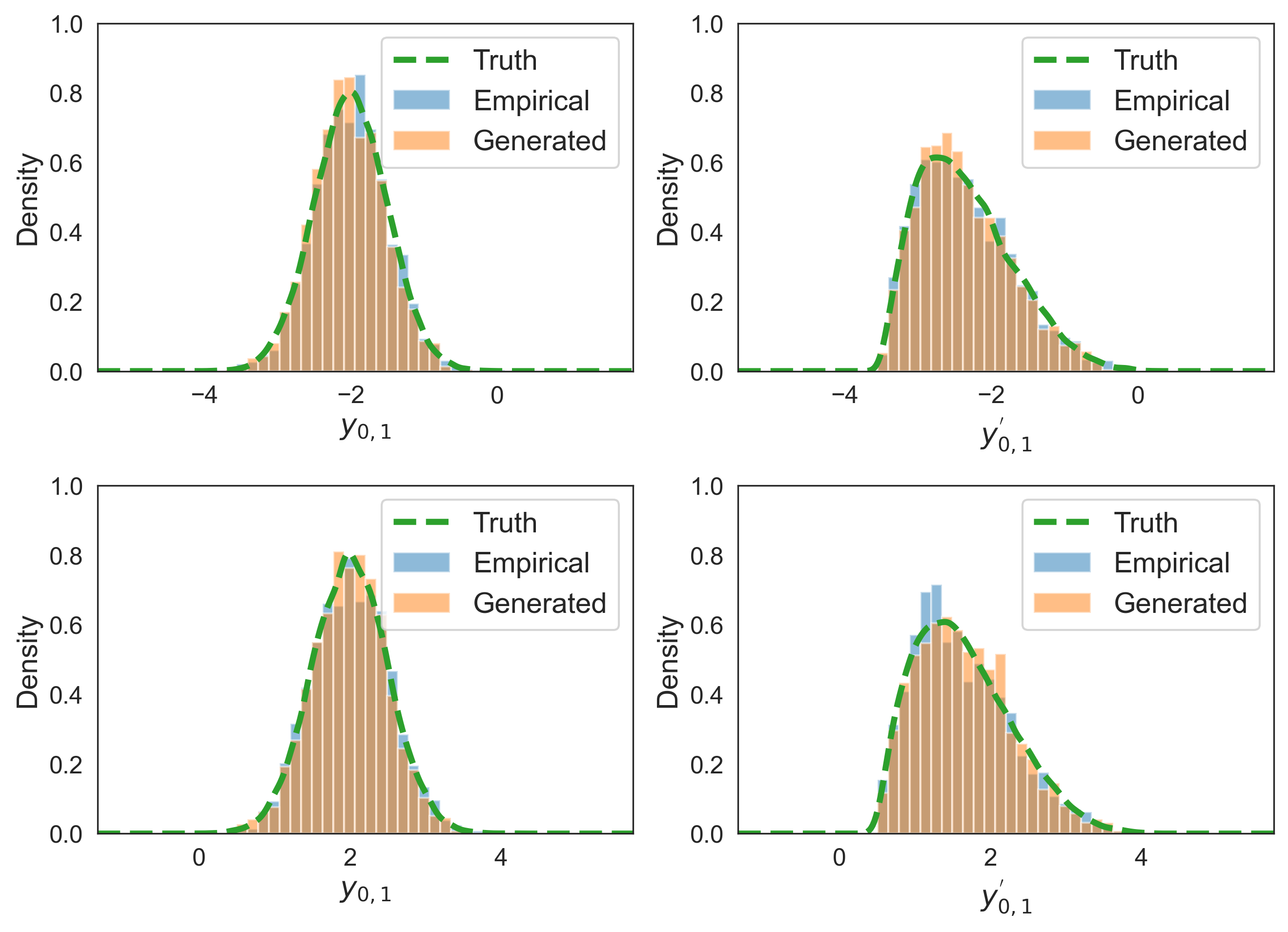}
    }
    \hfill
    \subfigure[$n=1$]{
        \includegraphics[width=0.48\textwidth]{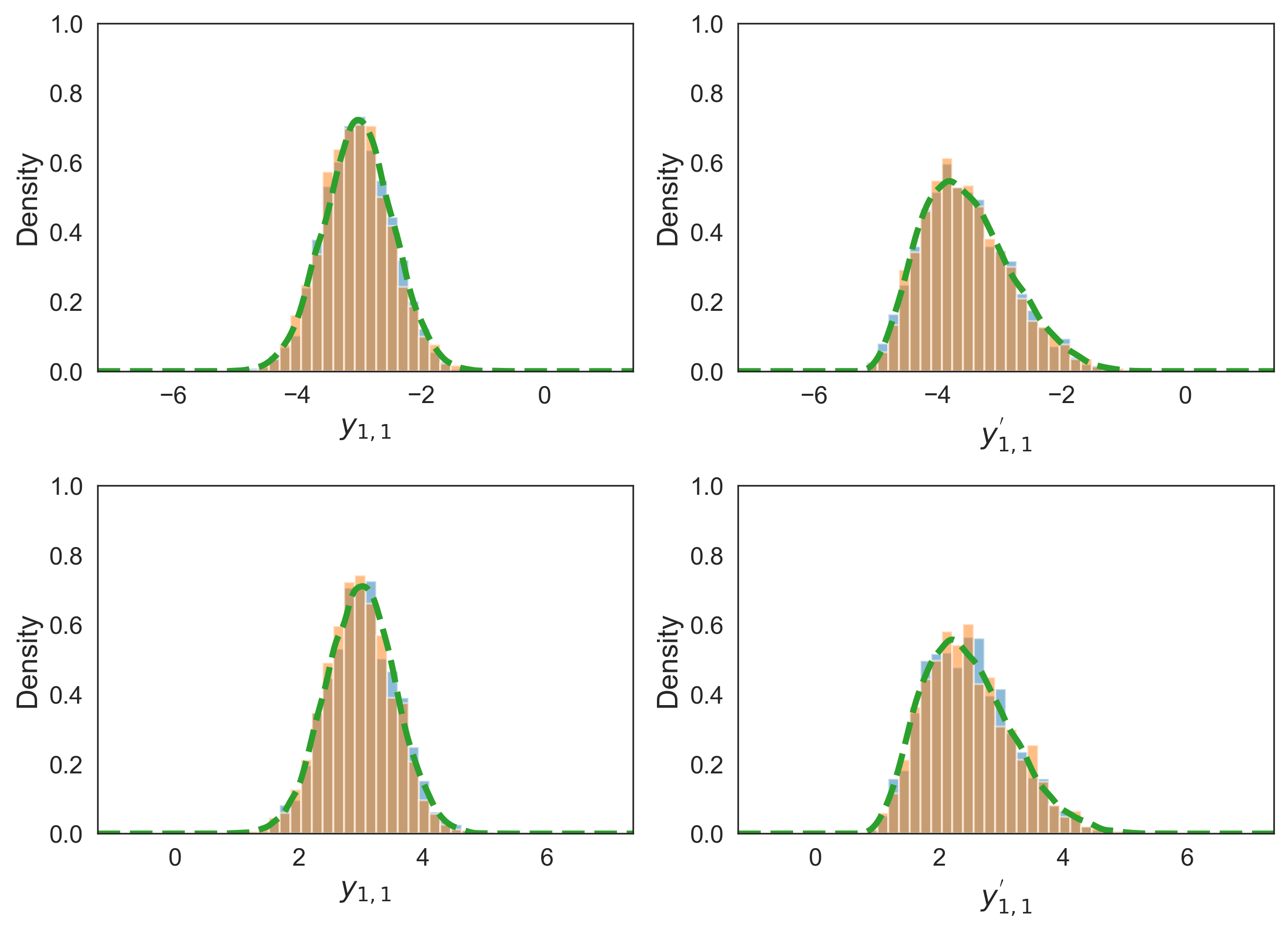}
    }
    
    \subfigure[$n=2$]{
        \includegraphics[width=0.48\textwidth]{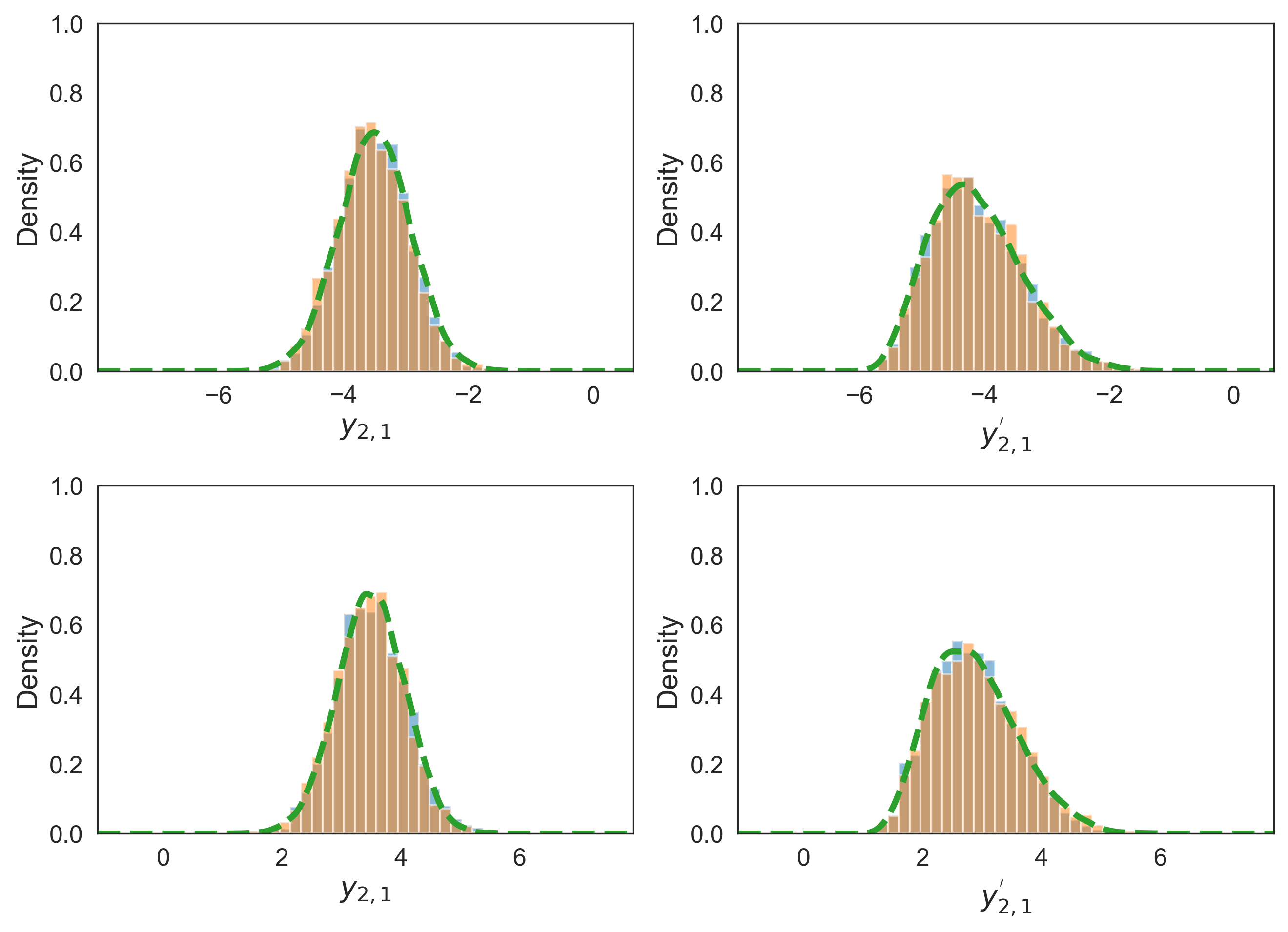}
    }
    \hfill
    \subfigure[$n=3$]{
        \includegraphics[width=0.48\textwidth]{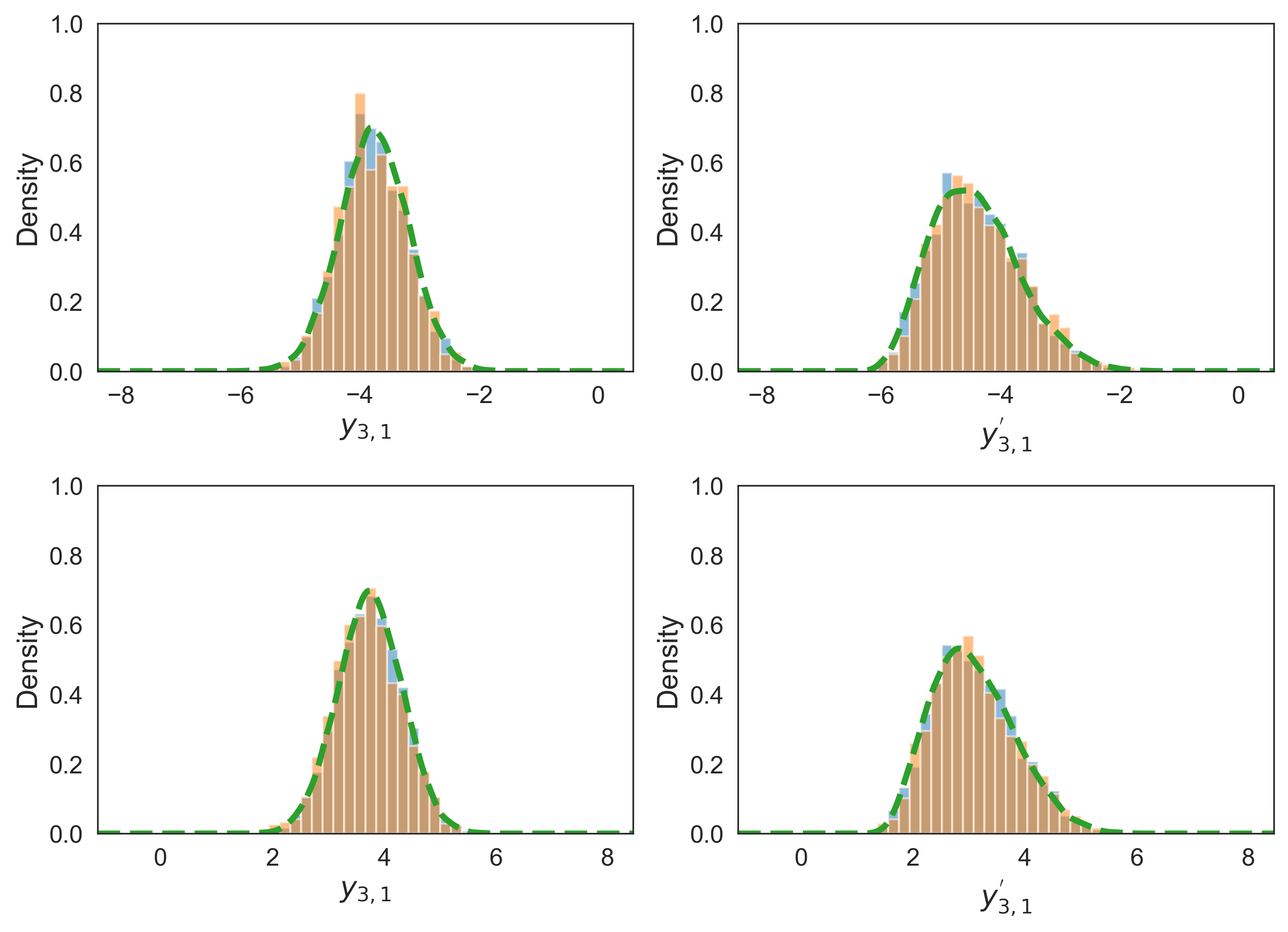}
    }
\end{figure}

\begin{figure}[!ht]
    \centering
    \caption{Marginal distributions across different time steps for the bimodal setting ($d=2$, $N=3$).
    \label{fig:bimodal_d2_N3_distributions}}
    \subfigure[$n=0$]{
        \includegraphics[width=0.48\textwidth]{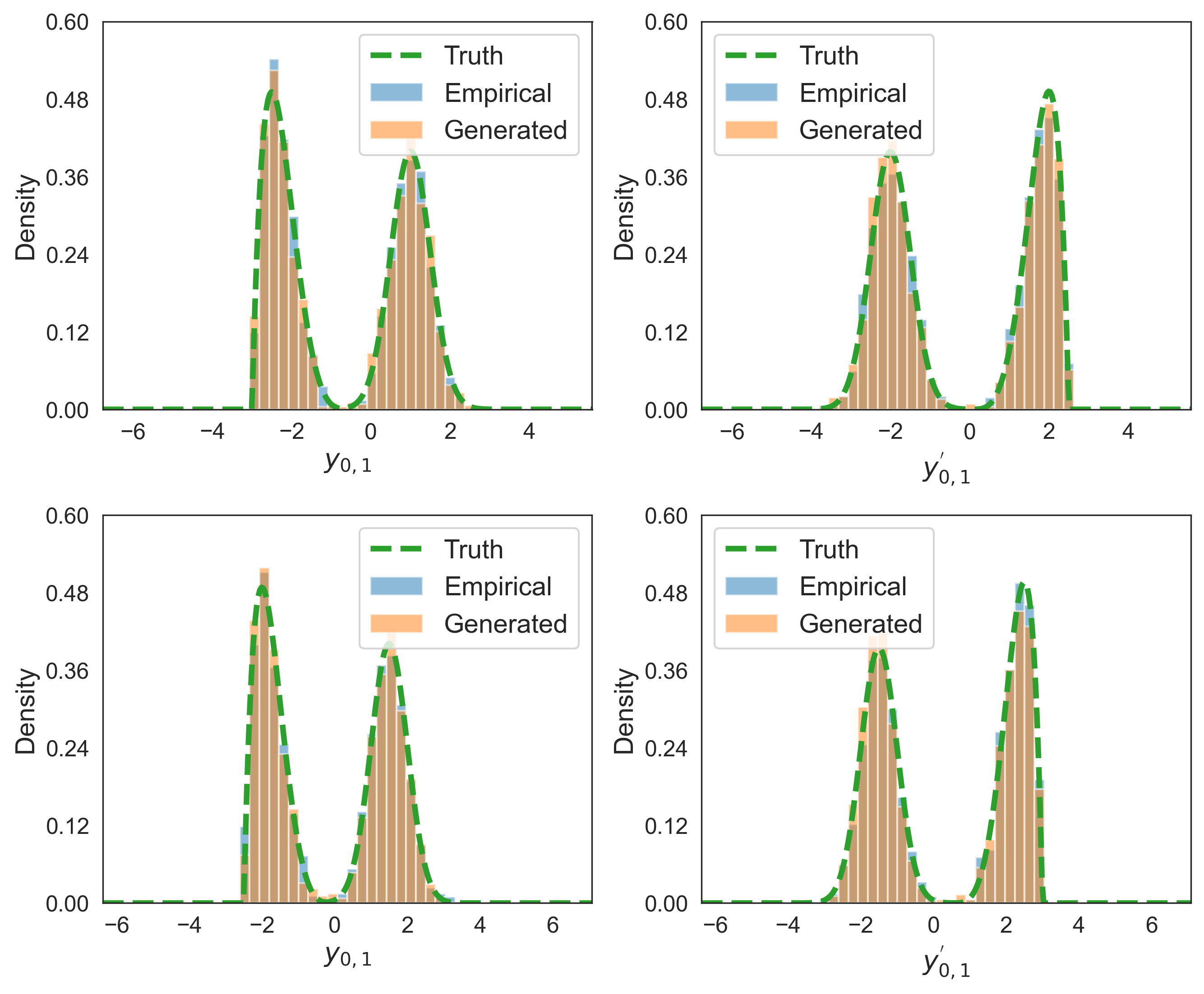}
    }
    \hfill
    \subfigure[$n=1$]{
        \includegraphics[width=0.48\textwidth]{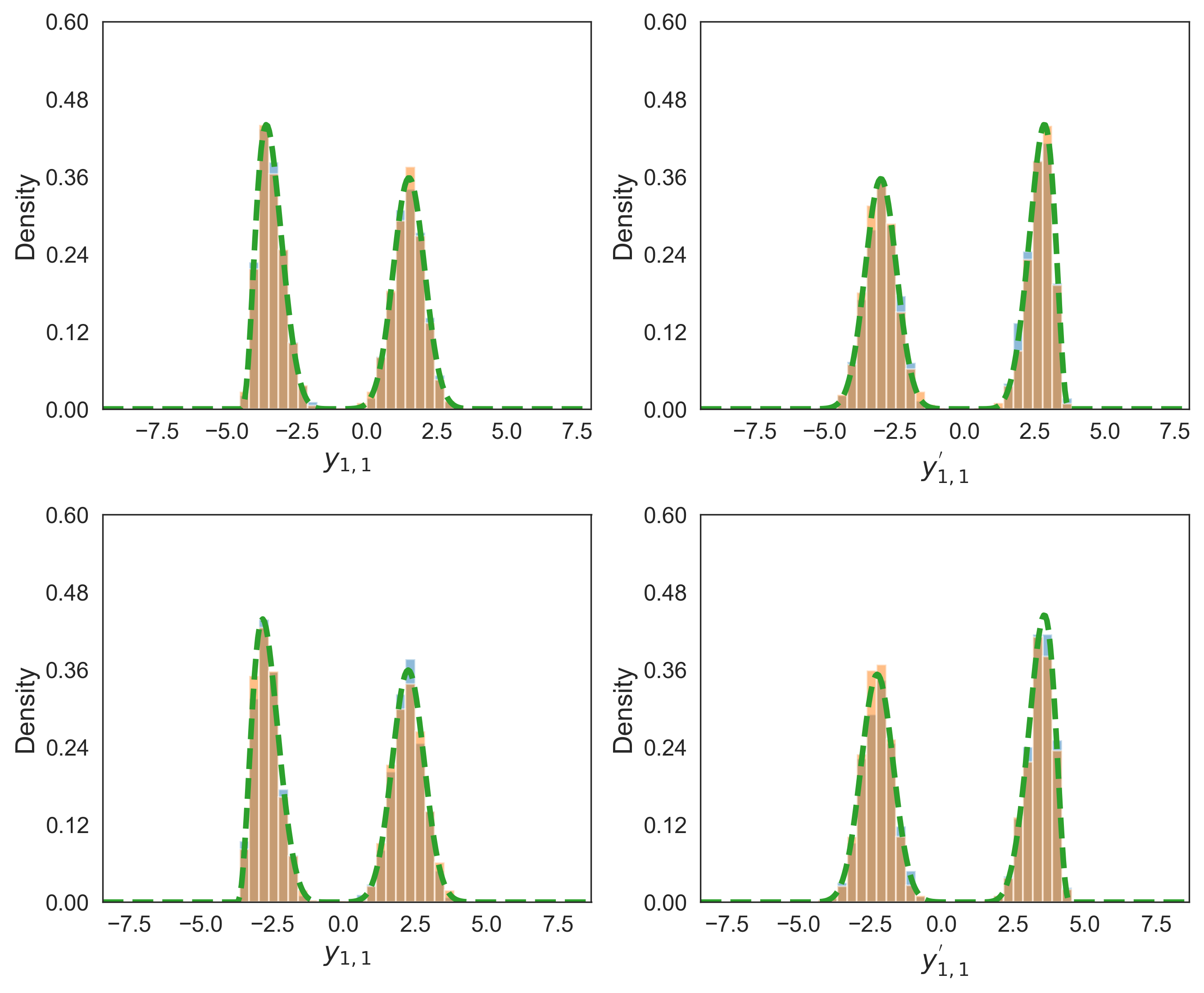}
    }
    
    \subfigure[$n=2$]{
        \includegraphics[width=0.48\textwidth]{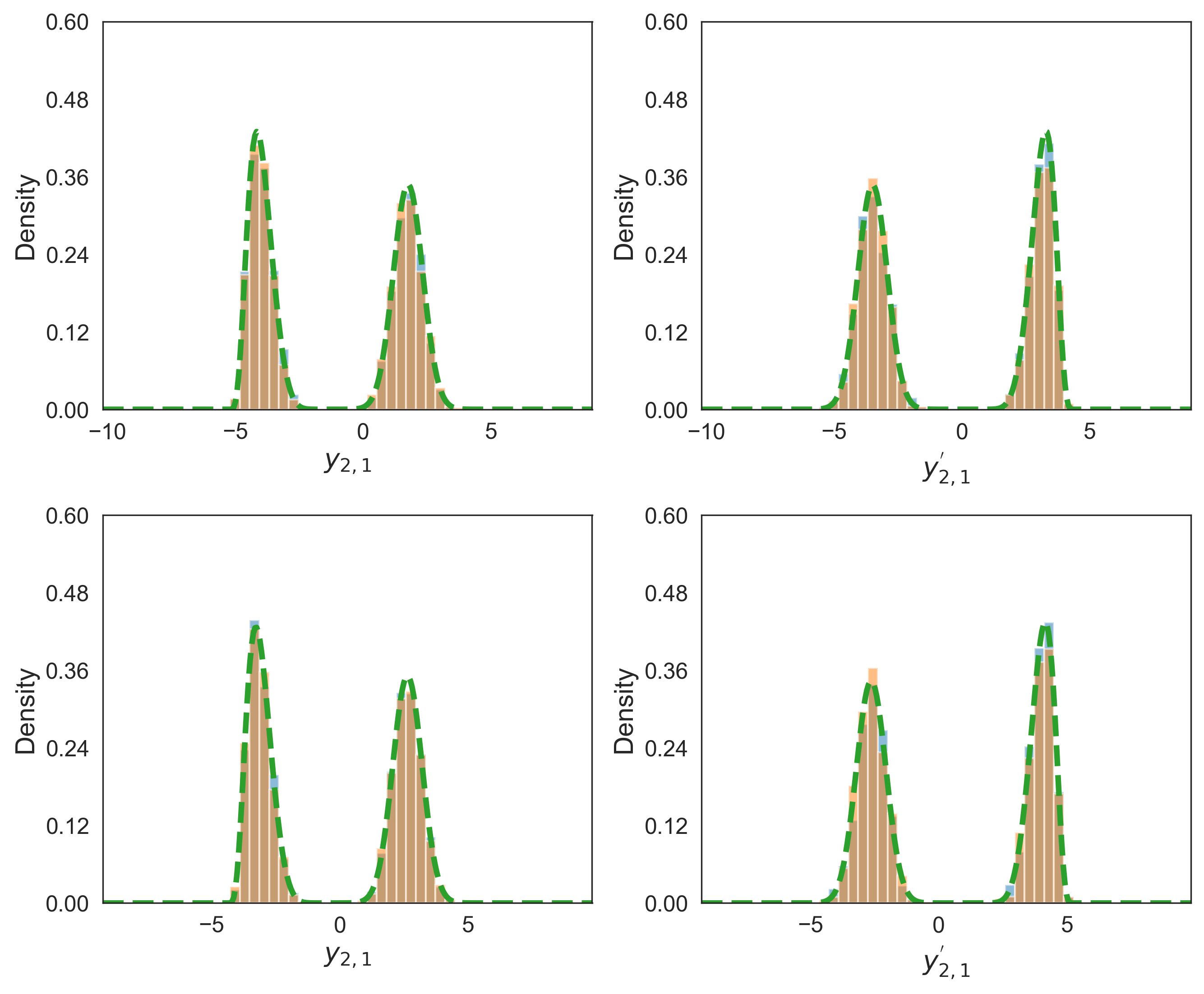}
    }
    \hfill
    \subfigure[$n=3$]{
        \includegraphics[width=0.48\textwidth]{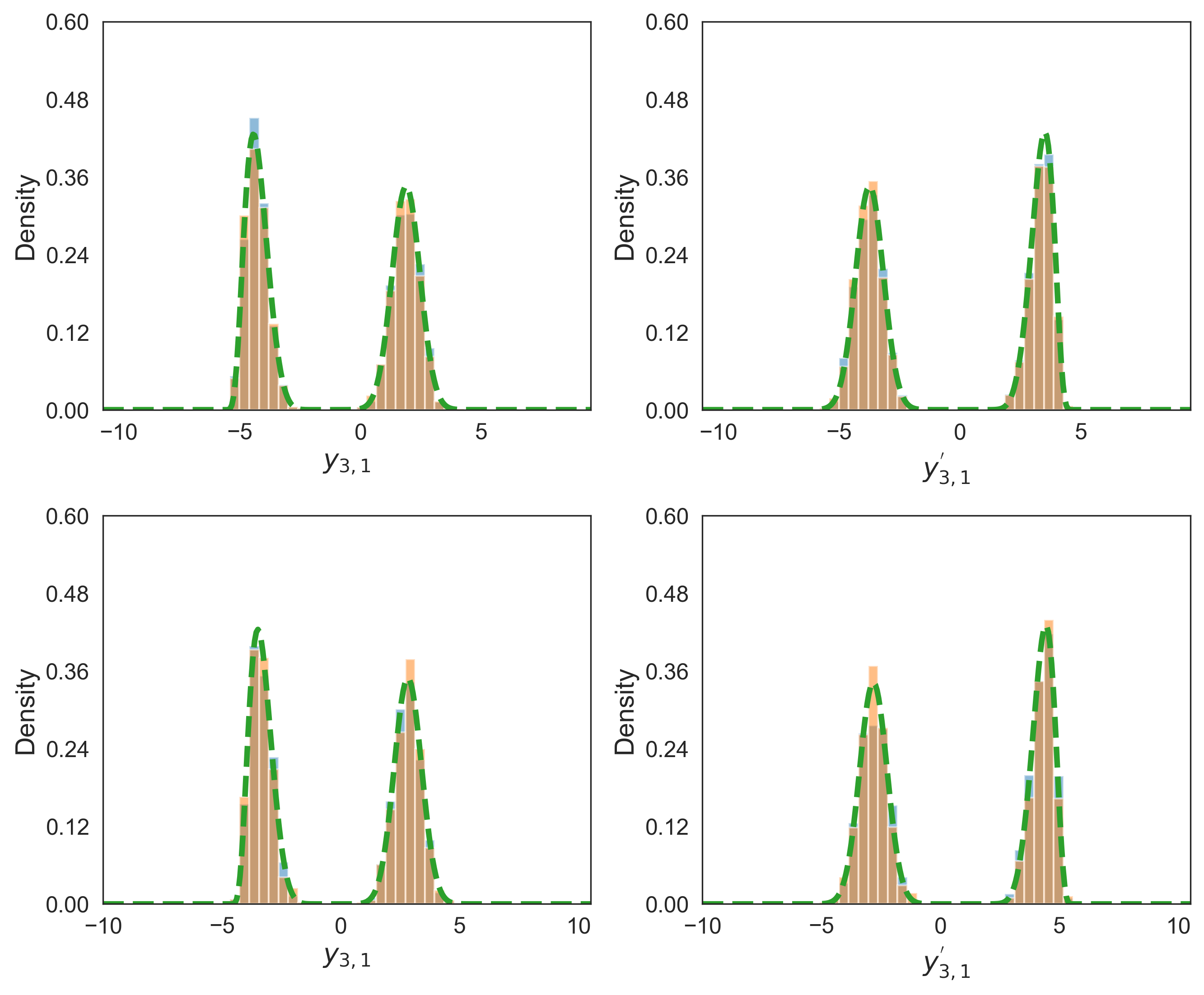}
    }
\end{figure}

We further assess whether the learned generator captures temporal dependence by comparing adjacent-step correlations in Figure~\ref{fig:simulation_d2_N3_corr}. The estimated correlations closely match the ground truth across the entire horizon, indicating that the method preserves the AR(1)-type serial structure beyond marginal matching. In particular, the absolute discrepancy between the generated and true adjacent-step correlations remains small, on the order of $10^{-3}$.

\begin{figure}[!ht]
    \centering
    \caption{Temporal correlations between adjacent time steps for both unimodal and bimodal settings ($d=2$, $N=3$).
    \label{fig:simulation_d2_N3_corr}}
    \subfigure[Unimodal.]{
        \includegraphics[width=0.48\linewidth]{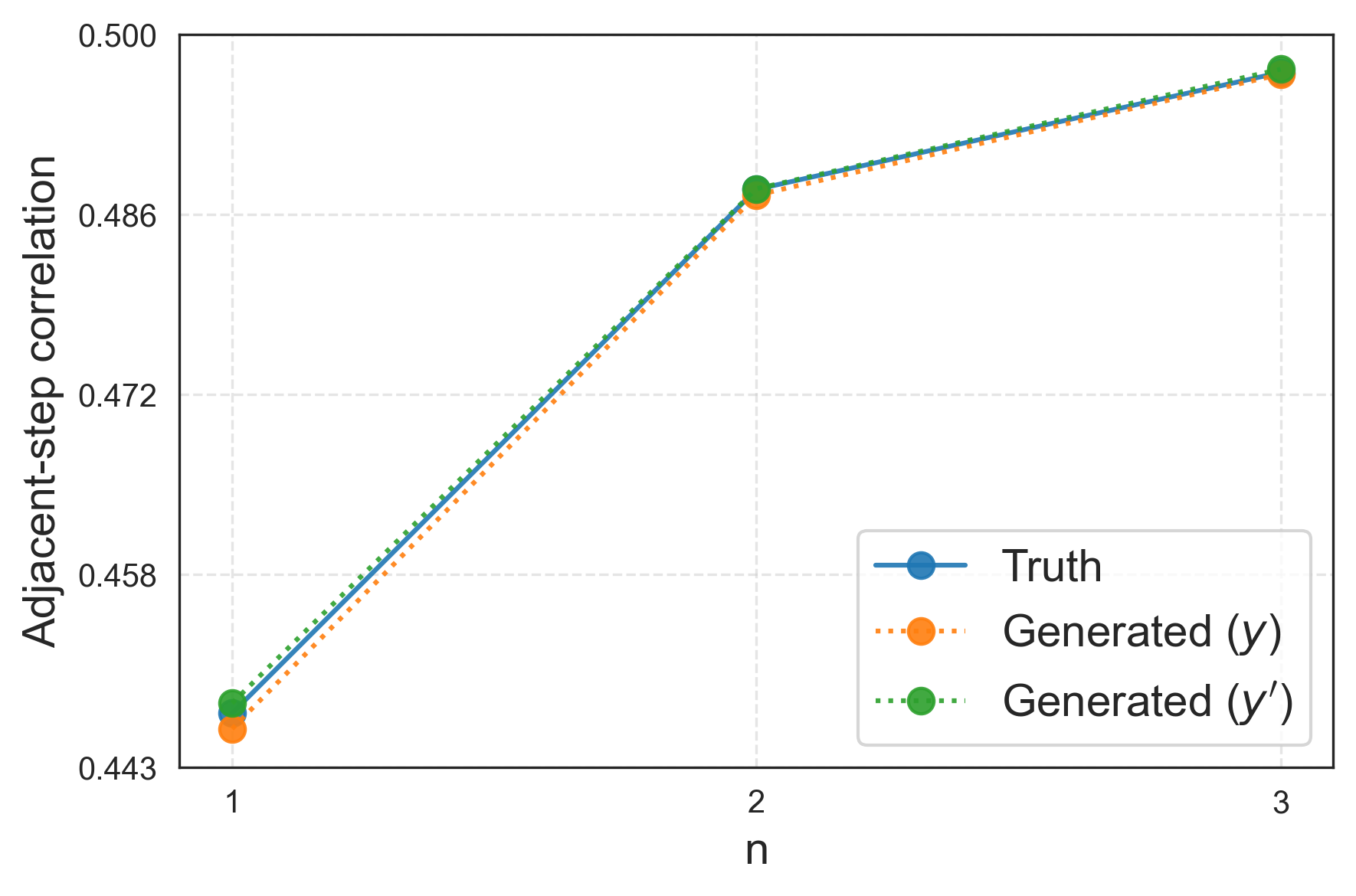}
    }
    \hfill
    \subfigure[Bimodal.]{
        \includegraphics[width=0.48\linewidth]{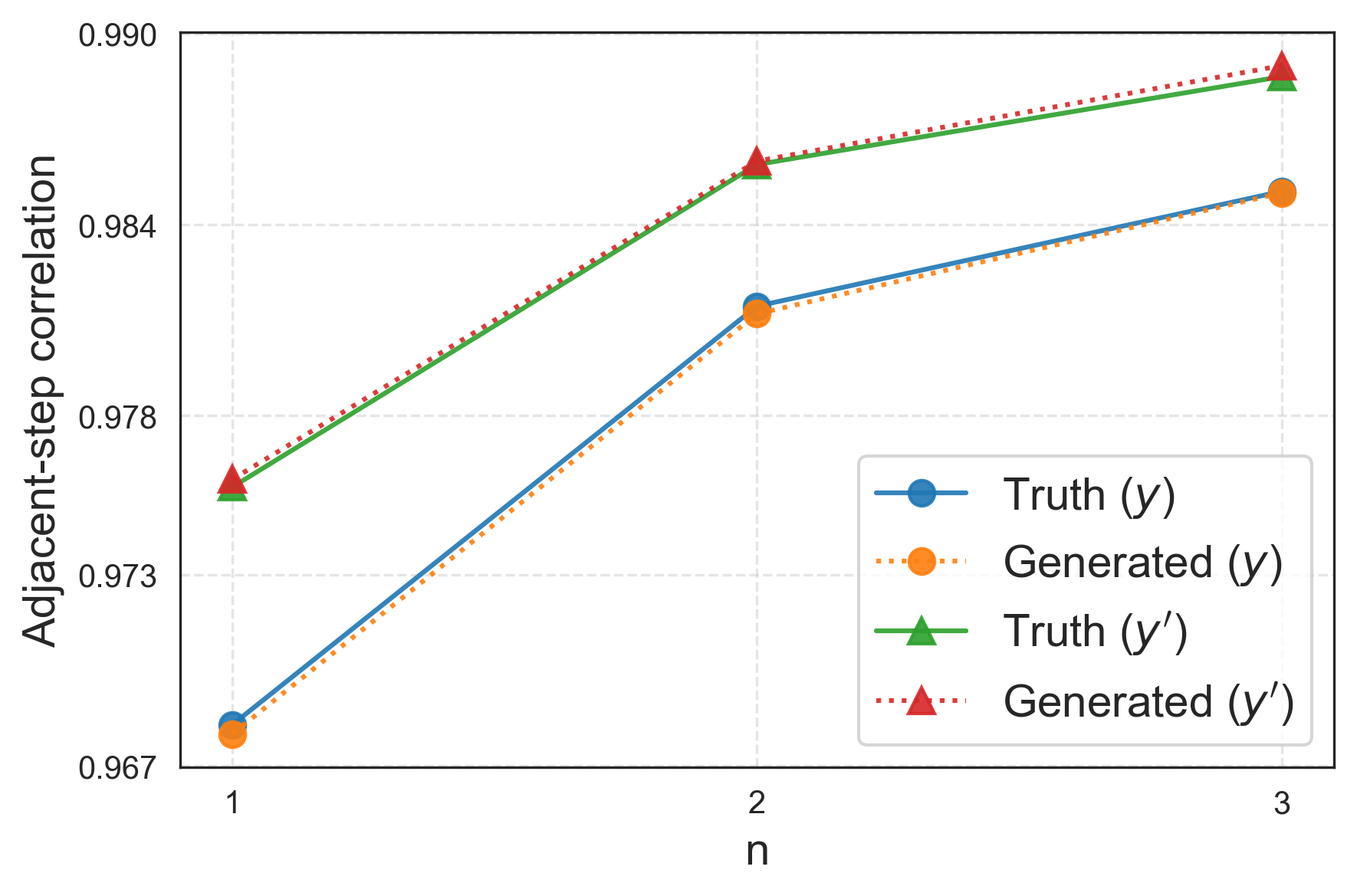}
    }
\end{figure}

The quantitative metrics confirm the same pattern. Table~\ref{tab:d2_N3_metrics} reports the distributional metrics and adjacent-correlation tests over 10 independent runs. The marginal discrepancy metrics, trajectory-level metrics, and the learned average cost all remain small in both unimodal and bimodal settings, indicating accurate recovery of both single-step marginals and the full joint trajectory distribution. Moreover, the adjacent-correlation discrepancies are small, and the permutation tests yield non-significant $p$-values in both cases (unimodal: $\mathtt{p(T_{\max})}=0.9790$, $\mathtt{p(T_2)}=0.9800$; bimodal: $\mathtt{p(T_{\max})}=0.6010$, $\mathtt{p(T_2)}=0.5789$). Therefore, we do not find statistically significant evidence that the generated and ground-truth adjacent-step correlations differ for $n=1,\dots,N$.

\begingroup
\setlength{\tabcolsep}{4pt}
\renewcommand{\arraystretch}{0.95}
\begin{table}[H]
\centering
\small
\caption{Performance on distributional metrics and adjacent-correlation tests ($d=2$, $N=3$). Mean values are reported with standard deviations in parentheses.\label{tab:d2_N3_metrics}}
\begin{tabular}{lccccccccc}
\toprule
& \multicolumn{5}{c}{\textbf{Distributional Metrics}} & \multicolumn{4}{c}{\textbf{Adjacent-Correlation Tests}} \\
\cmidrule(lr){2-6} \cmidrule(lr){7-10}
\textbf{Setting} & $\mathtt{Cost\_avg}$ & $\mathtt{W2\_avg}$ & $\mathtt{KS\_avg}$ & $\mathtt{SWD}$ & $\mathtt{MMD}^2$ & $\mathtt{T}_{\max}$ & $\mathtt{p(T_{\max})}$ & $\mathtt{T_{2}}$ & $\mathtt{p(T_{2})}$ \\
\midrule
\textbf{Unimodal} & 0.2858 & 0.0010 & 0.0129 & 0.0152 & 0.0124 & 0.0045 & 0.9790 & 0.0001 & 0.9800 \\
& (0.0476) & (0.0002) & (0.0002) & (0.0013) & (0.0005) & (0.0020) & (0.0205) & (0.0001) & (0.0195) \\
\addlinespace
\textbf{Bimodal} & 1.6224 & 0.0038 & 0.0129 & 0.6033 & 0.0241 & 0.0008 & 0.6010 & 0.0001 & 0.5789 \\
& (0.0823) & (0.0040) & (0.0003) & (0.0121) & (0.0034) & (0.0008) & (0.4459) & (0.0001) & (0.4263) \\
\bottomrule
\end{tabular}
\end{table}
\endgroup

\section{Additional details on the implemented normalizing-flow parameterization}
\label{sec: details of nf implementation}

In this section, we summarize the implemented normalizing-flow parameterization used in the synthetic experiments. For each $n \in [N]_0$, let $e_n \in \mathbb{R}^{d_{\mathrm{emb}}}$ denote the sinusoidal embedding of time step $n$. The recursion is initialized by $\widehat Y_{-1}=0$ and $\widehat Y_{-1}'=0$. The implementation of the normalizing-flow architecture is described below.

\begin{enumerate}
    \item \textbf{Normalization.}
    In implementation, the model is trained on coordinate-wise normalized trajectories (we normalize the data by subtracting the mean and dividing by the standard deviation coordinate-wise), and the corresponding affine Jacobian correction is added back when densities are evaluated on the original scale.
    \item \textbf{Time-dependent conditioning variables.}
    For each $n \in [N]_0$, given the state $(\widehat Y_{n-1},\widehat Y_{n-1}')$, we define
    $$
    C_n := (\widehat Y_{n-1},e_n),
    \qquad
    C'_n := (\widehat Y_{n-1}',e_n),
    \qquad
    C_n^{(\rho)} := (\widehat Y_{n-1},\widehat Y_{n-1}',e_n).
    $$
    The conditioning variables $C_n$ and $C'_n$ are used for the two generated conditional marginals $\widehat Y_n$ and $\widehat Y'_n$, while $C_n^{(\rho)}$ is used for the dependency network.

    \item \textbf{Latent Gaussian-copula coupling.}
    Inspired by \citep{laszkiewicz2021copula}, we introduce the dependence parameter produced by a neural network $\mathfrak{N}_n^{\theta_n}$:
    \begin{equation}
    \label{eq:def_varrho_n}
        \varrho_n = \mathfrak{N}_n^{\theta_n}(C_n^{(\rho)}) \in [-\rho_{\max},\rho_{\max}]^d,
    \end{equation}
    where $\rho_{\max}$ is a fixed upper bound of the correlation. Given two independent Gaussian vectors
    $\varepsilon_n^{(1)},\varepsilon_n^{(2)} \sim \mathcal N(0,I_d)$, we define the latent pair $(U_n,V_n)$ by
    \begin{equation}
    \label{eq:def U_n and V_n}
        U_n = \varepsilon_n^{(1)},
        \qquad
        V_n = \varrho_n \, \varepsilon_n^{(1)}
        + \sqrt{1-\varrho_n^2}\, \varepsilon_n^{(2)}.
    \end{equation}
    By construction, $U_n$ and $V_n$ have standard Gaussian marginals, and their dependence is determined by $\varrho_n$.

    \item \textbf{Conditional marginal flows.}
    Conditional on $C_n$ and $C'_n$, the marginal densities are estimated by conditional spline coupling flows \citep{durkan2019neural}. In our implementation, both flows use alternating coupling masks, and the associated conditioner networks are multilayer perceptrons (MLPs).

    \item \textbf{Joint one-step density.}
    Let $u_n,v_n \in \mathbb{R}^d$ denote the latent variables obtained by inverting the conditional flows from $\widehat Y_n$ and $\widehat Y_n'$ under conditioning $C_n$ and $C'_n$, respectively. Specifically, we define
    \begin{equation}
    \label{eq:def_invert_trans}
        \widehat Y_n = f_\varphi^{-1}(U_n; C_n),\quad \widehat Y_n' = g_\psi^{-1}(V_n; C'_n).
    \end{equation}
    The conditional log-density of the implemented one-step kernel is
    \begin{equation*}
    \log q_n^{\theta_n}(y_n,y_n' \mid y_{n-1},y_{n-1}')
    =
    \log q_n^{\theta_n,(1)}(y_n \mid c_n)
    +
    \log q_n^{\theta_n,(2)}(y_n' \mid c_n')
    +
    \log \ell_{\varrho_n}(u_n,v_n),
    \end{equation*}
    where $\ell_{\varrho_n}$ denotes the Gaussian-copula density with correlation vector $\varrho_n$.

    The two marginal log-densities are given by the change-of-variables formula:
    $$
    \log q_n^{\theta_n,(1)}(y_n \,|\, c_n)
    =
    \log \phi_d(u_n)
    +
    \log \bigl|\mathrm{Jcb}_n^{(1)}(y_n;c_n)\bigr|,
    $$
    and
    $$
    \log q_n^{\theta_n,(2)}(y_n' \mid c_n')
    =
    \log \phi_d(v_n)
    +
    \log \bigl|\mathrm{Jcb}_n^{(2)}(y_n';c'_n)\bigr|,
    $$
    where $\phi_d$ denotes the standard Gaussian density on $\mathbb{R}^d$, and $\mathrm{Jcb}_n^{(1)}$ and $\mathrm{Jcb}_n^{(2)}$ are the corresponding determinants of the inverse-flow Jacobians. The copula correction term is
    \begin{equation*}
    \log \ell_{\varrho_n}(u_n,v_n)
    =
    \sum_{j=1}^d
    \left[
    -\frac{1}{2}\log(1-\varrho_{n,j}^2)
    -\frac{u_{n,j}^2-2\varrho_{n,j}u_{n,j}v_{n,j}+v_{n,j}^2}{2(1-\varrho_{n,j}^2)}
    +\frac{1}{2}(u_{n,j}^2+v_{n,j}^2)
    \right].
    \end{equation*}

    \item \textbf{Sampling and likelihood evaluation.}
    Given $(\widehat Y_{n-1},\widehat Y_{n-1}')$, we first compute $\varrho_n$ using \eqref{eq:def_varrho_n}. Then, based on the latent pair $(U_n,V_n)$ sampled from \eqref{eq:def U_n and V_n}, we apply invertible mapping \eqref{eq:def_invert_trans} to generate $(\widehat Y_n,\widehat Y_n')$. Repeating this procedure for $n \in [N]_0$ yields one sample from the path law $\pi^{\pmb\theta}$.

    At each iteration $k$, given sampled trajectories $(\widehat{\pmb Y},\widehat{\pmb Y}')\sim\pi^{\pmb\theta_k}$, the per-trajectory statistics in \eqref{eq:Vhat-def}--\eqref{eq:Khat-def} are calculated. Then, the aggregated gradient estimator \eqref{eq:gkhat-def} can be constructed with the control variates \eqref{eq:agg-control-variate-def}. Finally, we apply the Adam optimizer to update the parameters using the estimated gradient.
\end{enumerate}

\section{Additional details on time series statistical downscaling}
\label{sec:details_downscaling}

\paragraph{Debiasing.}
We implement the debiasing step using Algorithm~\ref{alg:kl-pg}. In each round of training, we construct projected low-resolution samples $\{\pmb Y'^{(m)}\}_{m=1}^{M}$ from high-resolution trajectories $\{\pmb X^{(m)}\}_{m=1}^{M}$ via $C\pmb X^{(m)} = \pmb Y'^{(m)},  m \in [M]$, and pair them with observed low-resolution samples $\{\pmb Y^{(m)}\}_{m=1}^{M}$. These paired samples are then used to learn a bi-causal coupling that transports $\mu_{Y'}$ toward $\mu_Y$ while preserving sequential dependence.

For sampling, given an observed low-resolution path $\pmb Y$, we apply the learned transport map $T$ to obtain $\pmb Y' = T^{-1}(\pmb Y)$. This yields a low-resolution trajectory that is consistent with the observed low-resolution law and suitable for downstream high-resolution generation.

In Section~\ref{subsec:synthetic_exp_setup}, we explained the normalizing-flow parameterization used in Algorithm~\ref{alg:kl-pg}; see Appendix~\ref{sec: details of nf implementation} for implementation details. Concretely, given a projected low-resolution time series $\pmb Y = (Y_0',\dots,Y_{N_y}')$, we first map each coordinate to latent noise via $U_n = f_\varphi(Y_n'; C_n)$. We then construct a correlated latent variable
\[
V_n = \varrho_n U_n + \sqrt{1-\varrho_n^2}\,\varepsilon_n,
\]
where $\varepsilon_n \sim \mathcal N(0,I)$ and $\varrho_n$ is the $n$-th correlation coordinate produced by RhoNet. Finally, applying the flow $g_\psi^{-1}$ yields the LFLR sample $\pmb Y = g_\psi(^{-1}(V_n; C_n)$. In this way, the OT component converts a biased HFLR sample $\pmb Y'$ into a LFLR sample $\widehat{\pmb Y}$, and HFLR sample is statistically aligned with the observed low-resolution law and preserves temporal dependence. As in the previous experiments, we set the KL-penalty coefficient to $\beta = 50$.

\paragraph{Conditional modeling.}
Once a debiasing map is learned, we generate an HFHR trajectory conditional on it. To model the HFHR prior $\mu_X$, following \citet{wan2024debias}, we first train an unconditional diffusion model using an Efficient U-Net architecture with self-attention at the coarsest resolution.

Next, we formulate the conditional generation task through a conditional reverse-time SDE. Let $\mathcal E_y = \{C\bar X_T = y\}$ denote the event that the terminal path satisfies the prescribed low-resolution constraint. As derived below in ``additional details on conditional modeling'', the conditioned reverse-time process $\tilde X$ evolves according to
\begin{equation}
\label{eq:cond_sde_body}
    \dd\tilde X_t=\left[\bar b(t,\tilde X_t)+\bar\sigma\bar\sigma^\top(t)\nabla_x\log h(t,\tilde X_t;y)\right]\dd t+\bar\sigma(t)\dd\tilde W_t,
\end{equation}
where $h(t,x;y) = \mathbb P(C\bar X_T = y \mid \bar X_t = x).$

The term $\nabla_x \log h(t,x;y)$ acts as a correction that steers the reverse-time dynamics toward trajectories consistent with the prescribed low-resolution path. Since this correction is generally intractable to compute exactly, we follow \citet{finzi2023user}, building on \citet{chung2022diffusion,chung2022improving}, and approximate conditional sampling using a guidance-based procedure. In this way, the HFLR path produced by the bi-causal OT step is incorporated into the high-resolution diffusion model.

\paragraph{Additional details on conditional modeling.}
Firstly, we train an (unconditional) diffusion model to approximate $\mu_X$, which is often referred to as {\it the pre-trained model}. Consider a filtered probability space $(\Omega, \mathcal F,\mathbb{P},\mathbb{F}=\{\mathcal{F}_t\}_{t\geq0})$ satisfying the usual conditions. We consider the forward process, with state-independent volatility: 
    \begin{eqnarray}
        \label{eq:forward}
        \dd X_t = {b(t,X_t)}\dd t + {\sigma(t)}\dd W_t, \ \ X_0\sim\mu_0=\mu_X,
    \end{eqnarray}
    with drift $b:[0,T]\times \mathcal{X}\rightarrow\mathcal{X}$, volatility $\sigma:[0,T]\rightarrow\mathcal{X}$, standard Wiener process $W=\{W_t\}_{t\geq0}$ with $W_t\in \mathbb{R}^{(N_x+1)d}$, {supported by $\mathbb{F}$} and initial distribution $\mu_0$ being the underlying true distribution of the data. {Under mild conditions on $b$ and $\sigma$, \eqref{eq:forward} has a unique strong solution $X$. Let $\{\mu_t\}_{t\geq0}$ denote the distribution flow for the process $X$ such that $\mu_t=Law(X_t)$ for all $t\geq0$. Assume that $\mu_t$ admits a density function $m(t,\cdot)$ for all $t\geq0$. By the Fokker-Planck equation, the density function $m:[0,+\infty)\times\mathcal X\to\mathbb{R}$ should solve
    \begin{align*}
        \dot{m}(t,x)&=-\nabla_x\cdot[m(t,x)b(t,x)]+\frac{1}{2}\nabla_x\cdot\left\{\nabla_x\cdot\left[m(t,x)\sigma\sigma^\top(t)\right]\right\}\\
        &=-\nabla_x\cdot\left\{m(t,x)b(t,x)-\frac{\nabla_x\cdot\left[m(t,x)\sigma\sigma^\top(t)\right]}{2}\right\}
    \end{align*}
    }
    Fix any $T>0$. The {\it time-reverse SDE} was introduced in \cite{haussmann1986time}.{Let $\bar W=\{\bar W_t\}_{t\geq0}$ be a standard Brownian motion also supported by $\mathbb{F}$, independent from $W$.}
    Consider an SDE of the following form,
    \begin{eqnarray}
        \dd \bar{X}_t = \bar{b}(t,\bar{X}_t)\dd t + \bar{\sigma}(t) \dd \bar{W}_t, \text{ with } \bar{X}_0\sim \mu_T;\label{eq:backward}
    \end{eqnarray}
    {
    where for any $(t,x)\in [0,T]\times\mathcal X$, $\bar\sigma(t)=\sigma(T-t)$ and 
    \begin{align*}
    \bar b(t,x)&=-b(T-t,x)+\frac{\nabla_x\cdot[\sigma\sigma^\top(T-t)m(T-t,x)]}{m(T-t,x)}\\
    &=-b(T-t,x)+\sigma\sigma^\top(T-t)\nabla_x\log{m(T-t,x)}\\
    &=-b(T-t,x)+\sigma\sigma^\top(T-t)\underbrace{\nabla_x\log{m(T-t,x)}}_{\textit{Stein's score function}}.
    \end{align*}
    }

Then under mild conditions, \eqref{eq:backward} admits a strong solution $\bar X=\{\bar X_t\}_{t\in[0,T]}$. In addition, $\bar\mu_t=Law(\bar X_t)$ admits a density function $\bar m(t,\cdot)$ for each $t\in[0,T]$, by checking the corresponding Fokker-Planck equation, we have $\bar m(t,\cdot)=m(T-t,\cdot)$ for all $t\in[0,T]$, that is, $X_{T-t}\overset{d}{=}\bar X_t$. The drift function of the time-reverse SDE (\ref{eq:backward}) depends on the time-dependent score function, which is approximated by a neural network trained with denoising score matching.

Given the forward process \eqref{eq:forward} and its corresponding time-reverse process \eqref{eq:backward},
fix arbitrary $y\in\mathcal Y'$ and let $\tilde X=\{\tilde X_t\}_{t\in[0,T]}$ be a process where for each $t$, $\tilde X_t=\bar X_t \,|\, \mathcal{E}_{y}$ and $\mathcal{E}_y=\{C\bar X_T=y\}$. Recall that $\mathcal{A}_y=\{x\in\mathcal{X}:Cx=y\}$, then $\mathcal{E}_y=\{\bar X_T\in \mathcal{A}_y\}$; notice that $\mathbb{P}(\mathcal{E}_y)=\mathbb{P}(X_0\in \mathcal{A}_y)$. 
For any $(t,x)\in[0,T]\times\mathcal{X}$ and any  $t'\in[t,T]$, 
\begin{itemize}
    \item let $\mathbb{P}_{t'|t}(\cdot \,|\, x)$ denote the conditional distribution $\bar X_{t'} \,|\, \{\bar X_t=x\}$ with $\mathbb{P}_{t|t}(\cdot \,|\, x)=\delta_x(\cdot)$ and denote the corresponding density function $p_{t'|t}(\cdot \,|\, x)$;
    \item let $h(t,x;y)=\mathbb{P}(\mathcal{E}_y|\bar X_t=x)$;
    \item let $\tilde{\mathbb{P}}_{t'|t}(\cdot \,|\, x;y)$ denote the conditional distribution of $\bar X_{t'} \,|\, \{\bar X_t=x,\bar X_T\in \mathcal{A}_y\}\overset{d}{=}\tilde X_{t'} \,|\, \{\tilde X_t=x\}$, with density function $\tilde p_{t'|t}(\cdot \,|\, x;y)$. 
\end{itemize} 
By Bayes' rule,
\[\tilde p_{t'|t}(x' \,|\, x;y)=p_{t'|t}(x' \,|\, x)\frac{h(t',x';y)}{h(t,x;y)},\quad\forall x'\in\mathcal X.\]
For any smooth, bounded test function $\phi:\mathcal{X}\to\mathbb{R}$ with compact support, let $\mathcal L^t$ denote the infinitesimal generator for $\bar X$ at each $t$, with 
\begin{eqnarray}\label{eq:L_t}
    \mathcal L^t\phi(x)=\bar b(t,x)^\top\nabla_x\phi(x)+\frac{Tr\left(\bar\sigma\bar\sigma^\top(t)\nabla_x^2\phi(x)\right)}{2}
\end{eqnarray}
observe that 
\[\mathbb{E}\left[\phi(\bar X_{t'})-\phi(\bar X_t)\,\biggl|\,\bar X_t=x,E_y\right]=\frac{\mathbb{E}\left[\phi(\bar X_{t'})h(t',\bar X_{t'};y)-\phi(\bar X_{t})h(t,\bar X_{t};y)\,\biggl|\,\bar X_t=x\right]}{h(t,x;y)}.\]
By the tower property of conditional expectations, $(\partial_t+\mathcal{L}^t)h(t,\cdot;y)\equiv0$. Let $\tilde{\mathcal{L}}_y^t$ be the infinitesimal generator for $\tilde X$ at each $t$, then 
\begin{align*}
\tilde{\mathcal{L}}_y^t\phi(x)&=\lim_{t'\downarrow t}\frac{\mathbb{E}\left[\phi(\bar X_{t'})-\phi(\bar X_t)\,\biggl|\,\bar X_t=x,E_y\right]}{t'-t}=\frac{(\partial_t+\mathcal{L}^t)(h\phi)(t,x;y)}{h(t,x;y)}\\
&=\mathcal{L}^t\phi(x)+\langle\bar\sigma\bar\sigma^\top(t,x)\nabla_x\log{h(t,x;y)},\nabla_x\phi(x)\rangle.
\end{align*}
Thus, $\tilde X$ follows the following diffusion process
\begin{equation*}
\label{eq:cond_fidd_process}
    \dd\tilde X_t=\left[\bar b(t,\tilde X_t)+\bar\sigma\bar\sigma^\top(t)\nabla_x\log{h(t,\tilde X_t;y)}\right]\dd t+\bar\sigma(t)\dd\tilde W_t,
\end{equation*}
where $\tilde W$ is a standard Brownian motion supported by $\mathbb F$, independent from $W$ and $\bar W$.

Notice that, $h$ should solve 
\begin{eqnarray*}
\label{eq:pde_h2-simplify}
    -\partial_t{h}(t,x;y)=\mathcal{L}^th(t,x;y),\quad \forall (t,x)\in[0,T)\times\mathcal{X},\ \ y\in\mathcal Y'
\end{eqnarray*}
with terminal condition
\begin{eqnarray*}
\label{eq:boundary_relaxation}
    h(T,x;y) = {\bf 1}(x\in A_y)={\bf 1}(C\,x = y)
\end{eqnarray*}
where $\mathcal{L}^t$ is defined in \eqref{eq:L_t}.

\section{Evaluation metrics}
\label{sec:evaluation metrics}
This section provides the evaluation metrics used in the experiments. We first introduce the general metrics used in the synthetic experiments (Section~\ref{sec: synthetic experiments}), and then record the additional metrics specific to the robust hedging (Section~\ref{sec: robust hedging}) and time series statistical downscaling applications (Section~\ref{sec: statistical downscaling}).

\subsection{Evaluation metrics for synthetic experiments}
\label{sec:gen_set_metrics}

This subsection introduces the evaluation metrics used in the synthetic experiments. These metrics assess the learned coupling from three perspectives: the transport objective value attained by Algorithm~\ref{alg:kl-pg}, the accuracy of the generated distributions, and the preservation of temporal dependence.

\begin{itemize}
    
    \item \textbf{Transport cost.}
    To quantify the objective value attained by the relaxed formulation \eqref{eq:relaxed-bicausal-ot}, we report the empirical average cost induced by the learned coupling returned by Algorithm~\ref{alg:kl-pg}:
    \begin{equation*}
    \mathtt{Cost\_avg}
    :=\frac{1}{M(N+1)d}\sum_{m=1}^M \sum_{n=0}^{N}\sum_{i=1}^{d} \big(\widehat{Y}_{n,i}^{(m)}-\widehat{Y}'^{(m)}_{n,i}\big)^2.
    \end{equation*}

   \item \textbf{Wasserstein--2 distance for marginal distributions.}
    For each time step $n$ and asset $i$, define the empirical marginal measures
    \[
    \widehat P_{n,i}=\frac{1}{M}\sum_{m=1}^M \delta_{Y_{n,i}^{(m)}},\quad
    \widehat P'_{n,i}=\frac{1}{M}\sum_{m=1}^M \delta_{Y'^{(m)}_{n,i}},
    \quad
    \widehat Q_{n,i}=\frac{1}{M}\sum_{m=1}^M \delta_{\widehat{Y}_{n,i}^{(m)}},\quad \widehat Q'_{n,i}=\frac{1}{M}\sum_{m=1}^M \delta_{\widehat{Y}'^{(m)}_{n,i}}.
    \]
    The averaged squared Wasserstein--2 distance \citep{villani2009optimal} is given by
    \begin{equation*}
    \mathtt{W2\_avg}
    := \frac{1}{(N+1)d}\sum_{n=0}^N\sum_{i=1}^d
    \Big(W^2_2\big(\widehat P_{n,i},\widehat Q_{n,i}\big)+W^2_2\big(\widehat P'_{n,i},\widehat Q'_{n,i}\big)\Big),
    \end{equation*}
    where $W^2_2(P, Q) := \inf_{X \sim P, Y \sim Q} \mathbb{E}[(X-Y)^2]$ and this metric quantifies discrepancies between one-step marginals across time and assets.

    \item \textbf{Kolmogorov--Smirnov distance for marginal distributions.}
    Let $\widehat F_{n,i}$, $\widehat F'_{n,i}$, $\widehat G_{n,i}$, and $\widehat G'_{n,i}$ denote the empirical cumulative distribution
    functions corresponding to $\widehat P_{n,i}$, $\widehat P'_{n,i}$, $\widehat Q_{n,i}$, and $\widehat Q'_{n,i}$, respectively.
    The averaged Kolmogorov--Smirnov (KS) distance \citep{Smirnov1948,Massey1951} is defined as
    \begin{equation*}
    \mathtt{KS\_avg}
    := \frac{1}{(N+1)d}\sum_{n=0}^N\sum_{i=1}^d
    \Big(\sup_{y\in\mathbb{R}}
    \bigl|\widehat F_{n,i}(y)-\widehat G_{n,i}(y)\bigr|+\sup_{y'\in\mathbb{R}}
    \bigl|\widehat F'_{n,i}(y')-\widehat G'_{n,i}(y')\bigr|\Big),
    \end{equation*}
    which provides a complementary, distribution-free measure of marginal mismatch.

    \item \textbf{Sliced Wasserstein distance.}
    Let $Z^{(m)}=(Z^{(m)}_{1},\dots,Z^{(m)}_{N}):=\mathrm{vec}(\pmb Y^{(m)})\in\mathbb{R}^{p}$ and
    $\widehat{Z}^{(m)}=(\widehat Z^{(m)}_{1},\dots,\widehat Z^{(m)}_{N}):=\mathrm{vec}(\widehat{\pmb Y}^{(m)})$, where $p=2(N+1)d$. Given a projection direction $\alpha$, denote the marginal measures
    \[
    \widehat P_{\alpha}=\frac{1}{M}\sum_{m=1}^M \delta_{\alpha^\top Z^{(m)}} \quad \textrm{and} \quad
    \widehat Q_{\alpha}=\frac{1}{M}\sum_{m=1}^M \delta_{\alpha^\top \widehat{Z}^{(m)}}.
    \]
    The sliced Wasserstein distance (SWD) \citep{Bonneel2015} is defined as
    \begin{equation*}
    \mathtt{SWD}
    :=
    \E_{\alpha\sim\mathrm{Unif}(\mathbb{S}^{p-1})}
    \Bigl[
    W^2_2\bigl(\widehat P_{\alpha},\widehat Q_{\alpha}\bigr)
    \Bigr],
    \end{equation*}
    and is approximated numerically via Monte Carlo sampling over random projection directions.

    \item \textbf{Maximum mean discrepancy.}
    For the joint trajectory distributions $P$ and $Q$, the squared maximum mean discrepancy
    (MMD) \citep{Gretton2012} is given by
    \begin{equation*}
    \mathtt{MMD}^2
    :=
    \frac{1}{M^2}\sum_{m,l=1}^{M} \Big[ k(Z^{(m)}, Z^{(l)})
    + k(\widehat Z^{(m)},\widehat Z^{(l)})
    -2 k(Z^{(m)},\widehat Z^{(l)}) \Big],
    \end{equation*}
    where $k(\cdot,\cdot)$ is a positive definite kernel. This metric captures discrepancies
    in the full joint distribution of the stacked trajectories.

    \item \textbf{Adjacent-step correlation discrepancy.}
    To assess temporal dependence, for each time index $n \in \{1,\dots,N\}$, we compute the empirical adjacent-step correlation
    \begin{equation*}
    \widehat\rho_n^{(\mathrm{gen})}
    :=
    \frac{\sum_{m=1}^M
    \langle \widehat Z_n^{(m)}-\bar Z_n,\, \widehat Z_{n-1}^{(m)}-\bar Z_{n-1}\rangle}
    {\sqrt{\sum_{m=1}^M \|\widehat Z_n^{(m)}-\bar Z_n\|^2}\,
     \sqrt{\sum_{m=1}^M \|\widehat Z_{n-1}^{(m)}-\bar Z_{n-1}\|^2}},
    \qquad
    \bar Z_n:=\frac{1}{M}\sum_{m=1}^M \widehat Z_n^{(m)},
    \end{equation*}
    and $\widehat\rho_n^{(\mathrm{true})}$ is defined analogously by replacing the generated samples $\widehat Z^{(m)}$ with ground-truth samples $Z^{(m)}$. Let $\Delta_n:=\widehat\rho_n^{(\mathrm{gen})}-\widehat\rho_n^{(\mathrm{true})}$ and define
    \begin{equation*}
    \mathtt{T}_{\max}:=\max_{1\le n\le N}|\Delta_n|,
    \qquad
    \mathtt{T}_2:=\sum_{n=1}^N \Delta_n^2.
    \end{equation*}
    Statistical significance is assessed via a trajectory-level permutation test
    \citep{Good2005,WestfallYoung1993}, with $p$-values computed as
    \begin{equation*}
    \mathtt{p(T)}=\frac{1+\sum_{\ell=1}^L \mathbf{1}\{T^{(\ell)}\ge T^{\mathrm{obs}}\}}{B+1},
    \end{equation*}
    where $L$ is the number of permutations and $T^{(\ell)}$ is recomputed after permuting generated and ground-truth trajectories.
\end{itemize}

\subsection{Additional evaluation metrics for robust hedging}
\label{sec:eval_robust_hedging}

This subsection records the additional pricing metrics used in the robust hedging experiment. In addition to the distributional and temporal-dependence metrics introduced in Section~\ref{sec:gen_set_metrics}, we assess pricing accuracy through the estimated subhedging price under the learned coupling and its relative error with respect to the closed-form benchmark.

The empirical subhedging price is defined by
\begin{equation*}
\widehat{p(\xi)}
:=\frac{1}{M}\sum_{m=1}^M \sum_{n=1}^{N} (\widehat{Y}_n^{(m)}-\widehat{Y}'^{(m)}_n)^2,
\end{equation*}
where $\{(\widehat Y_0^{(m)},\widehat{Y}_0'^{(m)},\dots,\widehat Y_N^{(m)},\widehat{Y}_N'^{(m)})\}_{m=1}^M$ are i.i.d.\ samples from the learned coupling. Since the optimal value $\underline p(\xi)$ is available in closed form under our setup \citep[Lemma~2]{han2025kalman}, we also denote the relative error by
\begin{equation*}
\mathtt{RE}
:=\frac{\bigl|\widehat{p(\xi)}-\underline p(\xi)\bigr|}{|\underline p(\xi)|}.
\end{equation*}

\subsection{Additional evaluation metrics for time series statistical downscaling}
\label{sec:eval_met_sd}

This subsection records the additional evaluation metrics used in the time series statistical downscaling experiment. Beyond the general distributional and dependence metrics in Section~\ref{sec:gen_set_metrics}, we report the constraint relative root mean squared error (cRMSE), the Wasserstein--1 distance, and the kernel-density-estimated KL divergence (KLD).

Let $\{\widehat{\pmb X}^{(k,m)}\}_{m=1}^M$ denote generated samples, where $k$ indexes the conditioning value and $m$ indexes the generated sample under that conditioning. Let $\{ y^{(k)}\}_{k=1}^K$ denote the $K$ conditioning values. Each sample
\[
\widehat{\pmb X}^{(k,m)}
=
\{(\widehat X_{n,i}^{(k,m)})\}_{0\le n\le N_x,\;1\le i\le d}
\]
represents a length-$(N_x+1)$ trajectory of $d$ assets.

\begin{itemize}
    \item \textbf{Constraint relative root mean squared error.}
    This metric quantifies the extent to which the generated sample satisfies the conditioning constraint. It is defined by
    \begin{equation*}
        \mathtt{cRMSE}
        =
        \frac{1}{MK}\sum_{m=1}^M \sum_{k=1}^{K}
        \frac{\|C\widehat{\pmb X}^{(k,m)}- y^{(k)}\|_2}{\|C\widehat{\pmb X}^{(k,m)}\|_2}.
    \end{equation*}

    \item \textbf{Wasserstein-1 distance.}
    Let $F_n$ and $\widehat F_n$ denote the empirical cumulative distribution functions of the reference and generated marginal distributions at time $n$, respectively, pooled over samples, conditions, and asset dimensions. Then
    \begin{equation*}
    \mathtt{Wass}_1
    =
    \frac{1}{N_x+1}\sum_{n=0}^{N_x}
    \int_{-\infty}^{\infty} |\widehat F_n(t)-F_n(t)|\,dt.
    \end{equation*}
    This measures the discrepancy between pooled univariate marginal distributions across time, which is sensitive to differences in distributional location and shape.

    \item \textbf{Kernel-density-estimated Kullback--Leibler divergence.}
    Let $p_n$ and $\widehat p_n$ denote the kernel density estimates of the reference and generated marginal densities at time $n$, respectively. Then, define
    \begin{equation*}
    \mathtt{KLD}
    =
    \frac{1}{N_x+1}\sum_{n=0}^{N_x}
    \int_{-\infty}^{\infty}
    p_n(t)\log\frac{p_n(t)}{\widehat p_n(t)}\,dt,
    \end{equation*}
    which quantifies the discrepancy between pooled univariate marginal densities across time and particularly captures local density mismatch and tail undercoverage.
\end{itemize}

\end{document}